%% file: thesis.tex
\documentclass[12pt]{report}
\usepackage[latin1]{inputenc}
\usepackage[T1]{fontenc}
\usepackage{latexsym,amsmath,amssymb,amsthm,amsfonts,stmaryrd,pb-diagram,amscd,mathrsfs}
\usepackage{cite}
\usepackage{graphicx}
\usepackage[osf]{mathpazo}
\DeclareMathAlphabet{\mathsf}{\encodingdefault}{phv}{m}{n}
\usepackage[12pt]{moresize}
\usepackage{titlesec}
\titleformat{\chapter}[display]{\sf}{\HUGE\filleft\thechapter}{12pt}{\Huge\filleft}
\usepackage[top=1in,bottom=1in,left=1in,right=1in]{geometry}
\usepackage{cite,hyperref}
\usepackage{color}

\def\ev{\mathrm{ev}}
\def\OO{\mathcal{O}}
\def\PP{\mathbb{P}}
\def\AA{\mathbb{A}}
\def\GG{\mathbb{G}}
\def\pr{\mathrm{pr}}

\def\sR{\mathcal{R}}
\def\sM{\mathcal{M}}
\def\sN{\mathcal{N}}
\def\sL{\mathcal{L}}
\def\sK{\mathcal{K}}
\def\sC{\mathcal{C}}
\def\sD{\mathcal{D}}
\def\sI{\mathcal{I}}
\def\sU{\mathcal{U}}
\def\sA{\mathcal{A}}
\def\sE{\mathcal{E}}
\def\sF{\mathcal{F}}

\def\sQ{\mathcal{Q}}
\def\sT{\mathcal{T}}
\def\sfA{\mathsf{A}}
\def\sfB{\mathsf{B}}
\def\sfC{\mathsf{C}}
\def\sfD{\mathsf{D}}
\def\sfE{\mathsf{E}}

\def\Gr{\mathsf{Gr}}
\def\LG{\mathsf{LG}}
\def\OG{\mathsf{OG}}
\def\fg{\mathfrak{g}}
\def\fp{\mathfrak{p}}
\def\fm{\mathfrak{m}}

\def\bir{\mathrm{bir}}
\def\free{\mathrm{free}}
\def\min{\mathrm{min}}
\def\arc{\mathrm{arc}}
\def\ss{\mathrm{ss}}
\def\sing{\mathrm{sing}}
\DeclareMathOperator\sHom{\mathcal{H}om}
\DeclareMathOperator\Sym{\mathrm{Sym}}
\DeclareMathOperator\Hom{\mathrm{Hom}}
\DeclareMathOperator\Isom{\mathrm{Isom}}
\DeclareMathOperator\End{\mathrm{End}}
\DeclareMathOperator\Aut{\mathrm{Aut}}
\DeclareMathOperator\Imm{\mathrm{Imm}}
\DeclareMathOperator\Spec{\mathrm{Spec}}
\DeclareMathOperator\Proj{\mathrm{Proj}}
\DeclareMathOperator\Spf{\mathrm{Spf}}
\DeclareMathOperator\Arc{\mathrm{Arc}}
\DeclareMathOperator\ArcHilb{\mathrm{ArcHilb}}

\DeclareMathOperator\Hilb{\mathrm{Hilb}}

\DeclareMathOperator\Pic{\mathrm{Pic}}
\DeclareMathOperator\id{\mathrm{id}}
\DeclareMathOperator\rk{\mathrm{rk}}
\DeclareMathOperator\ind{\mathrm{index}}
\DeclareMathOperator\codim{\mathrm{codim}}

\DeclareMathOperator\SL{\mathsf{SL}}
\DeclareMathOperator\GL{\mathsf{GL}}
\DeclareMathOperator\PGL{\mathsf{PGL}}
\DeclareMathOperator\Char{\mathrm{char}}
\DeclareMathOperator\Gal{\mathrm{Gal}}

\DeclareMathOperator\length{\mathrm{length}}

\theoremstyle{plain}
\newtheorem{thm}{Theorem}[section]
\newtheorem{lem}[thm]{Lemma}
\newtheorem{pro}[thm]{Proposition}
\newtheorem{cor}[thm]{Corollary}
\newtheorem*{thm*}{Theorem}
\newtheorem*{cor*}{Corollary}
\theoremstyle{definition}
\newtheorem*{defn}{Definition}
\title{Hwang--Mok rigidity of \\ cominuscule homogeneous varieties
\\ in positive characteristic \\ (second draft)}
\author{Jan Gutt \\ 
	 Stony Brook University}
\begin{document}
\setcounter{tocdepth}{1}
\sloppy

\begin{titlepage}
\begin{center}
\setlength{\baselineskip}{30pt}
\ \\ 
[0.5in]{\LARGE\bf Hwang--Mok rigidity of cominuscule homogeneous varieties 
in positive characteristic} \\
\setlength{\baselineskip}{24pt}
\vfill
A Dissertation Presented \\ 
by \\
[0.1in] {\bf Jan Aleksander Gutt} \\
[0.1in] to \\ 
The Graduate School \\ 
in Partial Fullfillment of the Requirements \\
for the Degree of \\
[0.1in] {\bf Doctor of Philosophy} \\
[0.1in] in \\
[0.1in] {\bf Mathematics} \\
\vfill Stony Brook University \\
May 2013 \\
\end{center}
\end{titlepage}


\pagenumbering{roman}
\setcounter{page}{2}
\begin{center}
\ \\
[0.1in] {\bf Stony Brook University} \\
[0.1in] The Graduate School \\
[0.3in] {\bf Jan Aleksander Gutt} \\
[0.3in] We, the dissertation committee for the above candidate for the
Doctor of Philosophy degree, hereby recommend acceptance of this dissertation.
\\
[0.8in] \line(1,0){0} \\
{\bf Jason Starr -- Dissertation Advisor} \\
{\bf Associate Professor, Department of Mathematics} \\ 
[0.5in] \line(1,0){0} \\
{\bf Samuel Grushevsky -- Chairperson of Defence} \\
{\bf Associate Professor, Department of Mathematics} \\ 
[0.5in] \line(1,0){0} \\
{\bf Radu Laza} \\
{\bf Assistant Professor, Department of Mathematics} \\ 
[0.5in] \line(1,0){0} \\
{\bf Aise Johan de Jong} \\
{\bf Professor of Mathematics, Columbia University} \\ 
[0.8in]
This dissertation is accepted by the Graduate School. \\
[0.5in] \line(1,0){0} \\
Charles Taber \\
Interim Dean of the Graduate School
\end{center}

\newpage
\begin{center}
\setlength{\baselineskip}{24pt}
\  \\
[0.1in] Abstract of the Dissertation \\
[0.1in] {\Large\bf Hwang--Mok rigidity of cominuscule homogeneous varieties 
in positive characteristic} \\
[0.1in] by \\
[0.1in] {\bf Jan Aleksander Gutt} \\
[0.1in] {\bf Doctor of Philosophy} \\
[0.1in] in \\
[0.1in] {\bf Mathematics} \\
[0.1in] Stony Brook University \\
2013 \\
\end{center}
\vfill
Jun-Muk Hwang and Ngaiming Mok have proved the rigidity of
irreducible Hermitian symmetric spaces of compact type under Kaehler
degeneration. I adapt their argument to the algebraic setting in
positive characteristic, where cominuscule homogeneous varieties serve
as an analogue of Hermitian symmetric spaces. The main result gives an
explicit (computable in terms of Schubert calculus) lower bound on the
characteristic of the base field, guaranteeing that a smooth
projective family with cominuscule homogeneous generic fibre is
isotrivial. The bound depends only on the type of the generic fibre,
and on the degree of an invertible sheaf whose extension to the
special fibre is very ample. An important part of the proof is a
characteristic-free analogue of Hwang and Mok's extension theorem for
maps of Fano varieties of Picard number 1, a result I believe to be
interesting in its own right.
\vfill
\newpage
\begin{flushright}
\ \\
\vfill
\textit{To green beans, rosemary, garlic and cherry tomatoes.}
\vspace{2cm}
\end{flushright}

\newpage
\tableofcontents

\chapter*{Acknowledgements}

I would like to express my deepest gratitude to my advisor,
Professor Jason Starr. I consider myself exceptionally lucky
to have been introduced to Algebraic Geometry through his
impeccable style. It is from him that I had learned about the
problem that has become my thesis project. The present Dissertation
would not be possible without his wise guidance, granting me
the freedom to explore, yet saving me from going astray.

I thank the members of my defence committee, Professors Samuel Grushevsky,
Radu Laza and Johan de Jong. Interacting with them at
various stages of this project gave me the very important opportunity
to see my work from other points of view, leading to questions
that wouldn't have occured to me otherwise.

Finally, it is  my pleasure to acknowledge the nurturing environment of the Mathematics
Department at Stony Brook, a place I am proud to have belonged to.

\newpage
\setcounter{page}{1}
\pagenumbering{arabic}
\input{chap1} 

\input{chap2} 

\input{chap3} 

\input{chap4} 


\bibliography{thesis}{}
\bibliographystyle{abbrv}

\addcontentsline{toc}{chapter}{Bibliography}

\end{document}

%% file: chap1.tex
\chapter{Introduction}\label{chap:intro}

\section{Overview of the work of Hwang and Mok}

\subsection{The study of VMRT}
The results of~\cite{hwang-mok-rigidity} and~\cite{hwang-mok-extension}
forming a basis for the generalisation attempted in this dissertation,
it is the philosophy of the wider research programme due to Hwang and Mok that informs
its overall architecture. We shall thus begin with a review of some of the main concepts,
stating the prototypical theorems and reconstructing from~\cite{hwang-mok-rigidity,
hwang-mok-extension} the sketch of a proof that could serve as a point of departure
for our own argument. The lecture notes~\cite{hwang-trieste} provide an accessible
introduction to this circle of ideas.

By the celebrated theorem of Mori, nonsingular Fano varieties are uniruled, and
in fact chain-connected by rational curves. In the particularly simple case
of a nonsingular Fano variety $X$ of Picard number $1$, 
one can in fact reduce to a family of irreducible rational curves
of minimal degree. Since such curves cannot degenerate to reducible ones, the
family is \emph{unsplit}, that is, it becomes proper after taking a quotient
by the group of automorphisms of $\PP^1$ (or its subgroup fixing the origin $0\in \PP^1$). 
Associating to a rational curve immersed at a general point $x \in X$ its tangent direction
in the projectivised tangent space $\PP T_{X,x}$, one obtains a rational map
from the space of minimal degree rational curves through $x$ into $\PP T_{X,x}$,
called the \emph{tangent map} (a theorem of Kebekus~\cite{kebekus} shows that
in the setting we are to consider, the tangent map is in fact an everywhere-defined,
finite morphism). 
Its closed image, called the variety of minimal rational tangents (VMRT) at $x$,
is the principal object of study in Hwang and Mok's approach. 

Since $X$ is chain-connected by minimal rational curves,
the VMRT at general points of $X$ connect global information about the geometry of $X$
with the local data of a closed subvariety in a projectivised tangent space, and
its infinitesimal variation. More accurately, this is true in characteristic zero, where
one can integrate first order infinitesimal data.\footnote{In positive characteristic,
we will need to replace the VMRT with an object encoding infinite order information.}
For example, one can expect certain classification results for complex
Fano manifolds of Picard number $1$,
based on the behaviour
of VMRT. For the simplest such $n$-fold, $\PP^n_\mathbb{C}$, the VMRT are just
entire projectivised tangent spaces---that is, there is a rational curve of minimal
degree through every point, and in every direction---and indeed $\PP^n_\mathbb{C}$ is completely  
characterised by this property, implying Fano index $n+1$.
A similar result exists for quadrics, whose VMRT are quadrics themselves,
implying Fano index $n$ (cf. ~\cite{kobayashi-ochiai}). 
A more general class of complex Fano manifolds of Picard number $1$
 is provided  
by Hermitian symmetric spaces of compact type. 
One application of the study of the VMRT is then the following rigidity theorem~\cite{hwang-mok-rigidity}.
\begin{thm*}[Hwang--Mok]
Let $X \to \Delta$ be a proper family of smooth complex manifolds over the unit disc,
such that the fibres $X_t$, $t\neq0$ are biholomorphic to a fixed irreducible Hermitian symmetric
space $G/P$ of compact type. Assume $X_0$ is Kaehler. Then $X_0$ is biholomorphic
to $G/P$.
\end{thm*}
Its original proof used Ochiai's application of Cartan's equivalence method~\cite{ochiai}
to a flat $L$-structure defined by the VMRT on the central fibre $X_0$ (where $L$
is isogeneous to the Levi factor of $P$). An underlying prolongation procedure
becomes cumbersome in characteristic $p>0$, requiring an immediate introduction
of conditions on $p$.
Fortunately,
this approach has since been completely replaced by a more recent `Cartan-Fubini type' extension
theorem~\cite{hwang-mok-extension}.
\begin{thm*}[Hwang--Mok]
Let $(X,\sM)$ and $(Y,\sN)$ be complex Fano manifolds of Picard number $1$,
together with a choice of an irreducible component of rational curves of
minimal degree. Let $\sC\subset \PP T_X$ and 
$\sD\subset \PP T_Y$ be the corresponding families of VMRT, and
assume that the fibre $\sC_x \subset \PP T_{X,x}$ at a general point 
$x\in X$ is positive-dimensional,
with generically finite Gauss map (as an embedded projective variety). Let
$U \subset X$ and $V \subset Y$ be connected analytic open subsets
together with a biholomorphic map $\varphi : U \to V$ such that
$\varphi_* : \PP T_U \to \PP T_V$ maps
$\sC|_U$ isomorphically onto $\sD|_V$. Then $\varphi$ extends to
a biholomorphism $\phi:X \to Y$.
\end{thm*}
We shall first sketch an argument reducing the rigidity theorem to the extension theorem,
and then describe the main steps in the proof of the latter. Along the way, we will indicate
some of the main difficulties that arise upon passage to positive characteristic.
The first issue is of course that the notion of a Hermitian symmetric space
is a complex-analytic one. The proper algebraic replacement is a \emph{cominuscule
homogeneous variety}, defined in Chapter \ref{chap:tools}. Complex cominuscule homogeneous
varieties are precisely irreducible Hermitian symmetric spaces of compact type, and
we will from now on use the former notion.

\subsection{Rigidity theorem}
In the setting $X \to \Delta$ of the rigidity theorem, after eliminating
the case where $G/P$ is a projective space, Hwang and Mok
study specialisations of rational curves of minimal degree on the
cominuscule homogeneous general fibre $G/P$ to the central fibre $X_0$.
These are in fact of degree $1$ with respect to the ample generator $\OO_{G/P}(1)$
of the Picard group of $G/P$.
Observing that $X_0$ is also Fano of Picard number one, we have
an ample invertible sheaf $\OO_X(1)$ on $X$ extending
$\OO_{G/P}(1)$, and the specialisations of degree one curves
on $G/P$ to $X_0$ are again irreducible rational curves of minimal degree. 
Now, the space of rational curves of minimal degree through a general point
of $X_0$ is smooth (i.e. all such curves are free; this can fail completely in positive
characteristic).
In particular, such curves deform to the general fibre: thus the family of degree
one rational curves through a general point of $X_0$ coincides with the family of
specialisations of degree one curves from $G/P$.

For a general section $s:\Delta \to X$
we have a family $M \to \Delta$ such that $M_t$ is the space of degree one
rational curves $\PP^1_\mathbb{C} \to X$ mapping $0 \in \PP^1_\mathbb{C}$ to $\sigma(t)$,
modulo the action of the group of automorphisms of $\PP^1_\mathbb{C}$ fixing $0$.
The map $M \to \Delta$ is projective and has smooth fibres.
A fundamental result about cominuscule homogeneous varieties states that $M_t$, $t\neq 0$
is either a Segre variety, or a cominuscule homogeneous variety itself.
Dealing separately with the Segre case, one sees that an inductive application
of the rigidity theorem allows us to conclude that $M \to \Delta$ is isotrivial (note that
$M_t$ have strictly lower dimension than $X_t$). This shows that
the space of degree one curves through a general point of $X_0$ is \emph{abstractly}
biholomorphic to that on the model $G/P$.

In order to conclude the same about the VMRT (as a subvariety
of a projectivised tangent space), Hwang and Mok show that
the latter is linearly nondegenerate. The argument uses integrability
of the meromorphic distribution defined by the linear span of the VMRT,
checking that linear degeneracy leads to the existence of an algebraic
foliation of $X_0$, whose properties would force the Picard number to
be greater than $1$ (in characteristic $p>0$ this becomes a
statement about a purely inseparable quotient, and we will need some
conditions on $p$ to derive a contradiction). 

Having shown that the VMRT at a general point of $X_0$, with its
embedding into the projectivised tangent space, is isomorphic to
the VMRT at a point of $G/P$ via a linear identification of tangent spaces,
there remains one more step needed to satisfy the hypotheses of
the extension theorem for $X_0$ and $G/P$ with their rational curves of degree one
(on $X_0$ we choose the unique dominating component; the VMRT at a general
point does satisfy finitness of the Gauss map).
It has to be checked that the family
of VMRT over a small analytic open subset of $X_0$ can be
identified with the family of VMRT over a biholomorphic
analytic open subset of $G/P$. A natural differential-geometric
approach is to associate with it an $L$-structure.\footnote{A reduction
of the frame bundle to a sub-bundle whose structure group
is the image of the Levi factor $L$ of $P$ in 
$\GL(\frak g/\frak p)$.}
By
prolongation theory the latter admits a well-defined notion of
curvature, whose vanishing on $X_t$, $t\neq0$
extends by continuity to $X_0$, implying local equivalence 
of $L$-structures, and thus of families
of VMRT on $X_0$ and $G/P$~\cite{ochiai}. 
It has 
already been pointed out that such differential-geometric
machinery is not convenient in positive characteristic.
However, as we will explain below, the object we shall use
in our version of the extension theorem will be a family
of arcs of infinite order, rather than just the VMRT.
In that setting, flatness will follow from a general
result on the `moduli' of families of formal arcs on
a formal disc (Chapter \ref{chap:tools}). 

\subsection{Extension theorem}
Let us now briefly outline the proof of the extension theorem. Recall
that we are in the setting of a pair of Fano manifolds $X$, $Y$ with
irreducible components of minimal degree rational curves $\sM$, $\sN$
and corresponding families of VMRT $\sC$, $\sD$. The theorem
states that an analytic-local biholomorphism $\varphi:U \to V$, compatible
with the VMRT, extends to a global biholomorphism $\phi:X\to Y$. The argument
consists of several steps.
\begin{enumerate}
\item One shows that $\varphi$ in fact sends holomorphic germs of
$\sM$-curves to holomorphic germs of $\sN$-curves. This relies on
differential-geometric methods, and in fact will never hold
in positive characteristic: for example, a `constant' family
of subvarieties of the projectivised tangent bundle of
a formal disc is not affected by Artin-Schreier type automorphisms, although
these will typically not preserve its lift to a family of formal arcs.

It is thus here that we set our point of departure.
Our version of the extension theorem will be a statement about
an isomorphism of formal neighbourhoods of general points compatible with
families of formal arcs, rather
than just the VMRT.

\item A procedure of `analytic continuation along minimal rational curves' is applied.
A rational curve is called \textit{minimal}\footnote{`Standard'
in~\cite{hwang-mok-extension}.} if its normal bundle, pulled back
to the normalisation,
splits into line bundles of degrees $0$ and $1$. General members of $\sM$ and $\sN$
have this property.
The assumption that
the VMRT at a general point be positive-dimensional implies that there is
at least one summand of degree $1$, so that a general $\sM$-curve
admits a deformation fixing precisely one point. Consider now a general
curve $C$ passing through a general point $x_0 \in U$. 
Its germ at $x_0$ is sent by $\varphi$ to a germ of a curve $D$
at $\varphi(x_0)\in V$. 

Choose a point $x\in C$. A deformation
of $C$ fixing $x$ induces a deformation of the germ of $C$ at $x$,
that is, a family of germs of $\sM$-curves in $U$. By Step 1, these
are sent by $\varphi_*$ to germs of $\sN$-curves in $V$, thus
giving rise to a family of $\sN$-curves. By the assumption on generality,
these $\sN$-curves intersect at a single point $y \in D$.  
This yields a map $C \to D$ extending
$\varphi|_C : C\cap U \to D\cap V$. Furthermore, $C \to D$
can be extended to an open neighbourhood of $C$ swept out by its small deformations.

There is a Zariski-dense open subset of $X$ that can be covered by chains of
general $\sM$-curves (of some fixed length)
with the first segment passing through $x_0$. Applying
the continuation procedure inductively, we obtain a map into $Y$ from the space
parametrising such chains together with the choice of a point on the last segment.
Furthermore, the map is constant on small deformations of a chain fixing the marked 
point. It then follows that the map defined on chains descends to
a rational map  $\tilde\phi : \tilde X \dashrightarrow Y$ 
from a generically \'etale cover $\tilde X \to X$. There is a biholomorphic lift
$\tilde U \subset \tilde X$ of $U \subset X$ 
such that $\tilde\phi$ is defined on $\tilde U$
and $\tilde\phi|_{\tilde U} = \varphi$ via the identification $\tilde U\simeq U$.
Finally, $\tilde\phi$ maps $\sM$-curves to $\sN$-curves.

With analytic neighbourhoods replaced by formal ones, and considering
families rather than closed points, this part of the argument works
in arbitrary characteristic.

\item
One now checks that $\tilde\phi$ can in fact be descended
to a \emph{birational} map $\phi_0:X \dashrightarrow Y$. 
Replacing $\tilde X$ with the graph of $\tilde\phi$,
a careful examination of the construction of $\tilde\phi$
shows that $\tilde X \to X$ is trivialised over
a general $\sM$-curve. Then, a genericity argument,
together with simply-connectedness of $X$, rules out
ramification in $\tilde X \to X$, so that the latter is birational.
A similar argument is used to check that
$\tilde X \dashrightarrow Y$ is unramified in codimension one,
so that $\phi_0: X \dashrightarrow Y$ is birational.

Simply-connectedness of complex Fano manifolds,
being a consequence of their (separable) rational connectedness,
will have to be added as a hypothesis in positive characteristic.

\item
The open subvariety of $X$ on which $\phi_0$ is defined
contains a free $\sM$-curve, and thus so does its image.
It follows that $\phi_0$ induces an isomorphism
of complements of closed subvarieties of codimension at least two.
A standard argument with plurianticanonical embeddings
shows that in this case $\phi_0$ extends to an isomorphism $\phi : X \to Y$.
\end{enumerate}

\section{Main results}
The basic elements of the reasoning laid out in the former section
have to be carefully recast in an algebro-geometric language suitable for
positive characteristic. Some carry through without much change,
some need additional preparation to avoid pathological situations,
while others will merely be in a relation of analogy to the actual
arguments. In particular, intuitive analytic-local constructions
are replaced with somewhat more technical methods of formal geometry.

Chapter \ref{chap:tools} sets up the necessary theoretical foundations:
spaces of formal arcs, families of rational curves, and cominuscule
homogeneous varieties. We work with unparametrised pointed arcs of infinite
order, and define corresponding parameter spaces intrinsically, rather
than as an inverse limit. 
Since no convenient reference seems to be
available for this setting, we work out some basic properties. Proposition
\ref{pro:closed-orbits} may be of independent interest here.
The section on rational curves is mostly concerned with introducing
the notation for different objects associated with a family
of curves, and gathering some standard facts. The main reference
is~\cite{kollar}. We also cite the theorem of Kebekus~\cite{kebekus}
in a suitable form. Proposition \ref{pro:linear-nondegeneracy}
is a positive-characteristic analogue of~\cite[Prop. 13]{hwang-mok-rigidity}.
The substantial difference is that we need to consider a purely inseparable
quotient instead of a foliation by subvarieties. The final section
introduces the class of cominuscule homogeneous varieties, their
VMRT and some auxiliary intersection numbers.
 
Chapter \ref{chap:extension} states and proves a characteristic-free
analogue of the extension theorem. Rather than starting with
a single formal isomorphism, we choose to work with an entire
bundle parametrising isomorphisms between formal neighbourhoods
of points on two varieties $X$, $Y$. The bundle admits a natural 
stratification,\footnote{In the sense of an identification of infinitesimally
close fibres.}
restricting to a subscheme cut out by the condition of compatibility
with arcs associated with given families $\sM$, $\sN$ of rational curves satisfying
suitable conditions.
The central result is Proposition \ref{pro:extension-m1}, showing
that the latter subscheme admits natural horizontal (generic) trivialisations
along $\sM$-curves. This plays a role analogous to the analytic continuation
discussed in the previous section. Following Step 2 and Step 3 of the original
argument, we arrive at Proposition \ref{pro:c-descent}, producing
a horizontal (generic) trivialisation over $X$. Finally, Step 4 leads
to Theorem \ref{thm:extension} and its Corollary.

Chapter \ref{chap:rigidity} states and proves our version
of the rigidity theorem over an algebraically closed field of
positive characteristic.
The main problem with applying the strategy
outlined in the previous section is potential inseparability of
various evaluation maps, with the most serious consequence being
failure of smoothness of the space of minimal degree rational curves
through a general point of a nonsingular variety, in this case
the special fibre of a degeneration. Such behaviour
can be ruled out by imposing a lower bound on the characteristic. We are not
aware of a universal bound that would not require the knowledge an explicit very
ample invertible sheaf on the variety. Hence the main result of this
chapter, Theorem \ref{thm:rigidity}, refers to the notion of $d$-rigidity:
we call a cominuscule homogeneous variety $G/P$ $d$-rigid, if 
it does not admit nontrivial smooth projective degenerations with
$\OO_{G/P}(d)$ extending to a very ample invertible sheaf on the special fibre.
An inductive application of the Theorem establishes $d$-rigidity of $G/P$
under the assumption that the characteristic be greater than an explicit
integer, depending only on $d$ and $G/P$, and computable in terms of Schubert
calculus on the latter.

\section{Further directions}
It would be interesting to investigate the possibility
of applying our method, using formal arcs
instead of just VMRT, to other of the multitude of results
obtained by Hwang and Mok. Examples include: rigidity of
generically \'etale morphisms to cominuscule homogeneous varieties,
Lazarsfeld's problem for morphisms
from cominuscule homogeneous varieties,
rigidity of non-cominuscule homogeneous varieties (all approachable
via the extension theorem, see~\cite{hwang-trieste} for a review).

Another outstanding issue is that of improving the bounds on the characteristic
in our version of the rigidity theorem. The way we had obtained them
being far from subtle, there should be an approach exploiting the particular
setting of the degeneration problem to a greater degree. We also do not know
how far our bounds are from being effective: a counterexample to
rigidity in low characteristic should be enlightening (we know none).
 
\section{Conventions and notation}
\label{sec:notation}

We work over an algebraically closed field $k$. We will mostly
be interested in the case $\Char k > 0$, although we do not
assume this until Chapter \ref{chap:rigidity}.
A \emph{presheaf} will mean a
presheaf of sets on the category of $k$-schemes. A
\emph{sheaf} will mean a sheaf for the \textit{fpqc} topology.
We do not employ any notational convention to distinguish
between presheaves, sheaves, formal schemes and schemes.

Given a morphism $X \to S$ of presheaves, 
$\underline\Aut_SX$ denotes the presheaf whose value
at $T$ is the set of pairs $(T\to S,\varphi)$ where
$\varphi \in \Aut_T X_T$.
If $f:Z \to X$ is a morphism of presheaves over $S$,
$\underline\Aut_S(X,f)$ denotes the sub-presheaf
of $\underline\Aut_SX$ whose value at $T$ is
the subset of $(T\to S,\varphi)$ such that
$\varphi \circ f_T=f_T$. Given a second morphism $Y \to S$ 
of presheaves, 
$\underline\Hom_S(X,Y)$ denotes the presheaf whose value
at $T$ is the set of pairs $(T\to S,\psi)$ where
$\psi \in \Hom_T(X_T,Y_T)$. If $g:Z \to Y$ is a morphism
of presheaves over $S$, we let
$\underline\Hom_S(X,Y; f, g)$ be the sub-presheaf
of $\underline\Hom_S(X,Y)$ whose value at $T$ is
the subset of $(T\to S,\psi)$ such that
$\psi\circ f_T = g_T$.
There is are sub-presheaves $\underline\Isom_S(X,Y)
\subset\underline\Hom_S(X,Y)$ and
$\underline\Isom_S(X,Y;f,g) \subset \underline\Hom_S(X,Y;f,g)$
whose values at $T$ are restricted to those $\psi$ which are isomorphisms.

Given a morphism $X \to S$ of presheaves,
$(X/S)^i$ denotes the $i$-fold product $X\times_S \dots \times_SX$,
together with the natural morphism to $S$.
For $X$ a presheaf equipped with a pair of structure morphisms
$X \rightrightarrows S$, referred to as the left and right structure map, 
$(S \backslash X / S)^i$ denotes the $i$-fold product
$X\times_S \dots \times_S X$, together with the pair
of structure morphisms into $S$ given by the left
structure morphism from the leftmost factor and
the right structure morphism from the rightmost factor.
Given morphisms $Y \to X \to S$ of presheaves, 
$\prod(Y/X/S)$ denotes the presheaf whose value at $T$
is the set of pairs $(T\to S, \sigma)$ where
$\sigma : X_T \to Y_T$ is a section of the pullback
of $Y\to X$. When applied to sheaves, all these constructions
yield sheaves.

Given a morphism of sheaves
$X \to S$ and a sheaf of groups $G$ over $S$ acting from the
right on $X$, we have
the quotient sheaf $X/G$ over $S$. Its value at $T$
is the set of equivalence classes of diagrams 
$$\begin{CD}
T' @>>> X \\
@VVV @VVV \\
T @>>> S
\end{CD}$$
where
$T'\to T$ is a covering, and $T' \to X$
is such that the induced morphism
$T'\times_TT' \to X\times_SX$ factors through
the action morphism $X\times_SG \to X\times_SX$.
Two diagrams, the other say with $T\leftarrow T'' \to X$, are identified if 
$T' \times_T T'' \to X\times_SX$ factors through the action morphism.
If instead $G$ acts on $X$ from the left, we have
an analogous construction of $G\backslash X$.

A morphism $Y \to X$ of presheaves is \emph{schematic}
if for every morphism $T\to X$ from a scheme, the
pullback $T\times_XY$ is a scheme. Furthermore,
a schematic morphism $Y \to X$ is \emph{affine}
if for every morphism $T \to X$ from a scheme,
$T\times_XY\to T$ is affine.

We will mostly work with locally Noetherian formal schemes.
Occasionally, we allow \emph{adic} morphisms $Y \to X$
from a (not necessarily locally Noetherian) formal scheme to a 
locally Noetherian formal scheme,
i.e. such that the pullback of the underlying scheme
of $X$ gives an underlying scheme of $Y$. Such a morphism
is in particular schematic, and the ideal of definition
in $\OO_Y$ is locally finitely generated.

If $X \to S$ is a morphism of locally Noetherian formal schemes,
we denote by $(X/S)^\sharp \rightrightarrows X$ 
the completion of $(X/S)^2$ along the diagonal, 
together with the induced pair of structure morphisms into $X$.
It is a formal subscheme of $(X/S)^2$, containing $X$ and sharing the same
underlying reduced scheme. We sometimes consider
$(X/S)^\sharp$ as a bundle over $X$, using the \emph{left}
structure map. 
In particular, we let
$(X/S)^\sharp_x = x^*(X/S)^\sharp$ for a point or geometric point
$x$ of $X$.
The relative tangent sheaf
$T_{X/S}$ is as usual defined to be the $\OO_X$-module
dual to $\sI/\sI^2$, where $\sI$ is the ideal of the
diagonal in $(X/S)^2$. We let
$T_{X/S,x} = x^*T_{X/S}$. Given a sheaf $\sE$ over $X$,
an $S$-stratification on $\sE$ is an isomorphism
$X^\sharp\times_X\sE \to \sE\times_XX^\sharp$ satisfying
the usual cocycle condition. For example,
$X^\sharp\times_X\sE$ is canonically stratified. A morphism
of stratified sheaves is \emph{horizontal} if
it is compatible (in an obvious manner) with the stratifications.


The base $\Spec k$ will be usually omitted from notation,
so that $\prod(Y/X) = \prod(Y/X/k)$,
$X^i = (X/k)^i$, $X^\sharp = (X/k)^\sharp$,
$T_X = T_{X/k}$, $T_{X,x} = T_{X/k,x}$, 
etc. unless explicitly defined to mean otherwise.
We fix an origin $0$ in $\PP^1$. 
Given a  reduced group scheme $G$ over $k$,
a morphism $X \to S$ together with a $G$-action is a \emph{$G$-principal bundle}
if $X \times G \to X\times_SX$ is an isomorphsm, and
there is an \'etale cover $S' \to S$ together
with a section $S' \to S'\times_SX$.

We will only consider inverse/direct systems
indexed by integers.
A morphism $X\to S$ of formal schemes
is of \emph{pro-finite type} if it is
the limit of an inverse system of morphisms $X_i\to S$
of finite type. 
A scheme is called \emph{pro-algebraic} 
if it is of pro-finite type over $k$.
An affine group scheme is called \emph{pro-unipotent}
if it is the limit of a countable sequence
of successive extensions by $\GG_a$ (starting with the trivial
group).
$\AA^\infty$ denotes the spectrum of a polynomial
algebra in countably infinitely many variables.
In particular, a pro-unipotent group is
isomorphic to $\AA^\infty$ as a scheme.

%% file: chap2.tex
\chapter{Technical tools}\label{chap:tools}

\input{chap2-1}

\input{chap2-2}

\input{chap2-3}

%% file: chap2-1.tex
\section{Formal arcs}

\subsection{Formal discs}
The arguments in Chapters~\ref{chap:extension} and
\ref{chap:rigidity} rely on the study of families of formal arcs
on a nonsingular variety. A formal arc is a morphism from 
a one-dimensional formal disc. A formal arc through a closed
point $x$ of a nonsingular variety $X$ factors through
the completion $\hat X$ of $X$ at $x$, a formal disc of dimension
$\dim X$. We thus begin with the description of spaces of morphisms
between formal discs. The reader is referred to Section~\ref{sec:notation}
for notation and conventions.

\begin{defn}
A formal scheme $\hat X$ is called a \emph{formal disc}
if it is isomorphic to $\Spf k[[z_1,\dots,z_n]]$ for some $n\ge0$.
Its unique $k$-point is called the \emph{origin} of $\hat X$.
The integer $n$ is referred to as the \emph{dimension} of $\hat X$.
\end{defn}

\begin{lem}\label{lem:hom-formal-discs}
Let $\hat X$ and $\hat Y$ be formal discs of dimensions
$n>0$ and $m>0$
with origins
$x$ and $y$. Then there is an isomorphism
$$ \underline\Hom(\hat X,\hat Y; x,y)\simeq\AA^\infty $$
such that the tangent map to $\underline\End(T_{\hat X,x},T_{\hat Y,y})
\simeq\AA^{nm}$ corresponds to a linear projection.
\end{lem}
\begin{proof}
Fix identifications $\hat X = \Spf k[[z_1,\dots,z_n]]$ and
$\hat Y = \Spf k[[w_1,\dots,w_m]]$.
Let $B^+$ be the set indexing non-constant coefficients
of a power series in $k[[z_1,\dots,z_n]]$, and let
$$P = k[a_{i\beta}\ |\  {1 \le i \le m, \beta \in B^+}]$$
be a polynomial algebra over $k$ with free generators indexed
by $\{1,\dots,m\}\times B^+$.
There is a
continuous homomorphism
$$ \Phi : k[[w_1,\dots,w_m]] \to P[[z_1,\dots,z_n]] $$
such that the $\beta$-th coefficient of $\Phi(w_i)$ is $a_{i\beta}$.
We claim that 
$$\underline\Hom(\hat X,\hat Y ; x,y) \simeq \Spec P $$
with the universal morphism
$\underline\Hom(\hat X,\hat Y;x,y) \times \hat X \to \hat Y$
corresponding to $$\Phi_*:\Spf P[[z_1,\dots,z_n]] \to \Spf k[[w_1,\dots,w_m]].$$ 
Indeed, by the universal property of
$\underline\Hom(\hat X,\hat Y;x,y)$, $\Phi_*$ induces a morphism
$$ \phi :\Spec P \to \underline\Hom(\hat X,\hat Y;x,y) $$
of sheaves. To construct its inverse, we can restrict
to affine schemes. Given a $k$-algebra $Q$, and an element
$f \in \underline\Hom(\hat X,\hat Y;x,y)(\Spec Q)$ defined by
$$f^* :  Q[[w_1,\dots,w_m]] \to  Q[[z_1,\dots,z_n]]$$
we let $\psi(f) : \Spec Q \to \Spec P$ be the morphism
such that
$$\psi(f)^* : k[a_{i\beta}\ |\ {1\le i\le m,\beta\in B^+}] \to Q$$
sends $a_{i\beta}$ to the $\beta$-th coefficient of $f^*w_i$.
This defines a morphism
$$ \psi : \underline\Hom(\hat X,\hat Y;x,y) \to \Spec P $$
such that $\phi$ and $\psi$ are mutual inverses.

Finally, denoting by $j \in B^+$ the coefficient of $z_j$, we have
that the tangent map to $$\underline\End(T_{\hat X,x},T_{\hat Y,y})
\simeq \Spec k[a_{ij}\ |\ {1\le i\le m, 1 \le j \le n}]$$ is induced
by the natural inclusion $k[a_{ij}] \to k[a_{i\beta}]$.
\end{proof}

\begin{lem}\label{lem:aut-formal-disc}
Let $\hat X$ be a formal disc of dimension $n>0$, with origin $x$.
Then $\underline\Aut(\hat X,x)$ is a pro-algebraic affine group scheme.
Furthermore, the action of origin-preserving automorphisms
on the tangent space at the origin induces an exact sequence
$$
0 \to \sR_u\underline\Aut(\hat X,x) \to \underline\Aut(\hat X,x)
\to \GL(T_{\hat X,x}) \to 1 
$$
with a pro-unipotent kernel.
\end{lem}
\begin{proof}
Recall first that an endomorphism
of a power series algebra is invertible if its linear part is
(cf.~\cite{bourbaki}). It follows that there is a Cartesian diagram
$$\begin{CD}
\underline\Aut(\hat X,x) @>>> \underline\Hom(\hat X,\hat X,x,x) \\
@VVV @VVV \\
\GL(T_{\hat X,x}) @>>> \underline\End(T_{\hat X,x})
\end{CD}$$
so that, by Lemma~\ref{lem:hom-formal-discs},
$$ \underline\Aut(\hat X,x) \simeq \GL_n \times_{\AA^{n^2}}\AA^\infty $$
an affine pro-algebraic group scheme, with an epimorphism
onto $\GL(T_{\hat X,x})$.

Let $X_r$ be the $r$-th infinitesimal neighbourhood
of $x$ in $\hat X$, so that
$\hat X = \varinjlim X_r$. 
Letting $K_r$ be the kernel in
$$ 0 \to K_r \to \underline\Aut(\hat X,x) \to \underline\Aut(X_r,x)
\to 1 $$
we have short exact sequences
$$ 0 \to K_r/K_{r+1} \to K_1/K_{r+1} \to  K_1/K_r \to 0 $$
Using the notation of the proof of Lemma~\ref{lem:hom-formal-discs},
applied to $\underline\Hom(\hat X,\hat X,x,x)$, we have
$ K_r/K_{r+1} \simeq \prod_{i,\beta}\GG_a $
where the product is over $1 \le i \le n$ and $\beta \in B^+$
such that $\deg \beta = r$. It follows that
$K_1/K_r$ is unipotent, so that
$$ \sR_u\underline\Aut(\hat X,x) = \varprojlim K_1 / K_r $$
is pro-unipotent.
\end{proof}

Morphisms from formal discs to schemes admit well-behaved parameter
spaces as well. The following result is sufficiently general for
our purposes.
\begin{lem}\label{lem:sec-over-formal-disc}
Let $\hat X$ be a formal disc of dimension $n>0$, with origin $x$.
Let $Y \to \hat X$ be an adic morphism from a (not necessarily
locally Noetherian) 
formal scheme, so that the fibre $Y_x=x\times_XY$ is a scheme.
Then $\prod(Y/\hat X)$ is a scheme, and the morphism
$e_x : \prod(Y/\hat X) \to Y_x$ given by evaluation at $x$
is affine. If $Y \to \hat X$ is of pro-finite type, then
so is $e_x$.
\end{lem}
\begin{proof}
If $Y_x = \bigcup U_i$ is a cover by open subschemes,
then $\prod(Y/\hat X) = \bigcup e_x^{-1}(U_i)$ is a 
cover by open subfunctors, and it is enough to
check that each $e_x^{-1}(U_i)$ is a scheme. We can thus assume
$Y$ is affine.
Identifying $\hat X$ with $\Spf k[[z_1,\dots,z_n]]$, we have
$Y = \Spf R$ where $R$ is a topological
algebra over $k[[z_1,\dots,z_n]]$, with topology induced
by $(z_1,\dots,z_n)R$. Consider a presentation
$R = S/I$
where $S$ is a polynomial algebra over $k[[z_1,\dots,z_n]]$,
possibly of infinite type, with topology induced by
$(z_1,\dots,z_n)S$. Let $A$ be the set
indexing free generators $y_\alpha$ of $S$, and let
$B$ be the set indexing coefficients of a power series
in $k[[z_1,\dots,z_n]]$.
Let $$P = k[a_{\alpha\beta}\ |\ {\alpha\in A,\beta\in B}]$$ 
be a polynomial algebra
over $k$ with free generators $a_{\alpha\beta}$ indexed by $A\times B$.
There is a homomorphism
$$ \Phi : S \to P[[z_1,\dots,z_n]] $$
such that for $\alpha\in A$, $\beta\in B$,
the $\beta$-th coefficient of $\Phi(y_\alpha)$
is $a_{\alpha\beta}$.
Let
$$ \Psi : I\times B \to P $$
be the map sending $(s,\beta)$ to the 
$\beta$-th coefficient
of the power series $\Phi(s)$. Finally let
$J \subset P$ be the ideal generated by the image $\Psi(I\times B)$.
Then $\Phi$ factors through
$$\bar\Phi : R=S/I \to (P/J)[[z_1,\dots,z_n]].$$
We claim that
$$ \prod(Y/\hat X) \simeq \Spec P/J $$
with the universal morphism
$ \prod(Y/\hat X) \times \hat X \to Y $
corresponding to $$\bar\Phi_*: \Spf (P/J)[[z_1,\dots,z_n]]\to
\Spf R. $$
Indeed, by the universal property of $\prod(Y/\hat X)$,
$\bar\Phi_*$ induces a morphism
$$\phi : \Spec(P/J) \to \prod(Y/\hat X)$$
of sheaves. To construct its inverse, we can restrict to affine schemes.
Given a $k$-algebra $Q$,
and an element
$f \in \prod(Y/\hat X)(\Spec Q)$
defined by
$$f^* : Q\otimes (S/I) \to Q[[z_1,\dots,z_n]]$$
we let $\psi(f) : \Spec Q \to \Spec P/J$ be the unique morphism
such that
$$\psi(f)^* : k[a_{\alpha\beta}\ |\ {\alpha\in A,\beta\in B}]/J
\to Q $$
sends $\bar a_{\alpha\beta}$ to the $\beta$-th coefficient
of $f^*y_\alpha$. Note that for each
$(s,\beta)\in I\times B$, $f^*s=0$ so that
is $\beta$-th coefficient is zero, so that $\psi(f)^*$
takes the generators of $J$ to zero. Hence $\psi(f)^*$
is well-defined.
This defines a morphism
$$ \psi : \prod(Y/\hat X) \to \Spec (P/J) $$
and that $\phi$ and $\psi$ are mutual inverses.

Finally, if $R$ is countably generated over $k[[z_1,\dots,z_n]]$,
then $A$ and $A\times B$ are countable, so that
$P/J$ is countably generated over $k$. Hence
$Y \to \hat X$ being pro-finite type implies the same for
$\prod(Y/\hat X) \to Y_x$.
\end{proof}

Naturally, one would like to have a relative notion
of a formal disc. For our purposes, the following is
the most convenient
(note that
in the definition below, dimensions of the fibres are
locally constant, i.e. $n_i=n_j$ unless $S_i$, $S_j$ are disjoint).
\begin{defn}A morphism $\hat X \to S$
of locally Noetherian formal schemes, together with a section
$S \to \hat X$,
is called a 
\emph{bundle of formal discs}
if there is a Zariski open cover $S = \bigcup S_i$ such that 
$S_i\times_S\hat X
\simeq S_i \times \Spf k[[z_1,\dots,z_{n_i}]]$ over $S_i$
for some $n_i\ge 0$,
and the section $S_i \to S_i\times_S \hat X$ corresponds
to the pullback of the origin $\Spec k \to \Spf k[[z_1,\dots,z_{n_i}]]$.
\end{defn}

\begin{lem}\label{lem:completion-at-section}
Let $X \to S$
be a smooth morphism of locally Noetherian schemes. Then
$(X/S)^\sharp$ together with the diagonal section
 is a bundle of formal discs
over $X$.
If, moreover, $S$ is reduced and $\sigma \in X(S)$ is a section, then
$\sigma^*(X/S)^\sharp$ is naturally isomorphic
to the completion of $X$ along $\sigma(S)$.
\end{lem}
\begin{proof}
Since the question is local, we can assume
there is an \'etale morphism
$f:X \to \AA^n_S$ and
$f\circ \sigma : S \to \AA^n_S$ is the
zero-section. Since $f$ is \'etale,
it induces an isomorphism $(X/S)^\sharp \simeq f^*(\AA^n_S/S)^\sharp$
and an isomorphism 
of the completion of $X$ along $\sigma(S)$
onto the completion of $\AA^n_S$ along $(f\circ\sigma)(S)$.
We can thus replace $X$ with $\AA^n_S$ and
$\sigma$ with the zero-section. Then $(X/S)^\sharp$ is identified with
$S\times (\AA^n)^\sharp$, and the completion of
$X$ along $\sigma(S)$ is the product of $S$ with
the completion of $\AA^n$ at $0$. We are thus finally
reduced to the case $S=\Spec k$, $X=\AA^n$, $\sigma=0$,
where the result follows by straightforward inspection.
\end{proof}

The condition that $(X/S)^\sharp$ be a bundle of formal discs is
a variant of what in ~\cite{SGA-I} and~\cite{EGA-IV-4} is caled
\emph{lissit\'e differentielle} (it coincides with the latter
at least for finite type morphisms of locally Noetherian schemes).
We thus make the following definition.
\begin{defn}
A morphism $f:X \to S$ of locally Noetherian formal schemes
is \emph{differentially smooth (with differential dimension $n$)} if 
$(X/S)^\sharp$, together with the diagonal section, 
is a bundle of formal discs (of dimension $n$) over $X$.
\end{defn}

By Lemma~\ref{lem:completion-at-section},
a smooth morphism of locally Noetherian schemes
is differentially smooth. Conversely,
one can show that a flat, finite type differentially
smooth morphism of locally Noetherian schemes is  smooth~\cite{SGA-I},
although we will not need this fact.
At the same time,
a bundle $\hat X\to S$ of formal discs is differentially smooth 
since $(\hat X/S)^\sharp = (\hat X/S)^2$.
Sections of a differentially smooth morphism to
a formal disc are parametrized by a particularly
simple object. 
\begin{lem}\label{lem:sec-of-diff-smooth}
Let $\hat X$ be a formal disc with origin $x$,
$Y$ a locally Noetherian formal scheme, and
$Y \to \hat X$ a differentially smooth
morphism with positive differential dimension. Then
$\prod(Y/\hat X) \to Y_x$ is a Zariski-locally trivial
$\AA^\infty$-bundle.
\end{lem}
\begin{proof}
Note that for any scheme $T \to \hat X$,
and a morphism $f:T\times \hat X \to T\times_{\hat X}Y$
over $T$, the induced morphism
$\langle (f\circ x)\times \hat X, f \rangle
: T\times\hat X \to T\times_{\hat X} (Y/X)^2$
factors through $T\times_{\hat X}(Y/X)^\sharp$.
We thus have an isomorphism
\begin{eqnarray*}
\prod(Y/\hat X) &\simeq& 
\underline\Hom_{Y}(Y \times\hat X, (Y/\hat X)^\sharp ; 
Y \times x, \Delta_Y) \\
&\simeq&
\underline\Hom_{Y_x}(Y_x \times\hat X, Y_x \times_Y (Y/\hat X)^\sharp ; 
Y_x \times x, Y_x\times_Y \Delta_Y)
\end{eqnarray*}
where $\Delta_Y : Y \to (Y/\hat X)^\sharp$ is the diagonal morphism.
By hypothesis, $Y_x\times_Y(Y/\hat X)^\sharp$ is a bundle
of positive-dimensional formal discs over $Y_x$. 
Since the problem is local, we
can assume that the bundle is trivial, so that
$$ \prod(Y/\hat X) \simeq Y_x \times \underline\Hom(\hat X,\hat Y ; x,y)
=Y_x\times\underline\Hom(\hat X,\hat Y;x,y) $$
where $\hat Y$ is a formal disc of positive dimension, with origin $y$.
Then, by Lemma~\ref{lem:hom-formal-discs}, we have that
$\prod(Y/\hat X) \simeq Y_x\times\AA^\infty$ as desired.
\end{proof}

\subsection{Space of unparametrised pointed arcs}

Let $\hat\PP^1$ denote the completion of $\PP^1$ at $0$.
It is a one-dimensional formal disc, so that in particular
we have the pro-algebraic group scheme 
$\underline\Aut(\hat\PP^1,0)$. Given a differentially smooth
morphism $X\to S$ with positive differential dimension
of locally Noetherian formal schemes,
we define
$$\underline\Imm_S(S\times\hat\PP^1,X) \subset
\underline\Hom_S(S\times\hat\PP^1,X)$$
to be the subsheaf whose value at $T/S$ consists
of morphisms $\gamma:T\times\hat\PP^1 \to X_T$
inducing a nowhere-vanishing map
$\gamma':\OO_T\otimes T_{\PP^1,0} \to \gamma|_{T\times 0}^* T_{X_T/T}$.
That is, $\underline\Imm_S(S\times\hat\PP^1,X)$
is the space of families of parametrised formal arcs in $X/S$, unramified
at the origin $0$. 
The map sending $\gamma$ to the image of $\gamma'$ 
induces a morphism $\underline\Imm_S(S\times\PP^1,X)
\to \PP T_{X/S}$ into the projectivised tangent bundle.
The automorphism group $\underline\Aut(\hat\PP^1,0)$
has a natural right action on $\underline\Imm_S(S\times\hat\PP^1,X)$,
reparametrising the arc while preserving
the tangent direction at $0$, and thus compatible with the morphism to
 $\PP T_{X/S}$.
We define the space of \emph{unparametrised
pointed arcs} in $X/S$ to be the quotient
$$ \Arc_{X/S} = \underline\Imm_S(S\times\hat\PP^1,X) /
\underline\Aut(\hat\PP^1,0), $$
in the category of sheaves over $\PP T_{X/S}$.
\begin{lem}\label{lem:arc-affine}
Let $X\to S$ be a differentially smooth morphism with positive
differential dimension
of locally Noetherian formal schemes. Then
$\Arc_{X/S} \to \PP T_{X/S}$ 
is a Zariski-locally trivial $\AA^\infty$-bundle.
\end{lem}
\begin{proof}
Since the problem is local, we can assume by differential smoothness
that there is a positive-dimensional formal disc $\hat X$ with origin $x$,
such that
$(X/S)^\sharp \simeq X \times \hat X$
inducing an isomorphism $\PP T_{X/S} \simeq X \times
\PP T_{\hat X}$. Arguing as in the proof Lemma~\ref{lem:sec-of-diff-smooth}
we then have an isomorphism
$$ \underline\Imm_S(S\times\hat\PP^1,X) \simeq X \times
\underline\Imm(\hat \PP^1,\hat X) $$
equivariant under the action of $\underline\Aut(\hat\PP^1,0)$,
and compatible with the morphisms to
$\PP T_{X/S}$ and $\PP T_{\hat X}$. It will thus be enough to prove
that
$$ \Arc_{\hat X} \to \PP T_{\hat X} $$
is a Zariski-locally trivial $\AA^\infty$-bundle.

Identifying $\hat X$ with the completion of
$\GG_a \times \dots \times \GG_a$ at identity, we turn
it into a formal group scheme. The action of $\hat X$
on itself trivialises $\Arc_{\hat X}$ and $\PP T_{\hat X}$
so that
$$ \Arc_{\hat X} \simeq \hat X \times \Arc_{\hat X,x},\quad
\PP T_{\hat X} \simeq \hat X \times \PP T_{\hat X,x} $$
compatibly with the morphism $\Arc_{\hat X}\to\PP T_{\hat X}$.
It is thus enough to check that $\Arc_{\hat X,x} \to \PP T_{\hat X,x}$
is a Zariski-locally trivial $\AA^\infty$-bundle.

Now, identifying $\hat X = \Spf k[[z_1,\dots,z_n]]$
so that $\PP T_{\hat X,x} = \Proj k[z_1,\dots,z_n]$,
we have that $\PP T_{\hat X,x}$ is covered by open affines
$D(z_i)$, and $\Arc_{\hat X,x}$ is covered by open subfunctors
$U_i = D(z_i)\times_{\PP T_{\hat X,x}}\Arc_{\hat X,x}$. It will be
enough to show that $U_i \simeq D(z_i) \times \AA^\infty$,
and by permuting the variables we only need to consider $i=n$.
Let $\hat D = \Spf k[[z_1,\dots,z_{n-1}]]$ and identify
$\hat\PP^1 = \Spf k[[z_n]]$, inducing isomorphism
$\hat D \times \hat \PP^1 \simeq \hat X$. There is a natural
isomorphism
$$ \underline\Hom(\hat\PP^1,\hat D) 
\times \underline\Hom(\hat\PP^1,\hat\PP^1) \simeq
\underline\Hom(\hat\PP^1,\hat D \times \hat \PP^1) $$
which, pulled back to the $k$-point
$\id_{\hat \PP^1}$ of $\underline\Hom(\hat\PP^1,\hat\PP^1)$,
and restricted over the origin $\bar x \in \hat D(k)$,
gives a monomorphism
$$ \tilde\phi:\underline\Hom(\hat\PP^1,\hat D ; 0,\bar x) \to 
D(z_n)\times_{\PP T_{\hat X,x}}\underline\Imm(\hat\PP^1,\hat X). $$
Composing with projection to the quotient by $\underline\Aut(
\hat\PP^1,0)$, we obtain
$$ \phi : \underline\Hom(\hat\PP^1,\hat D ; 0,\bar x)
\to U_n $$

We first check that $\phi$ is a monomorphism. 
Consider two $T$-points of $\underline\Hom(\hat\PP^1,\hat D;
0,\bar x)$, corresponding to morphisms
$b,c : T\times\hat\PP^1 \to T\times\hat D$. If $\phi(b)=\phi(c)$,
then there is a covering $T'\to T$ and an automorphism
$g:T'\times\hat\PP^1 \to T'\times\hat\PP^1$ such that
$\tilde\phi(b)_{T'}\circ g = \tilde\phi(c)_{T'}$. But then
$$g = \pr_{T'\times\hat\PP^1} \circ \tilde\phi(b)_{T'}\circ g
= \pr_{T'\times\hat\PP^1}\circ\tilde\phi(c)_{T'}
= \id_{T'\times\hat\PP^1}$$
so that $\tilde\phi(b)_{T'} = \tilde\phi(c)_{T'}$,
thus $b_{T'}=c_{T'}$ and $b=c$ by descent.

To check that $\phi$ is an epimorphism,
consider a $T$-point $f$ of $U_n$,
represented by a covering $T' \to T$ and a morphism
$f':T'\times\hat\PP^1 \to T'\times\hat D\times\hat\PP^1$ satisfying
the condition that 
$$\pr_1^*f' = \pr_2^*f' \circ g : T'\times_TT'\times\hat\PP^1\to 
T'\times_TT'\times\hat D\times\hat\PP^1$$
for some automorphism $g: T'\times_TT'
\times\hat\PP^1 \to T'\times_TT'\times\hat\PP^1$,
where
$\pr_1,\pr_2 : T'\times_TT'\to T'$ are the two projections.
Furthermore, since $\pr_{T'\times\hat\PP^1}\circ f' :
T'\times\hat\PP^1 \to T'\times\hat\PP^1$ is unramified at the
zero section, it is an automorphism. We then let
$$ c' = \pr_{T'\times\hat D} \circ f'\circ 
(\pr_{T'\times\hat\PP^1}\circ f')^{-1} : 
T'\times\hat\PP^1 \to T'\times\hat D, $$
defining a $T'$-point of $\underline\Hom(\hat\PP^1,\hat D;0,\bar x)$.
It is immediate that 
$\tilde\phi(c')=f' \circ (\pr_{T'\times\hat\PP^1}\circ f')^{-1}$, so that 
$\phi(c') = T'\times_T f$.
Pulling $c'$ back by $\pr_1,\pr_2:T'\times_TT' \rightrightarrows T'$, we have
\begin{eqnarray*}
\pr_1^*c' &=& 
\pr_{T'\times_TT'\times \hat D}
\circ \pr_1^*f' \circ (\pr_{T'\times_TT'\times\hat\PP^1} \circ \pr_1^*f')^{-1}
\\ &=&
\pr_{T'\times_TT'\times \hat D}
\circ \pr_2^*f' \circ g \circ g^{-1} \circ (\pr_{T'\times_TT'\times\hat\PP^1} \circ \pr_2^*f')^{-1}
\\ &=& \pr_2^*c'
\end{eqnarray*}
so that $c'$ descends to a $T$-point corresponding to
$c:T\times\hat\PP^1 \to T\times\hat D$.

We thus have an isomorphism
$U_n \simeq \underline\Hom(\hat\PP^1,\hat D;0,\bar x)$.
By Lemma~\ref{lem:hom-formal-discs}, $U_n \simeq \AA^\infty$
and the morphism to $D(z_i)\simeq\AA^n$ is a linear projection,
so that $U_n \simeq D(z_i)\times\AA^\infty$ as a $D(z_i)$-scheme.
\end{proof}

\begin{cor}\label{cor:arc-torsor}
In the setting of Lemma~\ref{lem:arc-affine},
$ \underline\Imm_S(S\times\hat\PP^1,X) \to \Arc_{X/S}$
is a Zariski-locally trivial $\underline\Aut(\hat \PP^1,0)$-torsor.
\end{cor}
\begin{proof}
A morphism from $\hat\PP^1$ to $X$ is a formal closed
immersion if it is unramified at $0$, so that
$\underline\Imm_S(S\times\hat\PP^1,X)$ is a pseudo-torsor
for $\underline\Aut(\hat\PP^1,0)$.
Zariski-local sections are provided by the morphism
$\tilde\phi$ of the former proof.
\end{proof}

Remaining in the previous context, we introduce the pointwise space
of families of arcs determined by their tangent directions.
Considering the fibration $\Arc_{X/S} \to \PP T_{X/S}$,
let $$ \Hilb \PP T_{X/S} \leftarrow \sU \rightarrow \PP T_{X/S}$$
be the universal family over the relative Hilbert scheme\footnote{
Note that, by differential smoothness, $\PP T_{X/S}$ is locally
trivial, so that $\Hilb\PP T_{X/S}$ is locally just a pullback
of $\Hilb \PP^{n-1}$, where $n$ is the differential dimension of $X/S$.
}
of the projectivised tangent bundle.
We then define $\ArcHilb_{X/S}$ to be the sheaf of sections
$$
\ArcHilb_{X/S}=\prod\left(\sU \times_{\PP T_{X/S}} \Arc_{X/S}\ /\ 
 \sU\ /\ \Hilb \PP T_{X/S}\right)
$$
so that its value at $f:T \to \Hilb\PP T_{X/S}$
is the set of sections of $T \times_X \Arc_{X/S}$
over the subscheme of $T\times_X\PP T_{X/S}$ defining $f$.
In particular, if $X$ is differentially smooth over $k$,
then $k$-points of $\ArcHilb_X$ correspond to
sections of $\Arc_X$
over closed subschemes in fibres of $\PP T_X$.
\begin{lem}\label{lem:archilb-affine}
Let $X\to S$ be a differentially smooth morphism with 
positive differential dimension of locally Noetherian
formal schemes. Then $\ArcHilb_{X/S}
\to \Hilb \PP T_{X/S}$ is a pro-finite type affine adic morphism from
a formal scheme.
\end{lem}
\begin{proof}
By Lemma~\ref{lem:arc-affine}, $\sU\times_{\PP T_{X/S}}\Arc_{X/S}
\to \sU$ is a Zariski-locally trivial $\AA^\infty$-bundle. On the
other hand, $\sU \to \Hilb \PP T_{X/S}$ is
a projective adic morphism of locally Noetherian formal schemes.
Identify $$\sU\times_{\PP T_{X/S}}\Arc_{X/S} =
\Spec_\sU \sA$$ where $\sA$ is a quasi-coherent sheaf
of $\OO_\sU$-algebras
locally isomorphic to the symmetric algebra of a direct sum
of countably many copies of $\OO_\sU$.
Then
$$\ArcHilb_{X/S}(T \xrightarrow{} \Hilb \PP T_{X/S})
= \Hom_{\sU_T-\mathrm{alg}}(\sA_T,\OO_{\sU_T}) $$
Suppose one can show that the functor $\underline\Hom_\sU(\sA,\OO_\sU)$
 sending
$T \to \Hilb \PP T_{X/S}$ to the set of $\sU_T$-module
morphisms $\sA_T \to \OO_{\sU_T}$ is representable
by a formal scheme, 
affine, adic and of pro-finite type
over $\Hilb \PP T_{X/S}$. Then $\ArcHilb_{X/S}$
is naturally a closed subfunctor of $\underline\Hom_\sU(\sA,\OO_{\sU})$,
hence satisfies the same properties, as desired. 
To simplify notation, the claim that remains to be proven is
reformulated as the 
following Lemma (replacing $\Hilb \PP T_{X/S}$ with $S$, $\sU$ with $X$
and $\sA$ with $\sE$).
\end{proof}
\begin{lem}
Suppose $X\to S$ is a flat projective adic morphism of locally Noetherian
formal schemes, and $\sE$ a locally countably generated quasi-coherent sheaf
on $X$. Then 
$\prod(\sE^\vee/X/S)$ is representable by a formal scheme, affine, adic
and of pro-finite type over $S$.
\end{lem}
\begin{proof}
Since the problem is local on the base, we can assume $S$ is Noetherian.
Then $X$ is quasi-compact, and we can filter $\sE$
by \emph{coherent} subsheaves $\sE_1 \subset \sE_2 \subset\dots$
so that $\bigcup \sE_i = \sE$.
By ~\cite[Thm. 7.7.6]{EGA-III-2}, there is a coherent sheaf\footnote{
While the result in~\cite{EGA-III-2} is stated over a base scheme,
it extends easily to the formal case.
Letting $\sI$ be an ideal of definition of $S$,
we have suitable coherent sheaves $\sQ_i^{(r)}$ on 
$S_r=\Spec_S \OO_S/\sI^{r+1}$.
By the universal property of the functor of sections,
there are isomorphisms $\sQ_i^{(r+1)}|_{S_r} \to \sQ_i^{(r)}$,
so that $\sQ_i=\varprojlim \sQ_i^{(r)}$ is a coherent sheaf
on $S$. Since its pullback to any scheme factors through one
of $\sQ^{(r)}_i$, it has the desired property.
} $\sQ_i$ on
$S$ such that
$$\prod(\sE_i^\vee/X/S) = \Spec_S \Sym_S \sQ_i$$
where the relative $\Spec$ 
construction over the formal scheme $S$
is understood as locally taking formal spectra
with the adic topology induced from $\OO_S$.
Now,
$\sE^\vee = \varprojlim \sE^\vee_i$, so that
$$\prod(\sE/X/S) =\varprojlim\Spec_S\Sym_S\sQ_i 
= \Spec_S\varinjlim\Sym_S\sQ_i$$
This is a formal scheme, affine, adic and of pro-finite type over $S$.
\end{proof}

\subsection{Families of arcs on a formal disc}

We would now like to describe families of pointed arcs 
on a formal disc, up to manageable equivalence.
Suppose $\hat X$ is a formal disc of positive dimension, with origin
$x$. The action of $\underline\Aut(\hat X,x)$ on $\hat X$
lifts to $\ArcHilb_{\hat X}$, so that we can consider
$\underline\Aut(\hat X,x)$ acting on the space of sections
of $\ArcHilb_{\hat X}$ over $\hat X$. The morphism
$\ArcHilb_{\hat X} \to \Hilb \PP T_{\hat X}$
induces a morphism on the spaces of sections and, upon
evaluation at $x$, a morphism
$$ \prod\left(\ArcHilb_{\hat X}/\hat X\right) \to \Hilb \PP T_{\hat X,x}. $$
\begin{lem}\label{lem:sec-archilb}
Let $\hat X$ be a formal disc of positive dimension, with origin $x$. Then
$ \prod\left(\ArcHilb_{\hat X}/\hat X\right) \to \Hilb \PP T_{\hat X,x}
$ is a pro-finite type affine morphism of schemes.
\end{lem}
\begin{proof}
By Lemma~\ref{lem:archilb-affine}, 
$\ArcHilb_{\hat X} \to \Hilb \PP T_{\hat X}$
is a pro-finite type affine adic morphism from a formal scheme,
and thus so is its composite with projection to $\hat X$.
We then have by Lemma~\ref{lem:sec-over-formal-disc} that
$\prod(\ArcHilb_{\hat X}/\hat X) \to \ArcHilb_{\hat X,x}$
is a pro-finite type affine morphism of schemes. Again
by Lemma~\ref{lem:archilb-affine}, the projection
$\ArcHilb_{\hat X,x} \to \Hilb\PP T_{\hat X,x}$
is a pro-finite type affine morphism of schemes.
Hence the composite
$\prod(\ArcHilb_{\hat X}/\hat X) \to \Hilb \PP T_{\hat X,x}$
is a pro-finite type affine morphism of schemes.
\end{proof}

The projection to $\Hilb\PP T_{X,x}$
 is preserved by the action of the pro-unipotent subgroup
$\sR_u\underline\Aut(\hat X,x)$ of automorphisms inducing
identity on $T_{\hat X,x}$. 
\begin{pro}\label{pro:closed-orbits}
Let $\hat X$ be a formal disc of positive dimension, with origin $x$.
Then the orbits of $k$-points under the
$\sR_u\underline\Aut(\hat X,x)$-action on
$\prod (\ArcHilb_{\hat X}/\hat X)$
are closed.
\end{pro}
\begin{proof}
As we can restrict to fibres over
$k$-points of $\Hilb\PP T_{\hat X,x}$,
the Proposition is an immediate
corollary of Lemma~\ref{lem:sec-archilb} and
the following pro-algebraic
version of Borel's fixed point theorem.
\end{proof}
\begin{lem}
Suppose a pro-unipotent group scheme acts on an affine pro-algebraic scheme.
Then the orbits of $k$-points are closed.
\end{lem}
\begin{proof}
Let $U$ be a pro-unipotent group scheme acting
on $\Spec A$. Since $U$ is a limit of extensions
by $\GG_a$, we have $U\simeq\AA^\infty$ as a scheme,
so that $k[U]\simeq k[u_1,u_2,\dots]$
is a polynomial algebra in countably infinitely
many variables. Let
$$ \rho:A \to A[u_1,u_2,\dots] $$
be the action morphism. 

We claim that
for any finite collection $a_1,\dots,a_n \in A$
there is a finitely generated subalgebra $A_0 \subset A$
containing $a_1,\dots,a_n$
such that $\rho$ restricts to an action morphism
$A_0 \to A_0[u_1,u_2,\dots]$. Indeed, let
$A_0$ be the subalgebra generated by the coefficients
of the polynomials $\rho(a_1),\dots,\rho(a_n) \in A[u_1,u_2,\dots]$
(these include in particular $a_1,\dots,a_n$).
We have a commutatitive diagram
$$\begin{CD}
k[a_1,\dots,a_n] @>{\rho}>> A_0[v_1,v_2,\dots] \\
@V{\rho}VV @VV{\rho\otimes 1}V \\
A_0[u_1,u_2,\dots] @>{1\otimes\mu}>> A[u_1,u_2,\dots][v_1,v_2,\dots]
\end{CD}$$
where $\mu : k[u_1,u_2,\dots] \to k[u_1,u_2,\dots][v_1,v_2,\dots]$
is the multiplication morphism in $U$. The bottom morphism
factors through $A_0[u_1,u_2,\dots][v_1,v_2,\dots]$, so that
for each $\rho(a_i)$, a polynomial in the
$v_1,v_2,\dots$, we have that $\rho$ sends its coefficients
into $A_0[u_1,u_2,\dots]$. Since these coefficients generate $A_0$,
the claim follows.

Now, we can filter $A$ by a sequence of finitely generated algebras 
$A_1 \subset A_2 \subset \dots$, 
so that $A = \bigcup A_i$, and the action of $U$ descends to 
each $A_i$. Recalling that $U$ is a limit of extensions
by $\GG_a$, we can assume that for each $n>0$,
$$ U \simeq \Spec k[u_1,u_2,\dots] \to \Spec k[u_1,\dots,u_n] $$
is naturally an epimorphism onto a unipotent algebraic group.
For each $i$, choose the largest $n>0$ such that
$u_n$ appears in the polynomials $\rho(A_i) \subset A_i[u_1,u_2,\dots]$,
and set $U_i = \Spec k[u_1,\dots,u_n]$. 
Then $U_i$ is naturally
a unipotent algebraic quotient of $U$, and
the action of $U$ on $\Spec A_i$
factors through $U_i$. Since $\Spec A = \varprojlim \Spec A_i$,
the action of $U$ on $\Spec A$ factors through $\varprojlim U_i$.
We assume without loss of generality that $U = \varprojlim U_i$.

Fix an orbit map $\varphi:U \to \Spec A$,
given by pullback of the action morphism $U\times\Spec A
\to\Spec A$ to a $k$-point.
For each $i$ consider the commutative diagram
$$\begin{CD}
U @>{\varphi}>> \Spec A \\
@V{\theta_i}VV @V{\pi_i}VV \\
U_i @>{\varphi_i}>> \Spec A_i
\end{CD}$$
where $\pi_i$ and $\theta_i$ are the natural projections,
and $\varphi_i$ the factorisation of $\pi_i\circ\varphi$.
By Borel's Fixed Point Theorem, $D_i=\varphi_i(U_i)$ is closed
in $\Spec A_i$. Let $D = \bigcap \pi_i^{-1}D_i$, a closed
subset of $\Spec A$. We then have that $\varphi$ factors through $D$,
and it remains to check surjectivity. Given $x \in D$,
consider
the pullback
$$\begin{CD}
V_i @>>> U \\
@VVV @V{\varphi_i\circ\theta_i}VV \\
\Spec\kappa(x) @>{\pi_i\circ x}>> \Spec A_i
\end{CD}$$
defining a nonempty closed subscheme $V_i \subset U \otimes \kappa(x)$.
Noting that
$V_1 \supset V_2 \supset\dots$, we 
let $V = \bigcap V_i$, a closed subset of $U\otimes\kappa(x)$.
Since $U\otimes\kappa(x)$ is affine, and thus quasi-compact,
$V$ is nonempty. Hence its image in $U$ is nonempty,
so that $x \in \varphi(U)$. 
\end{proof}

We remark that in the setting of Proposition~\ref{pro:closed-orbits}
the $U_i$ and $\Spec A_i$ of the former proof can be explicitly
constructed by filtering $\hat X$ and $\hat\PP^1$ by infinitesmial 
neighbourhoods
of their origins (cf. the proof of Lemma~\ref{lem:aut-formal-disc}):
this induces a natural presentation of the action of 
$\sR_u\underline\Aut(\hat X,x)$
on $\prod(\ArcHilb_{\hat X}/\hat X)$ as a limit of an inverse
system of actions of unipotent algebraic groups on affine algebraic
schemes.

%% file: chap2-2.tex
\section{Rational curves}\label{sec:ratcurves}
\subsection{Space of unparametrised pointed curves}

We now turn to the study rational curves
on a family of projective varieties. Let
$X \to S$ be a smooth morphism of Noetherian schemes.
The space of morphisms
$\underline\Hom_S(\PP^1_S,X)$
is a locally Noetherian $S$-scheme.
Relevant classes of its geometric points are defined as follows.
\begin{defn}
Let $\Spec L \to \underline\Hom_S(\PP^1_S,X)$ be a geometric point,
corresponding to a morphism $f:\PP^1_L \to X$.
Then
$f$ is called \emph{free} (resp. \emph{minimal}\footnote{
Called `standard' in~\cite{hwang-mok-rigidity}.
})
if $f^* T_{X/S}$ is a direct sum of
invertible sheaves of non-negative degree
(resp. a direct sum of $\OO(2)$ and
invertible sheaves of degree $0$ or $1$).
\end{defn}
There is an open subscheme
$$ \underline\Hom_{S,\bir}(\PP^1_S,X) \subset \underline\Hom_S(\PP^1_S,X) $$
whose geometric points correspond to morphisms
$\PP^1_L \to X$ birational to their images. We let
$\underline\Hom_{S,\bir}^n(\PP^1_S,X)$ be its normalisation.
The right action of $\underline\Aut(\PP^1)$ 
on $\underline\Hom_S(\PP^1_S,X)$
restricts to $\underline\Hom_{S,\bir}(\PP^1_S,X)$ and lifts
to $\underline\Hom_{S,\bir}^n(\PP^1_S,X)$.
The quotient morphisms
$$ 
\underline\Hom_{S,\bir}^n(\PP^1_S,X) \to
\underline\Hom_{S,\bir}^n(\PP^1_S,X) / \underline\Aut(\PP^1,0) $$
and
$$ 
\underline\Hom_{S,\bir}^n(\PP^1_S,X) \to
\underline\Hom_{S,\bir}^n(\PP^1_S,X) / \underline\Aut(\PP^1) $$
are principal bundles~\cite{kollar}. 

\begin{defn}
Given a smooth morphism $X \to S$ of Noetherian schemes,
a \emph{family of rational curves} on $X/S$ is a closed
$\underline\Aut(\PP^1)$-invariant subscheme of
$\underline\Hom_{S,\bir}^n$.
\end{defn}
Note that the definition does not assume irreducibility.
However, most families we consider will be in fact  
irreducible components of $\underline\Hom_{S,\bir}^n(\PP^1,X)$.
Given a family
$\sM \subset \underline\Hom_{S,\bir}^n(\PP^1_S,X)$,
we introduce the following notation for the quotients:
$$ \sM_0 = \sM / \underline\Aut(\PP^1) $$
$$ \sM_1 = \sM / \underline\Aut(\PP^1,0) $$
where $\sM \to \sM_0$, resp. $\sM \to \sM_1$,
is an $\underline\Aut(\PP^1)$-principal bundle,
resp. an $\underline\Aut(\PP^1,0)$-principal bundle.
It follows that $\sM_1 \to \sM_0$ is a $\PP^1$-bundle.
We also consider the associated $\PP^1$-bundle
$$ \sM_2 = \sM \times^{\underline\Aut(\PP^1,0)} \PP^1. $$ 
Let $\ev : \sM \times \PP^1 \to X$ be the evaluation morphism.
The composite $\ev \circ 0_\sM$ descends to a structure map
$\sM_1 \to X$. Its composite with the projection $\sM_2 \to \sM_1$
yields a morphism $\sM_2\to X$, which we consider as the \emph{left}
structure map. On the other hand, $\ev$ descends to a morphism
$\sM_2\to X$, which we consider as the \emph{right} structure map.
We thus have a double fibration $\sM_2 \rightrightarrows X$,
and think of $\sM_2$ as the space of unparametrised $2$-pointed
$\sM$-curves. Completion of $\sM_2$ along the $0$-section
$\sM_1\to \sM_2$ gives a bundle of formal discs $\hat\sM_2 \to \sM_1$. 
We also define the products
$$ \sM_2^i = (X \backslash \sM_2 / X)^i $$
together with morphisms $\sM_2^i \rightrightarrows X$ as
introduced in~\ref{sec:notation},
and think of them as spaces of $2$-pointed $i$-chains
of $\sM$-curves. 
Since $\sM \to \sM_0$ is a principal $\underline\Aut(\PP^1)$-bundle,
we have natural isomorphisms
$$ \sM_1 \simeq \sM\times^{\underline\Aut(\PP^1)}\PP^1,\quad
\sM_2 \simeq
\sM\times^{\underline\Aut(\PP^1)}(\PP^1 \times \PP^1). $$
The transposition
on $\PP^1\times\PP^1$
induces an involution
$\sM_2 \to \sM_2$
over $\sM_0$, swapping the two marked points (i.e. the two structure maps to $X$).

There are open subschemes
$\sM^\free$, $\sM^\min$, $\sM^\arc$ of $\sM$
such that the curves corresponding to their geometric points are 
free, resp. minimal, resp. unramified at $0$. Since $\sM^\free$
and $\sM^\min$ are $\underline\Aut(\PP^1)$-equivariant, they
descend to open subschemes $\sM_1^\free$, $\sM_0^\free$
and $\sM_1^\min$, $\sM_0^\min$. Since $\sM^\arc$ is
$\underline\Aut(\PP^1,0)$-invariant, it descends to an open subscheme
$\sM_1^\arc$. The inclusion $\hat\PP^1 \to \PP^1$ induces morphisms
$$ \sM_1^\arc \to \Arc_{X/S} \to \PP T_{X/S} $$
Finally, we let
$$ \sM_2^\free = \sM_1^\free \times_{\sM_1}\sM_2,\quad\quad
\sM_2^{i,\free} = (X\backslash\sM_2^\free/X)^i. $$
We remark that if $\sM$ is an irreducible component of $\underline\Hom_\bir^n(\PP^1,X)$,
then $\sM_1^\free \to X$ and both $\sM_2^\free \rightrightarrows X$ are smooth~\cite{kollar}.

\subsection{Rational curves on a variety}
Suppose now $X$ is an irreducible variety,
and $\sM \subset \underline\Hom_{\bir}^n(\PP^1,X)$ a
family of rational curves.
We will say that $\sM$ is \emph{dominating} if $\sM_1 \to X$ is dominant.
We will say that $\sM$ is \emph{unsplit} if $\sM_0$ is proper.
We will say that $X$ is \emph{chain-connected} by $\sM$-curves
if the induced morphism $\sM_2^i \to X\times X$
is dominant for some $i\ge0$.
In this case
$$ \coprod\sM_2^i \rightrightarrows X $$
is a transitive category-scheme, with composition
given by concatenation of chains. Given an object
over $X$ together with an isomorphism of its two pullbacks
along $\sM_2\rightrightarrows X$, 
we obtain an action of $\coprod\sM_2^i$,
i.e. a `parallel transport' along chains. We will use this
technique in Chapter~\ref{chap:extension} to extend
$\infty$-jets of morphisms between varieties chain-connected
by suitable families of rational curves. A simpler application
is the following Proposition.
\begin{lem}\label{lem:trivial-along-curves}
Let $X$ be a nonsingular variety, $\sL$ an 
invertible sheaf on $X$, and $\sM \subset\underline\Hom_\bir^n(\PP^1,X)$
a family of rational curves such that the pullback
of $\sL$ by the generic $\sM$-curve 
$\PP^1 \otimes k(\sM) \to X$
is trivial. Then the two pullbacks of $\sL$ along
$\sM_2 \rightrightarrows X$ are isomorphic.
\end{lem}
\begin{proof}
Considering the $\underline\Aut(\PP^1,0)$-equivariant diagram
$$ \sM \overset{p_1}{\underset{0_\sM}{\leftrightarrows}}
 \sM\times\PP^1 \xrightarrow{\ev} X $$
we have that $\ev^*\sL \simeq p_1^*\sK$ for some 
$\underline\Aut(\PP^1,0)$-equivariant
invertible
sheaf $\sK$ on $\sM$. There is then an
$\underline\Aut(\PP^1,0)$-equivariant
isomorphism
$$ \ev^*\sL \simeq p_1^*\sK =(0_\sM\circ p_1)^*p_1^* \sK
\simeq (0_\sM\circ p_1)^* \ev^* \sL = (\ev\circ 0_\sM\circ p_1)^*
\sL $$
descending to
$e_1^* \sL \simeq e_0^*\sL$
where $e_0, e_1 : \sM_2=\sM \times^{\underline\Aut(\PP^1,0)}\PP^1 \to X$
are the left and right morphisms to $X$.
\end{proof}

\begin{pro}\label{pro:picard-number}
Suppose $X$ is a nonsingular projective variety,
and $\sM \subset \underline\Hom_\bir^n(\PP^1,X)$
a connected unsplit family of rational curves, 
such that $X$ is chain-connected by $\sM$-curves.
Then $X$ has Picard number $1$.
\end{pro}
\begin{proof}
We are going to show that if a divisor $D$ on $X$ intersects
some (hence every) $\sM$-curve trivially, then it is numerically trivial.
Indeed, we have by
Lemma~\ref{lem:trivial-along-curves} that there is an isomorphism
of the two pullbacks of $\OO(D)$ along $\sM_2\rightrightarrows X$.
It then follows that there is an isomorphism of
the pullbacks of $\OO(D)$ along $\sM_2^i\rightrightarrows X$
for each $i\ge0$. 
Let $i$ be such that 
$\sM_2^i \to X\times X$
is dominant, hence surjective. Fix a closed point $x \in X$.
Then the right evaluation morphism $e:x\times_X\sM_2^i \to X$
is surjective, and $e^*\OO(D)$ is trivial. Since the
components of $\underline\Hom(\PP^1,X)$
are quasi-projective, it follows by properness that $\sM_1$
is projective. Now, for any irreducible curve $C \subset X$,
there is an irreducible curve $\tilde C \subset x\times_X\sM_2^i$
surjecting onto $C$. Since $\deg_{\tilde C} e^*\OO(D)=0$,
we have $(\deg e|_{\tilde C}) (C.D) = 0$,
so that $C.D=0$ as desired.
\end{proof}

We say that an irreducible family of rational curves $\sM$ on $X$
is of degree $d$ with respect to
an invertible sheaf $\sL$ if the pullback of $\sL$
by the generic $\sM$-curve $\PP^1\otimes {k(\sM)} \to X$
has degree $d$. We say that $\sM$ is of minimal
degree with respect to $\sL$ if there is no
component of $\underline\Hom_\bir^n(\PP^1,X)$
of lower degree with respect to $\sL$.
Unsplit families arise from curves of minimal degree
with respect to an ample invertible sheaf (essentially
unique, \emph{a posteriori}, in case of a chain-connected
nonsingular variety).
\begin{lem}\label{lem:unsplit}
Suppose $X$ is a nonsingular projective variety, $\sL$ an ample invertible
sheaf, and $\sM \subset \underline\Hom_\bir^n(\PP^1,X)$
an irreducible family of rational curves of minimal degree with respect to $\sL$.
Then $\sM$ is unsplit.
\end{lem}
\begin{proof}
We use the valuative criterion of properness.
Let $T$ be a spectrum of a discrete valuation ring,
with closed point $t_0$ and generic point $t_1$,
and let $f_1 : t_1 \to \sM_0$ be a morphism. The image
of $t_1 \times_{\sM_0}\sM_1 \to t_1 \times X$ defines
a morphism $t_1 \to \Hilb X$, extending by properness
of the Hilbert scheme to $\tilde f : T \to \Hilb X$.
Let $C \subset T\times X$ be the pullback of the universal family,
so that $p_1:C \to T$ is a flat family whose fibres are rational
cycles of dimension one. Write $[ t_0\times_T C ] = \sum_{i=1}^r a_i [C_i]$
where $C_i$ are integral rational curves. Since $\sM$ has minimal
degree with respect to $\sL$, 
we have
$$\deg_{C_j}p_2^*\sL \ge \sum_{i=1}^r a_i \deg_{C_i}p_2^*\sL $$
for each $j$. Hence $r=1$, $a_1=1$, and $t_0\times_T C$
is an integral rational curve. It follows that the normalization
$\nu:\tilde C \to C$  is a $\PP^1$-bundle over $T$. 
After unramified base change $T'\to T$,
we have $T'\times_T\tilde C \simeq T' \times \PP^1$.
Then $$\id_{T'}\times (p_2 \circ \nu) : T'\times_T\tilde C \simeq 
T\times\PP^1 \to T'\times X$$ induces
a morphism $T' \to \underline\Hom_{\bir}^n(\PP^1,X)$, necessarily
factoring through $\sM$. Its composite with $\sM \to \sM_0$
gives $$f':T' \to \sM_0.$$
Letting $q_1,q_2 : T'\times_TT' \to T'$ be the two projections,
consider $$\delta f' = \langle q_1^*f',q_2^*f'\rangle : T'\times_TT' \to \sM_0\times\sM_0.$$
Observing that $f'|_{t_1\times_TT'}$ is the pullback of $f_1$,
we have that $\delta f'|_{t_1\times_T(T'\times_TT')}$ factors through
the diagonal $\sM_0 \to \sM_0\times\sM_0$. But then, by separatedness
of $\sM_0$, so does entire $\delta f'$. Hence $f'$ descends along
the \'etale surjection $T'\to T$ to $f:T\to \sM_0$
such that $f|_{t_1}=f_1$.
\end{proof}

\begin{lem}\label{lem:free-intersection}
Let $X$ be a nonsingular projective variety
and $\sM \subset \underline\Hom_\bir^n(\PP^1,X)$ an
irreducible component such that the generic $\sM$-curve
$f : \PP^1\otimes k(\sM) \to X$ is free. Let $D \subset X$
be a reduced closed subscheme of codimension $1$, 
and $W \subset X$ a closed subscheme of codimension
$2$. Then $f^*D$ is reduced, and $f^*W$ is empty.
\end{lem}
\begin{proof}
By \cite[Cor. 3.5.4]{kollar}, the evaluation
morphism $\ev:\sM^\free\times\PP^1 \to X$ is smooth.
\end{proof}

We end this subsection showing that one can
often restrict to chains of free curves.\footnote{I have learned this
argument from Jason Starr.}
\begin{lem}\label{lem:free-chains}
Suppose $X$ is a nonsingular projective variety of Picard
number $1$, and
$\sM\subset\underline\Hom_{\bir}^n(\PP^1,X)$ an irreducible component
such that the generic $\sM$-curve is free.
Then
$\sM_2^{i,\free}\to X\times X$ is dominant for some $i\ge0$.
\end{lem}
\begin{proof}
Let $\sM_2^{i,\free} = \bigcup M_j$ be the irreducible
components. For each $j$, let $W_j \subset X\times X$
be the closed image subscheme of $M_j$ under
$\sM_2 \to X\times X$. Choose $j_0$ such that
$W_{j_0}$ has maximal dimension among all $M_j$.
Set $M=M_{j_0}$, $W=W_{j_0}$, with the pair of projections
$W \rightrightarrows X$.

Let $x$ be the generic point of $X$ and set $W_x = x\times_XW$.
Let $\eta$
be the generic point of $W_x$. By construction, 
$\eta\times_X \sM_2^\free \to X$
factors through $W_x$ (for otherwise $\dim W_j > \dim W$ for some $j$,
a contradiction). 
Since by freeness $X\leftarrow \sM_2^\free$
is smooth, we have that
$W_x\times_X \sM_2^\free$ is the closure of ${\eta\times_X\sM_2^\free}$
so that
$ W_x \times_X \sM_2^\free \to X $
factors through $W_x$ as well. 

As a closed subscheme of
$X \otimes\kappa(x)$,
$W_x$ defines an $x$-point of $\Hilb X$ which,
by properness of the Hilbert scheme, extends to
$$ q : U \to \Hilb X $$
over an open subscheme $U\subset X$ whose complement
has codimension at least $2$ in $X$.
It will be enough to show that $q$ is constant.
Note that it is constant at least on the fibres of $X \leftarrow W$,
in particular $q|_{W_x}$ factors through $q|_x$. Letting
$f:\PP^1 \otimes k(\sM) \to X$ be the generic $\sM$-curve,
we have by freeness that $f(0) = x$, so that by the previous
paragraph $f$ factors through $W_x$. On the other hand,
since $X\setminus U$ has codimension at least $2$, 
$f$ factors through $U$ by Lemma \ref{lem:free-intersection}. 
It then follows that for
any ample invertible sheaf $\sL$ on $\Hilb X$, the pullback
$f^* q^* \sL$ is trivial. Since $X$ has Picard number $1$,
it follows that $\sL$ is numerically trivial, so that
$q$ must be a constant morphism. But $W_x \to X$ is dominant,
so that $q$ must factor through the $k$-point of $\Hilb X$
corresponding to entire $X$.
\end{proof}

\begin{lem}\label{lem:free-stuff}
Suppose $X$ is a nonsingular projective variety,
and $\sM \subset \underline\Hom_\bir^n(\PP^1,X)$ an
irreducible component
such that the generic $\sM$-curve is free,
and
the generic fibre of $\sM_1^\free \to X$
is geometrically connected.
Then:
\begin{enumerate}
\item $X\leftarrow {\sM_2^{i,\free}}$ has geometrically integral generic fibre,
\item $\sM_2^{i+1,\free} \to \sM_2^{i,\free}$ is a smooth surjection,
\item if the generic $\sM$-curve is minimal (resp. unramified at $0$),
then the generic point of $\sM_2^{i,\free}$ is a chain
of minimal (resp. unramified at $0$) curves
\end{enumerate}
for all $i\ge0$.
\end{lem}
\begin{proof}
Note that by freeness $X\leftarrow \sM_1^\free$ is smooth,
so that, being geometrically connected,
its generic fibre is geometrically integral.
The same holds for
$X \leftarrow \sM_2^\free$, since $\sM_2^\free \to \sM_1^\free$ is a $\PP^1$-bundle.
Assume by induction that $X \leftarrow \sM_2^{i,\free}$ has geometrically
integral generic fibre.
By freeness, $\sM_2^{i,\free} \to X$ is dominant.
On the other hand,
$X \leftarrow \sM_2^{\free}$ is smooth with geometrically connected generic fibre.
Hence
$$ \sM_2^{i+1,\free} = \sM_2^{i,\free}\times_X\sM_2^\free$$
has geometrically integral generic fibre, and the projection
$ \sM_2^{i+1,\free} \to \sM_2^{i,\free}$
is dominant and smooth. To check surjectivity, it is enough to note
that it has an obvious section, duplicating the rightmost link in the $i$-chain.

For the last statement, it is enough to note that projection
 $\sM_2^{i,\free} \to \sM_2$ to the rightmost factor is dominant for all $i\ge 0$:
indeed the generic point of $\sM_2$ lifts to a point in $\sM_2^{i,\free}$
corresponding to a chain consisting of $i$ copies of the single generic
$\sM$-curve.
\end{proof}

\subsection{The tangent map and VMRT}
Recall that given a nonsingular projective
variety $X$ and a family $\sM$ of rational
curves on $X$,
we have morphisms
$$ \sM_1^\arc \to \Arc_X \to \PP T_X. $$
The composite will be called the \emph{tangent map}.
The results of Kebekus~\cite[Thm. 3.3 and 3.4]{kebekus} describe\footnote{
There are concerns about the proof in~\cite{kebekus} when
the degree of the family with respect to the ample invertible sheaf
is divisible by $p$. This is clearly not the case here.
}
the
tangent map at the generic point of $X$.
\begin{pro}[Kebekus]
\label{pro:kebekus}
Let $X$ be a nonsingular projective variety with
 generic point $x$,
$\sL$ an ample invertible sheaf on $X$, and
$\sM$ a dominating irreducible family of rational curves on $X$ of
degree $1$ with respect to $\sL$ (hence unsplit by Lemma~\ref{lem:unsplit}). Then:\begin{enumerate}
\item $x^*\sM_1^\arc = x^*\sM_1$
\item $x^*\sM_1 \to \PP T_{X,x}$ is finite.
\end{enumerate}
\end{pro}
Under the hypotheses of Proposition~\ref{pro:kebekus},
we call the closed image scheme of $x^*\sM_1 \to \PP T_{X,x}$ the
 \emph{variety of $\sM$-rational tangents at x}.
Denote by $m_1$ the generic point of $\sM_1$.
If the generic $\sM$-curve is minimal, then
$x^*\sM_1 \to \PP T_{X,x}$ unramified at $m_1$,
and induces a well-defined morphism (cf.~\cite{hwang-mok-rigidity})
$$ \PP T_{\sM_1/X,m_1} \to \PP \Lambda^2 T_{X,x}. $$
The remainder of this subsection will be occupied by
a proof of the following Proposition,
an analogue of~\cite[Prop. 13]{hwang-mok-rigidity}.
\begin{pro}\label{pro:linear-nondegeneracy}
Assume $\Char k = p>0$. Let $X$ be a nonsingular projective Fano variety of Picard
number $1$ and with $\ind(X) < p$.
Let $\sM \subset \underline\Hom_\bir^n(\PP^1,X)$ an
irreducible component of degree $1$.
Denote by $x$ the generic point of $X$,
and by $m_1$ the generic point of $\sM_1$.
Assume that:
\begin{enumerate}
\item The generic $\sM$-curve is minimal.
\item The generic fibre of $\sM_1 \to X$ is geometrically
irreducible.
\item There is $i\ge0$ such that
$\sM_2^{i,\free}$
contains a subscheme separably dominating $X\times X$.
\item 
Letting 
$\sD_x \subset T_{X,x}$  
be the linear span of the variety of $\sM$-rational tangents
at $x$,
the image of the natural morphism
$$\PP T_{\sM_1/X,m_1} \to \kappa(m_1)\otimes_{\kappa(x)} \PP \Lambda^2 T_{X,x}$$
spans
$\kappa(m_1)\otimes_{\kappa(x)}\Lambda^2\sD_x$.
\end{enumerate}
Then the variety of $\sM$-rational tangents
at $x$ is linearly nondegenerate in $\PP T_{X,x}$.
\end{pro}

\begin{proof}
Note that Proposition
\ref{pro:kebekus} applies, so that hypothesis 4 makes sense.
We extend $\sD_x \subset T_{X,x}$
to a saturated subsheaf $\sD \subset T_X$,
a sub-bundle away from codimension $2$. 
Note that $\rk \sD >0$.
The argument relies on the existence of
a height one purely inseparable quotient of $X$
associated with $\sD$.
The proof of integrability is essentially due to Hwang and Mok.
In characteristic $p$ we need $p$-closedness as well (see
\cite{shen} for a related argument).
\begin{lem}\label{lem:d-integrable}
$\sD$ is integrable and $p$-closed.
\end{lem}
\begin{proof}
By \cite[Lemma 4.2]{ekedahl}, we need to check that the maps
$$ \theta : \Lambda^2 \sD \to T_X / \sD,\quad \theta(\xi_1\wedge\xi_2) = [\xi_1,\xi_2]+\sD$$
$$ \phi : F^*\sD \to T_X/\sD,\quad \phi(1\otimes\xi) = \xi^p + \sD $$
are 
zero, where $F : X\to X$ is the absolute Frobenius morphism,
$\theta$ is $\OO_X$-linear, and $\phi$ is $\OO_X$-linear if $\theta=0$.
Let $\bar m_1$ be the geometric generic point of $\sM_1$, and
$\bar x$ the geometric generic point of $X$.
The rational curve
$$f:\PP^1_{\bar m_1} \simeq \bar m_1 \times_{\sM_1}\sM_2 \to X$$ 
is minimal and unramified at $0$, factors through the locus over which
$\sD \subset T_X$ is a sub-bundle, and sends $0$ to $\bar x$.
Fix a splitting
$$ f^*T_X = \OO(2) \oplus \OO(1)^{\ind(X)-2}\oplus \OO^{\dim(X)+1-\ind(X)} $$
and a nonzero vector
$
u \in \kappa(\bar m_1) \otimes_{\kappa(x)} T_{X,x}$
in the image of $df|_0$. 
Note that $\sM_2^\free \times_X{\bar x}$ is nonsingular,
so that given a nonzero vector
$$v \in \kappa(\bar m_1) \otimes_{\kappa(x)} T_{X,x} \simeq f|_0^*T_X$$ 
contained in the $\OO(1)^{\ind(X)-2}$ summand,
one can find
a morphism
$$ \gamma_v : \Spf \kappa(\bar m_1\times \bar x)[[t]]
\to \sM_2^\free \times_X \bar x $$
over $\bar x$ such that
$$ (p\circ \gamma_v)(0) = \bar m_1 \times \bar x,\quad
d(e_1\gamma_v)|_0(\frac{\partial}{\partial t}) = v \otimes 1 $$
where $p:\sM_2^\free \times_X \bar x \to \sM_1^\free \times \bar x$ 
is the natural projection,
and $e_1:\sM_2^\free \times_X \bar x \to X \times \bar x$ 
is the \emph{left} evaluation
map. That is, $\gamma_v$ is a deformation of a $2$-pointed
rational curve,
fixing the second marked point, with $f \otimes 1$ 
as the central curve, and $v \otimes 1$
as the tangent vector to the corresponding deformation of the first marked point.
Considering $\pr_1 \circ p\circ \gamma_v : \Spf \kappa(\bar m_1\times\bar x)[[t]]
\to \sM_1$, 
we have a formal family of parametrised rational curves
$$ \delta_v : \Spf \kappa(\bar m_1\times\bar x)[[t]] \times_{\sM_1}\sM_2 \to X $$
such that $d\delta_v$ factors through $\delta_v^*\sD$,
and $d\delta_v|_{(0,0)}$ is an isomorphism onto
the base-change of the span of $u$ and $v$ in $\kappa(\bar m_1)\otimes_{\kappa(x)} T_{X,x}$. 
In particular, restricting to $\hat\sM_2 \subset \sM_2$ gives an unramified morphism
from a formal disc of dimension two:
$$ \Spf \kappa(\bar m_1\times\bar x)[[t]] \times_{\sM_1}\hat\sM_2 \to X, $$
tangent to $\sD$ everywhere, and to the span of $u$ and $v$ at the origin.
It follows that
$$ (\kappa(\bar m_1) \otimes_{\kappa(x)}\theta_x)(u\wedge v) = 0. $$
But then, since $v$ was arbitrary,
 $\theta_x$ vanishes on any element in the linear span of the
image of $\PP T_{\sM_1/X,m_1} \to \PP \Lambda^2 T_{X,x}$. Hence,
by hypothesis 4 of the Proposition, $\theta_x=0$,
and finally $\theta=0$ by construction of $\sD$.

We proceed to show vanishing of $\phi$, which we now know to be
$\OO_X$-linear. Consider an unramified morphism
$ h : \bar m_1 \times \hat W \to X $
from the base-change of a formal disc $\hat W$ of dimension $\dim(X) - 1$,
such that the image of $dh$ at the origin is transverse
to $f$. By smoothness of $\sM_1 \to X$ at the generic point,
there is a lift $\tilde h : \bar m_1 \times \hat W \to \sM_1$
of $h$, and a formal family of parametrised rational curves
$$ g : \bar m_1 \times \hat W \times \PP^1 \simeq 
\bar m_1\times \hat W \times_{\sM_1}\sM_2 \to X $$
whith $f$ as the central fibre. Restricting to $\hat \PP^1 \subset \PP^1$,
we obtain an unramified morphism from the base-change of a formal disc
of dimension $\dim(X)$
$$ \hat g : \bar m_1 \times \hat W \times \hat \PP^1 \to X. $$
Identifying $\hat\PP^1 = \Spf k[[t]]$, we have 
$$ (\hat g^*\phi) (d\hat g (1\otimes 1\otimes \frac{\partial}{\partial t})) = 0. $$
But
the smallest subspace of $T_{X,x}$ whose pullback
by $\hat g$ contains 
$d\hat g(1\otimes 1\otimes \frac{\partial}{\partial t})$
is precisely $\sD_x$. Hence $\phi_x=0$, and finally $\phi=0$ by construction
of $\sD$.
\end{proof}

It follows that $\sD$ defines a height one purely
inseparable morphism $\pi : X \to Y$ to a normal variety, 
flat away from codimension two, and factoring
the geometric Frobenius $F_X : X \to X'$.
There is an exact sequence (cf.~\cite{ekedahl})
\begin{equation}
\label{eq:ekedahl}
 0 \to \sD \to T_X \to \pi^* T_Y \xrightarrow{\pi^*\delta}
 \pi^*\sigma^* \sD \to 0 
\end{equation}
where $\sigma:Y \to X$ is a composite of
the natural morphism $Y \to X'$ factoring $F_X$
with the projection $X'\to X$,
and $\delta:T_Y \to \sigma^*\sD$ is an $\OO_Y$-module morphism.
In particular, $\sigma\circ\pi$ is the absolute Frobenius $X \to X$.
This leads to an equality
\begin{equation}
\label{eq:pc1d}
 (1-p)c_1(\sD) - \ind(X) + c_1(\pi^*T_Y) = 0.
\end{equation}
Let $f : \PP^1 \otimes k(\sM) \to X$ be the generic $\sM$-curve.
By definition of $\sD$, $\pi\circ f$ is everywhere ramified,
so that we have a commutative diagram
$$\begin{CD}
\PP^1\otimes k(\sM) @>{f}>> X \\
@V{F_{\PP^1}\otimes 1}VV @VV{\pi}V \\
\PP^1\otimes k(\sM) @>{\bar f}>> Y
\end{CD}$$
where $F_{\PP^1}:\PP^1\to\PP^1$ is the geometric Frobenius.
\begin{lem}\label{lem:c1ty}
$c_1(\pi^*T_Y) \ge p$.
\end{lem}
\begin{proof}
It will be enough
to check that $H^0(\PP^1,f^*\pi^*\omega_Y)=0$.
Using $\omega_X = \pi^!\omega_Y$ (cf.~\cite{hartshorne}) and the fact that
$\PP^1$ is Frobenius-split, we have
\begin{align*}
0 & = H^0(\PP^1,f^*\omega_X) = H^0(\PP^1,f^*\pi^!\omega_Y)=H^0(\PP^1,f^*\sHom_Y(\pi_*\OO_X,\omega_Y))
\\ & = \Hom_{\PP^1}(F_{\PP^1*}\OO_{\PP^1},\bar f^*\omega_Y) \supset
H^0(\PP^1,\bar f^*\omega_Y) = H^0(\PP^1,f^*\pi^*\omega_Y).\qedhere
\end{align*}
\end{proof}
By \eqref{eq:pc1d}, Lemma~\ref{lem:c1ty} and the hypothesis on the index of $X$,
it follows that
\begin{equation}\label{eq:c1d}
c_1(\sD) = \frac{c_1(\pi^*T_Y) - \ind(X)}{p-1} >0. 
\end{equation}
The sequence \eqref{eq:ekedahl} restricts to 
$$
0 \to f^*\sD \to f^* T_X \to F_{\PP^1}^* \bar f^* T_Y 
\xrightarrow{F_{\PP^1}^* \bar f^*\delta} F_{\PP^1}^* f^*\sD \to 0
$$
where exactnes on the left follows from the fact that the generic 
$\sM$-curve, being free, factors through the locus over which $\sD$ is 
a sub-bundle of $T_X$. 
In particular, we have a short exact sequence
\begin{equation}\label{eq:d-ses}
0 \to f^*\sD \to f^* T_X \to F_{\PP^1}^* \ker \bar f^* \delta
\to 0 \end{equation}
Since $f$ is by hypothesis minimal, there is a splitting
\begin{equation}\label{eq:splitting}
 f^*T_X = \overbrace{\OO(2) \oplus \OO(1)^{\ind(X)-2}}^{\sT^+} 
\oplus \overbrace{\OO^{\dim(X) + 1 - \ind(X)}}^{\sT^0} 
\end{equation}
By definition of $\sD$, we have $\sT_+ \subset f^*\sD$.
Identifying $f^*T_X / \sT^+$ with $\sT^0$, let
$$\sD^0 = f^*\sD / \sT^+ \subset f^*T_X/\sT^+ = \sT^0, $$
a sub-bundle of the trivial bundle $\sT^0$ over $\PP^1 \otimes k(\sM)$.
There is a splitting $ \sD^0 = \bigoplus \OO(d_i)$
where all $d_i \le 0$ since $\sD^0$ is a sub-bundle of the trivial
bundle $\sT^0$.
On the other hand, \eqref{eq:d-ses} yields
a short exact sequence
$$
0 \to \sD^0 \to \sT^0 \to F_{\PP^1}^* \ker\bar f^*\delta \to 0
$$
so that
$$ c_1(\sD^0) = -p c_1(\ker \bar f^*\delta) $$
and the inequality \eqref{eq:c1d} gives
$$ 0 < c_1(f^*\sD) = \ind(X) - p c_1(\ker\bar f^*\delta). $$
Since $\ind(X) < p$, it follows that $c_1(\ker\bar f^* \delta) \le 0$
and thus $c_1(\sD^0)=\sum d_i \ge 0$. 
But $d_i\le0$, so that necessarily $d_i=0$ for all $i$, i.e.
$\sD^0$ is a trivial bundle.
We can then assume that
\begin{equation}
\label{eq:d-final}
f^*\sD = \overbrace{\OO(2) \oplus \OO(1)^{\ind(X)-2}}^{\sT^+}
\oplus \overbrace{\OO^{\rk(\sD)+1-\ind(X)}}^{\sD^0}
\end{equation}
compatibly with the decomposition \eqref{eq:splitting}.
This leads to the following property, which will allow us to use induction
on generic chains.
\begin{lem}\label{lem:d-m2}
Let $m_2$ be the generic point of $\sM_2$,
and $e_1,e_2 : m_2 \to X$ the two projections. Then
$$ de_1^{-1}(e_1^*\sD) = de_2^{-1}(e_2^*\sD) $$
as subspaces of $T_{\sM_2,m_2}$.
\end{lem}
\begin{proof}
Immediate by \eqref{eq:d-final}.
\end{proof}
Let $m_2^i$ be the generic
point of $\sM_2^{i,\free}$, and $e_1^i,e_2^i:m_2^i \to X$
the two projections. 
By hypotheses 1--2 of the Proposition, Lemma
\ref{lem:free-stuff} applies, so that
$$ m_2^{i+1} \in m_2^i \times_X m_2. $$ 
Then Lemma~\ref{lem:d-m2} and induction on $i$
give
$$ (de_1^i)^{-1}(e_1^{i*}\sD) = (de_2^i)^{-1}(e_2^{i*}\sD) $$
on $m_2^i$ for all $i\ge0$. The equality extends to $\sM_2^{i,\free}$,
so that
given a subscheme $W \subset \sM_2^{i,\free}$ dominating $X \times X$,
we have that $W \to X\times X$ is separable
only if $\sD = T_X$.
This concludes the proof of Proposition~\ref{pro:linear-nondegeneracy}.
\end{proof}

%% file: chap2-3.tex
\section{Cominuscule homogeneous varieties}
\begin{figure}
\vspace{48pt}
\begin{center}\includegraphics{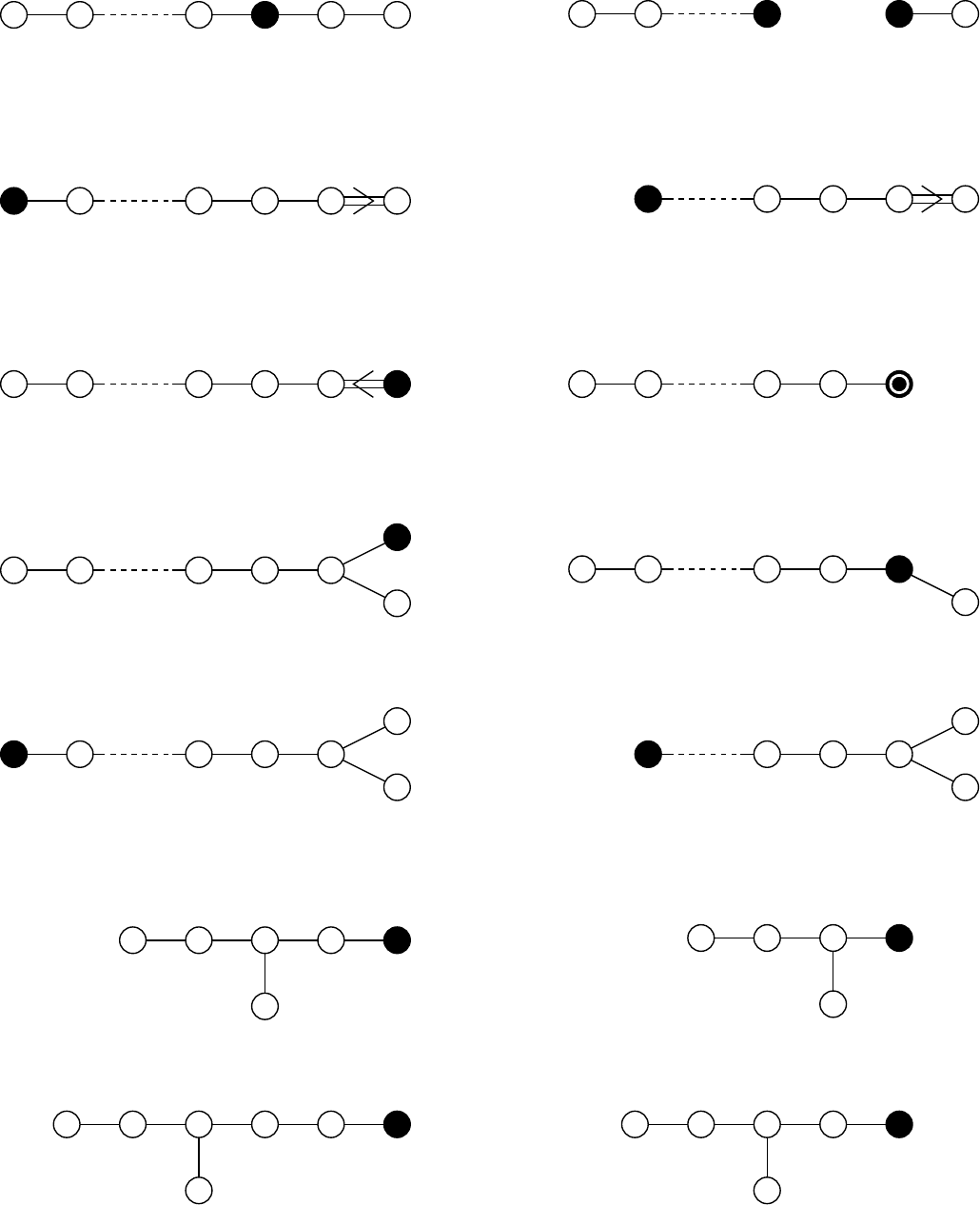}\end{center}
\vspace{48pt}
\caption{Marked Dynkin diagrams corresponding to
cominuscule homogeneous varieties (left) and their varieties
of minimal rational tangents (right). From top to bottom:
Grassmannian, odd-dimensional quadric, symplectic Grassmannian,
orthogonal Grassmiannian, even-dimensional quadric, Cayley plane,
Freudenthal variety.
}
\label{fig:diagrams}
\end{figure}

\subsection{Classification and properties}
We will devote the last section of this chapter
to a review of properties of the class of homogeneous varieties
to be considered in Chapter~\ref{chap:rigidity}. 
It is, with exception of one sub-class, closed under
the operation of taking the variety of minimal degree rational
tangents through a point. This forms the basis of an inductive
argument in the proof of the Rigidity Theorem.
Recall that we work over the algebraically
closed field $k$.
Most of the material here is standard.
\begin{lem}[cf. \cite{rrs}, Lemma 2.2]	\label{lem:cominuscule}
Let $X = G/P$, where $G$ is a connected, simply connected simple algebraic group,
and
$P$ a maximal reduced parabolic subgroup. Then the following are equivalent:
\begin{enumerate}
\item The unipotent radical $\sR_uP$ is abelian.
\item For a suitable choice of a maximal torus and Borel subgroup,
$P=P_\alpha$ is a standard maximal reduced parabolic
associated with a simple root $\alpha$ occuring with coefficient $1$ in the
simple root decomposition of the highest positive root.
\end{enumerate}
\end{lem}

\begin{defn}
A homogeneous variety $X$ satisfying either of the equivalent hypotheses
in Lemma~\ref{lem:cominuscule} is called 
\emph{cominuscule}.
\end{defn}

In the following, we will always assume that a cominuscule
variety is presented as $X=G/P$ as in Lemma~\ref{lem:cominuscule},
and that furthermore a maximal torus $T$ and a Borel subgroup $B$ have
been chosen so that $P=P_\alpha$ is a standard maximal parabolic
as in hypothesis 2 of the Lemma. We let $\Phi$ denote the root system
of $G$, $\Phi^+$ the set of positive roots (so that $\sR_uB$ is generated
by root spaces $U_{-\beta}$ with $\beta\in\Phi^+$), and $\Delta$ the set
of simple roots (so that in particular $\alpha \in \Delta$).
Let $\omega$ be the fundamental weight corresponding to $\alpha$,
and denote by $\Phi^+_\alpha \subset \Phi^+$ the subset
on which $\omega$ is zero.
Let $L_\alpha = P_\alpha / \sR_uP_\alpha$ be the Levi factor of $P_\alpha$,
together with an inclusion $L_\alpha \subset P_\alpha$ splitting the projection,
and let $L_\alpha^\ss$ be its semisimple part. $L_\alpha$ contains $T$
as a maximal torus, and the set of positive roots of $L_\alpha^\ss$ is identified
with $\Phi^+_\alpha \setminus \Phi^+_\alpha$. 
In particular, $\Delta \setminus
\{\alpha\}$ is the set of simple roots of $L^\ss_\alpha$.
We have $L_\alpha = P_\alpha \cap P_\alpha^+$
where $P_\alpha^+$ is the opposite parabolic.
We associate with $X$ the
Dynkin diagram of $\Phi$ with a marked node corresponding to $\alpha$.
A complete classification of cominuscule varieties in terms of their diagrams is then
given in the left column of Figure~\ref{fig:diagrams} (the right
column lists Dynkin diagrams of $L_\alpha^\ss$, where the marking will be explained
in the next subsection).

\begin{lem}
Let $X = G/P_\alpha$ be cominuscule. Then $X$ is a simply-connected
rational Fano variety with $\Pic(X)\simeq\mathbb{Z}$ generated by a very ample
invertible sheaf $\OO_X(1)$. 
\end{lem}
\begin{proof}
The projection $G \to X$ induces an 
open immersion
$\sR_u P_\alpha^+ \to X$, where $\sR_u P_\alpha^+ \simeq \AA^{\dim X}$
as a variety. It follows that $X$ is rational, hence simply connected. Now, invertible sheaves on $X$ are equivalent to
descent data on $G$, i.e. homomorphisms
$P_\alpha \to \GG_m$. Every such homomorphism factors through
$L_\alpha$ and is trivial on $L_\alpha^\ss$, hence factors through
the rank one torus $L_\alpha / L_\alpha^\ss \simeq \GG_m$. 

It follows
that $\Pic(X)\simeq\mathbb{Z}$, generated by the invertible sheaf $\OO_X(1)$
associated with the representation $P_\alpha \to \GG_m$ whose
restriction to $T$ is $\omega$. 
More precisely,
given an invertible sheaf $\sL_\lambda$ on $X$, with $\lambda$ the associated
weight of $T \subset P_\alpha$, we have $$\sL_\lambda \simeq \OO_X(\langle \lambda,\alpha^\vee\rangle).$$
This in particular shows that
$$ c_1(T_X) = \sum_{\beta \in \Phi^+ \setminus \Phi^+_\alpha} \langle \beta,\alpha^\vee\rangle 
= \langle 2\rho,\alpha\rangle - \sum_{\beta\in \Phi^+_\alpha}
\langle \beta,\alpha^\vee \rangle $$
where $\rho$ is the half sum of all positive roots. Since $\rho$
is dominant, the first term is non-negative. On the other hand,
the second term is negative, being given by the evaluation
on $\alpha^\vee$ on a positive combination of simple roots
in $\Delta \setminus \{\alpha\}$. Hence $c_1(T_X)>0$ and $X$ is Fano.
Very-ampleness of $\OO_X(1)$ follows from~\cite[Thm. 1]{ramanan}. 
\end{proof}

\begin{table}
\begin{center}
\fbox{\begin{tabular}{ll|lllll}
 $G$ & $\alpha$ & $X=G/P_\alpha$ &  embedding & $\dim(X)$ & $\ind(X)$ \\
\hline
$\sfA_n$ & $\alpha_i\sim\alpha_{n-i}$ & $\Gr(i,n+1)$ & Pl\"ucker & $i(n+1-i)$ &  $n+1$\\
$\sfB_n$ & $\alpha_1$ & $Q \subset \PP^{2n}$ & standard & $2n-1$ & $2n-1$\\
$\sfC_n$ & $\alpha_n$ & $\LG(n,2n)$ & Pl\"ucker & 
$n(n+1)/2$ & $n+1$\\
$\sfD_n$ & $\alpha_n\sim\alpha_{n-1}$ & $\OG(n,2n)$ & Pl\"ucker &
$n(n-1)/2$ & $2n-2$\\
$\sfD_n$ & $\alpha_1$ & $Q \subset \PP^{2n-1}$ & standard & $2n-2$ & $2n-2$\\
$\sfE_6$ & $\alpha_6\sim\alpha_1$ & Cayley plane&   & $16$ & $12$\\
$\sfE_7$ & $\alpha_7$ & Freudenthal variety&  & $27$ & $18$ 
\end{tabular}}\end{center}
\caption{Cominuscule varieties: root system of $G$, simple root defining $P_\alpha$,
root system of the semisimple part of the L\'evi factor,
type of $X=G/P_\alpha$, 
embedding by $|\OO_X(1)|$,
dimension and Fano index. 
We use Bourbaki's ordering of simple roots. 
}
\label{tab:cominuscule}
\end{table}

We list the basic information in Table~\ref{tab:cominuscule}.
Here $\LG(n,2n)$, resp. $\OG(n,2n)$, is the Grassmannian of maximal
isotropic subspaces in $k^{2n}$ equipped with a standard symplectic form,
resp. inner product; $Q \subset \PP^n$ denotes a quadric hypersurface;
the Cayley plane parametrises rays through idempotent elements in the Albert
algebra,
while the Freudenthal variety parametrises rays through strictly 
regular elements in the Freudenthal Triple System associated with the Albert
algebra (cf.~\cite{ferrar}). We will refer to $X \subset \PP H^0(X,\OO_X(1))^\vee$
as a \emph{minimally embedded} cominuscule variety.

\subsection{Varieties of line tangents}
We continue the notation of the previous subsection,
with $X=G/P_\alpha$ a cominuscule variety.
%
Let $I(\alpha) \subset \Delta$ be the set of simple roots corresponding
to nodes adjacent to $\alpha$ in the Dynkin diagram of $\Phi$, 
including $\alpha$.
There is a corresponding standard reduced parabolic
$P_{I(\alpha)}$ in $G$, and its intersection
$P_{I(\alpha)} \cap L_\alpha^\ss$ is a standard reduced parabolic in $L_\alpha^\ss$
corresponding to $I(\alpha) \setminus\{\alpha\} \subset \Delta\setminus\{\alpha\}$.
The latter is the set of marked nodes in the right column of Figure
\ref{fig:diagrams}.
It follows that $P_{I(\alpha)}$
is the normaliser of $U_\alpha$ in $P_\alpha$,
and $P_{I(\alpha)} \cap L_\alpha^\ss$ is the stabiliser
in $L_\alpha^\ss$ of the line $(\fg_{\alpha} + \fp)/\fp$ in $\fg/\fp$.

We now consider rational curves on $X$ of degree $1$ with
respect to $\OO_X(1)$. These are simply lines in 
$\PP H^0(X,\OO_X(1))^\vee$ contained in $X$, and we will
refer to them this way. We will use the notation of Section
\ref{sec:ratcurves}.
\begin{pro}\label{pro:cominuscule-lines}
Let $X=G/P_\alpha$ be a cominuscule variety with origin $x \in X(k)$.
Then:
\begin{enumerate}
\item Lines on $X$ form an irreducible component $\sM \subset\underline\Hom_\bir^n(\PP^1,X)$.
\item The natural action of $G$ on $\sM_1$ is transitive, and we have
$\sM_1 \simeq G / P_{I(\alpha)}$.
\item All lines are minimal, and the tangent map $\sM_1 \to \PP T_X$ is a closed
immersion.
\item $ x \times_X \sM_1 \simeq L^\ss_\alpha / (P_{I(\alpha)}\cap L^\ss_\alpha)$,
and the embedding $x\times_X \sM_1 \to \PP T_{X,x}$
is defined by the invertible sheaf associated with $-\alpha|_{T\cap L_\alpha^\ss}$.
\end{enumerate}
\end{pro}
\begin{proof}
This is essentially a corollary of the Main Theorem in~\cite{cohen}.
More explicitly, let $\SL_2 \hookrightarrow G$
be the subgroup corresponding to $\alpha$, so that
the maximal torus of $\SL_2$ maps to $T$, and
the positive root subgroup of $\SL_2$ maps to the root subgroup $U_\alpha \subset G$
associated with $\alpha$. Then $\SL_2 \cap P_\alpha$ is a Borel subgroup
of $\SL_2$, and the inclusion $\SL_2 \hookrightarrow G$ descends to 
a rational curve $c_\alpha : \PP^1 \to X$ of degree $1$ with respect to
$\OO_X(1)$.
The Main Theorem in~\cite{cohen} states that every line in $X$
is a $G$-translate of $c_\alpha(\PP^1)$. In fact, by construction, every
parametrised line $c:\PP^1 \to X$ is of the form $g \circ c_\alpha$ for
some $g\in G(k)$.

It follows that lines on $X$ are free (in fact minimal), and form a single 
irreducible component $\sM$, which is reduced (in fact nonsingular),
and thus $G$-homogeneous. Likewise for $\sM_1$, which is
additionally proper, hence a quotient of $G$ by a parabolic
subgroup. Since $c_\alpha(\PP^1) = \overline{U_\alpha P_\alpha}/P_\alpha$,
we have that the stabiliser of $[c_\alpha] \in \sM_1$
in $P_\alpha$ is $P_{I(\alpha)}$.

Obviously now $\sM_1 \to \PP T_X$
is a closed immersion,
so that $x\times_X\sM_1$ embeds into $\PP T_{X,x} \simeq \PP(\fg/\fp)$
as $L_\alpha^\ss / L_\alpha^\ss \cap P_{I(\alpha)}$. 
The tangent direction to
$c_\alpha(\PP^1)$ at $x$ is the image of the root subspace $\fg_\alpha$ in $T_{X,x}\simeq \fg/\fp$.
It follows that $L_\alpha^\ss \cap P_{I(\alpha)}$
acts on the fibre of $\OO_{\PP T_{X,x}}(-1)$ at $[c_\alpha] \in \PP T_{X,x}$
via $\alpha|_{T \cap L_{\alpha}^\ss}$. 
Hence $\OO(1)$ restricts on $x\times_X\sM_1$ 
to the invertible
sheaf associated with $-\alpha|_{T\cap L_{\alpha}^\ss}$.
\end{proof}

Note that $x\times_X\sM_1 \to \PP T_{X,x}$ is a closed immersion
onto the variety of line tangents at $x$. Associating with
$L_\alpha^\ss / (P_{I(\alpha)} \cap L_{\alpha}^\ss)$ the
Dynkin diagram of the root system $L_\alpha^\ss$,
and marking nodes corresponding to $I(\alpha) \setminus \{\alpha\}$,
gives the right column in Figure~\ref{fig:diagrams}.
The weight $-\alpha|_{T\cap L_\alpha^\ss}$ defining
the embedding is a combination of fundamental weights of $L_\alpha^\ss$,
and the coefficients determine the multiplicity of marking on
corresponding nodes (this is simply $1$ for all marked nodes, except
for $\LG(n,2n)$ where $-\alpha|_{T\cap L_\alpha^\ss} =
2\omega_{n-1}$). We list this information in Table~\ref{tab:vmrt}.
\begin{table}
\begin{center}
\fbox{\begin{tabular}{ll|lllll}
 $G$ & $\alpha$ & $L_\alpha^\ss$ & $x\times_X\sM_1$ & embedding &
$r$ \\
\hline
$\sfA_n$ & $\alpha_i$ & $\sfA_{i-1}\times\sfA_{n-i}$ & 
$\PP^{i-1}\times \PP^{n-i}$ & $|\OO(1,1)|$ & $\min(i,n+1-i)$\\
$\sfB_n$ & $\alpha_1$ & $\sfB_{n-1}$ &
$Q \subset \PP^{2n-2}$ & $|\OO(1)|$ & $2$ \\
$\sfC_n$ & $\alpha_n$ & $\sfA_{n-1}$ &
$\PP^{n-1}$ & $|\OO(2)|$ & $n$\\ 
$\sfD_n$ & $\alpha_n$ & $\sfA_{n-1}$ &
$\Gr(2,n)$ & $|\OO(1)|$ & $\lfloor \frac{n}{2}\rfloor $\\
$\sfD_n$ & $\alpha_1$ & $\sfD_{n-1}$ & 
$Q \subset \PP^{2n-3}$ & $|\OO(1)|$ & $2$ \\
$\sfE_6$ & $\alpha_6$ & $\sfD_5$ &
$\OG(5,10)$ & $|\OO(1)|$ & $2$\\
$\sfE_7$ & $\alpha_7$ & $\sfE_6$ &
Cayley plane & $|\OO(1)|$ & $3$
\end{tabular}}\end{center}
\caption{Varieties of line tangents and length of
connecting chains.}
\label{tab:vmrt}
\end{table}

\begin{cor}\label{cor:cominuscule-lines}The variety of line tangents at the origin of a cominuscule
homogeneous variety is one of the following:
\begin{enumerate}
\item a Segre variety, 
\item a Veronese variety of degree $2$, or
\item a minimally embedded cominuscule variety.
\end{enumerate}\end{cor}

\subsection{Chains of lines}

As in the proof of Proposition~\ref{pro:cominuscule-lines}, we consider
the subgroup $\SL_2 \subset G$ with root subgroups $U_{\pm\alpha}$, and
the line $c_\alpha :\PP^1 \to X$. Note that in particular the open immersion
$\sR_u P_\alpha^+ \simeq \AA^{\dim X} \to X$ restricts to $U_{\alpha} \simeq \AA^1 \to c_\alpha(\PP^1)$. The pullback of the $P_\alpha$-principal bundle $G \to X$ by $c_\alpha$ gives a subvariety
$$ c_\alpha^* G = \SL_2 \cdot P_\alpha = \overline{U_{\alpha}P_\alpha} $$
in $G$, a $P_\alpha$-principal bundle over $\PP^1$. We will inductively
construct morphisms $q^i$ fitting into a commutative diagram
$$ \begin{CD}
 ( \overline{U_{\alpha} P_\alpha})^i @>{q^i}>> [c_\alpha]\times_{\sM_1}\sM_2^i \\
@VVV @VVV \\
G @>>> X
\end{CD} $$
where the left vertical arrow is an $i$-fold product in $G$,
the right vertical arrow is the right structure morphism, and
$[c_\alpha] \in \sM_1(k)$ is the image of $c_\alpha \in \sM(k)$. 
By abuse of notation,
define $[c_\alpha]\times_{\sM_1}\sM_2^0$ to be
$\Spec k$, with both structure maps to $X$ given by $x$.
Then $q^0 = \id_{\Spec k}$.
Suppose by induction that $q^i$ has been defined
and makes the above diagram commute.
We then let $q^{i+1}$ be the composite
\begin{equation*} 
\overline{U_{\alpha}P_\alpha} \times (\overline{U_{-\alpha}P_\alpha})^i 
\xrightarrow{\id \times q^i}
\overline{U_{\alpha}P_\alpha} \times
[c_\alpha]\times_{\sM^1}\sM_2^i 
\xrightarrow{\langle \varphi \circ \pr_1, \psi \rangle}
[c_\alpha]\times_{\sM^1}\sM_2 \times_X \sM_2^i
\end{equation*}
where $\psi : G\times\sM_2^i \to \sM_2^i$ is the action morphism,
while $\phi : \overline{U_{\alpha}P_\alpha} \to \sM_2$ 
sends $g$ to $[g  c_\alpha, (g c_\alpha)^{-1}(x)] \in \sM_2$. 
It is clear by construction that $q^{i+1}$ makes the corresponding
diagram commute.
\begin{lem}\label{lem:qi}
$q^i$ is a composite of $P_\alpha$-principal bundles.
\end{lem}
\begin{proof}
There is a Cartesian diagram
$$\begin{CD}
\overline{U_{\alpha} P_\alpha} \times [c_\alpha]
\times_{\sM^1}\sM_2^i @>{\pr_1}>> U_{-\alpha}P_\alpha \\
@V{\langle \varphi\circ\pr_1,\psi\rangle}VV
@VVV \\
[c_\alpha]\times_{\sM_1}\sM_2\times_X\sM_2^i @>{c_\alpha^{-1}\circ e_2\circ \pr_1}>> \PP^1.
\end{CD} $$
where in the bottom horizontal
arrow $\pr_1$  projects to $[c_\alpha]\times_{\sM_1}\sM_2$,
and $e_2 : \sM_2 \to X$ is the right
structure morphism.
Since the right vertical arrow is a $P_\alpha$-principal bundle,
so is $\langle \varphi\circ \pr_1,\psi\rangle$. 
Assuming by induction that $q^i$ is a composite
of $P_\alpha$-principal bundles, so is $q^{i+1} = (\id\times q^i) \circ \langle
\varphi\circ\pr_1,\psi\rangle$.
\end{proof}

\begin{pro}\label{pro:cominuscule-chains}
Given a cominuscule variety $X=G/P$
with $\sM \subset \underline\Hom_{\bir}^n(\PP^1,X)$ the component
of lines,
let $r+1$ be the number of $P$-$P$ double cosets in $G$.
Then $\sM_2^i \to X\times X$ is separably surjective for $i \ge r$.
\end{pro}
\begin{proof}
As before, we can assume $P=P_\alpha$
for a suitable choice of a maximal torus
and a Borel subgroup. Consider first the action morphism
$$ \theta^i : G \times [c_\alpha]\times_{\sM_1}\sM_2^i \to \sM_2^i. $$
There is a Cartesian diagram
$$\begin{CD}
G \times [c_\alpha]\times_{\sM_1}\sM_2^i @>>> G \\
@V{\theta^i}VV @V{[c_\alpha]}VV \\
\sM_2^i @>>> \sM_1
\end{CD}$$
where the bottom horizontal arrow is a composite
of projection to the first segment $\sM_2^i \to \sM_2$
and the natural map $\sM_2 \to \sM_1$,
while the right vertical arrow is a $P_{I(\alpha)}$-principal bundle.
Hence $\theta^i$ is a $P_{I(\alpha)}$-principal bundle.

We now have a commutative diagram
$$\begin{CD}
G \times (\overline{U_{\alpha}P_\alpha})^i @>{\theta^i\circ(\id\times q^i)}>> \sM_2^i \\
@V{\id \times \mu^i}VV @VVV \\
G \times G @>>> X \times X
\end{CD}$$
where
the left vertical arrow is given by multiplication 
$\mu^i:(\overline{U_{\alpha}P_\alpha})^i \to G$, 
while the right vertical arrow is the evaluation morphism.
It will be enough to check that
$\mu^i$ is separably surjective.
By maximality of $P_\alpha$ it follows
that $\mu^i$ is surjective for $i$ sufficiently large.
On the other hand, its image is a union of $P_\alpha$-$P_\alpha$
double cosets, and its dimension stabilises
once the image of $\mu^{i+1}$ coincides with that of $\mu^i$.
Since $U_{\alpha}P_\alpha$ is strictly larger than
the first double coset $P_\alpha$, it follows that
$\mu^i$ is surjective for $i\ge r$.
For separability of $\mu^i$ it is enough to notice
that the normaliser of $U_{\alpha}$ in $P_\alpha$ is the
reduced parabolic $P_{I(\alpha)}$
stabilising $[c_\alpha]$.
\end{proof}

The numbers $r$ appearing in the Proposition are listed in
Table \ref{tab:vmrt}, following~\cite{rrs}

\subsection{Further properties}
\label{ss:numerical}
Let $H = c_1(\OO_X(1))$ be the hyperplane class for the minimal embedding.
We will use the notation
$$ \delta_X = [X]\cdot H^{\dim X}, $$
for the degree of the minimal embedding.
Consider the total evaluation morphism
$$ \nu^i : \sM_2^i \to X^{i+1} $$
for $i\ge r$ as in Proposition \ref{pro:cominuscule-chains}.
Letting $p^i_j : X^{i+1}\to X$, $0\le j\le i$ be the natural projections,
we define the following intersection number:
$$ \delta_X(i) = \nu^i_*[\sM_2^i]\cdot p^{i*}_0[x] \cdot 
\left(\sum_{j=1}^{i-1}p^{i*}_jH\right)^{d_i} \cdot p^{i*}_i [x] $$
where $d_i$ is the dimension of the generic fibre of $\sM_2^i \to X\times X$.
We will express $\delta_X(i)$ in terms of Schubert calculus on $X$.

Denote by $V \subset \PP T_{X,x}$ the variety of line tangents at $x$,
and by $C_x \subset X$ the subvariety swept by lines through $x$, the closure
in $X$ of the image of $P_\alpha U_\alpha P_\alpha$ under the projection
$\pi:G\to X$. Letting $\tilde C_x \to C_x$ be the blowing-up at $x$,
we have a natural commutative diagram
$$\begin{CD}
x\times_X \sM_1 @<<< x\times_X\sM_2 @>>> X \\
@V{\simeq}VV @V{\simeq}VV @AAA \\
V @<<< \tilde C_x @>>> C_x
\end{CD}$$
where the left horizontal arrows are Zariski-locally trivial $\PP^1$-bundles.

The morphism 
$G \times C_x \to X\times X$ sending $(g,c)$ to $(gx,gc)$ factors
through a closed immersion $G\times^{P_\alpha}C_x \to X\times X$,
and we let $C \subset X\times X$ be its image. It follows that
$C = G \cdot (x\times C_x)$, with the diagonal action of $G$ on $X\times X$.
Define
$C^i = (X \backslash C / X)^i,$
naturally a closed subvariety
of $X^{i+1}$. There is a commutative diagram
$$\begin{CD}
C^{i+1} @>>> X^{i+1} \\
@VVV @VVV \\
C^{i} @>>> X^{i}
\end{CD}$$
where the horizontal arrows are the natural closed immersions,
the right vertical arrow is the projection onto the first $i$ factors,
and the left vertical arrow is a Zariski-locally trivial $C_x$-fibration.
It follows that $C^i$ is an iterated $C_x$-fibration.
\begin{lem}\label{lem:m-c-bir}
The total evaluation map $\nu^i:\sM_2^i \to X^{i+1}$
factors through a birational morphism onto $C^i$.
\end{lem}
\begin{proof}
That $\nu^i$ factors through a proper surjective morphism onto
$C^i$ follows by induction on $i$.
Since $\sM$-curves are projective lines, the restriction
of $\nu^i$ over the complement of the diagonals in $X^{i+1}$ is a proper monomorphism,
i.e. a closed immersion. 
\end{proof}
It follows that
$$ \delta_X(i) = [C^i] \cdot p^{i*}_0[x] \cdot 
\left(\sum_{j=1}^{i-1}p^{i*}_jH\right)^{d_i} \cdot p^{i*}_i [x]. $$
Assuming $i\ge r$ as in Proposition \ref{pro:cominuscule-chains}, we have
\begin{equation*}\label{eq:dimension-ci}
d_i = \dim C^i - 2\dim X = i \dim C_x - \dim X = i\dim V+i-\dim X.
\end{equation*}
Letting $\{[X_{\sigma}]\}_{\sigma \in S}$ be an additive basis of the Chow ring
$A^*(X)$
consisting of closed Schubert varieties, and identifying $A^*(X^{i+1})$ with
$A^*(X)^{\otimes (i+1)}$, we can write
$$[C] = \sum a^{\sigma\tau} [X_{\sigma}]\otimes [X_{\tau}]$$
where $(a^{\sigma\tau})$ is a symmetric integer matrix. Then
$$
[C^i] = \sum a^{\rho_0\dots\rho_i} [X_{\rho_0}]\otimes\cdots\otimes [X_{\rho_i}] $$
with
\begin{equation*}\label{eq:arho}
a^{\rho_0\dots\rho_i} = 
\sum a^{\sigma_0\tau_0} \cdots a^{\sigma_i\tau_i} 
\delta^{\rho_0}_{\sigma_0} 
\mu_{\tau_0\sigma_1}^{\rho_1} \cdots \mu_{\tau_{i-1}\sigma_i}^{\rho_{i-1}}
\delta^{\rho_i}_{\tau_i}
\end{equation*}
where $(\mu_{\sigma\tau}^\rho)$ are the Littlewood-Richardson coefficients:
$[X_\sigma]\cdot[X_\tau] = \sum \mu_{\sigma\tau}^\rho [X_\rho]$.

Letting $\xi \in S$ be the unique element with $[X_\xi]=[X]$, we have
$$ p_0^{i*}[x] \cdot [C^i] \cdot p_i^{i*}[x] =
\sum a^{\xi\rho_1\dots\rho_{i-1}\xi} [x] \otimes [X_{\rho_1}]\otimes\cdots\otimes
[X_{\rho_{i-1}}]\otimes [x] $$
so that
\begin{eqnarray*}
 \delta_X(i) &=& \sum_{\substack{\rho_1,\dots,\rho_{i-1}\in S
\\ b_1+\dots+b_{i-1}=d_i}}
N(d_i;b_1,\dots,b_{i-1})
a^{\xi\rho_1\dots\rho_{i-1}\xi}
\prod_{j=1}^{i-1} [X_{\rho_j}]\cdot H^{b_j}
\\
&=&
\sum_{\rho_1,\dots,\rho_{i-1}}
N(i\dim V+i-\dim X;\dim X_{\rho_1},\dots,\dim X_{\rho_{i-1}})
a^{\xi\rho_1\dots\rho_{i-1}\xi}
\prod_{j=1}^{i-1} \deg X_{\rho_j}
\end{eqnarray*}
where $N(d;\underline b)$ are the multinomial coefficients:
$(\sum x_j)^d = \sum_{\underline b} N(d;\underline b) \prod x_j^{b_j}$.
It follows that $\delta_X(i)$ can be computed in terms of the matrix
$(a_{\sigma\tau})$ and the degrees and dimensions of Schubert varieties $X_\sigma$.

In order to apply Proposition \ref{pro:linear-nondegeneracy} in Chapter
\ref{chap:rigidity}, we will also need the following property. The proof
by Hwang and Mok~\cite[Proposition 14]{hwang-mok-rigidity} goes without change.
\begin{lem}\label{lem:cominuscule-wedge2}
Assume $\Char k \neq 2$. Then $(X,\sM)$ satisfy hypothesis (4) of Proposition \ref{pro:linear-nondegeneracy}.
\end{lem}

%% file: chap3.tex
\chapter{The extension theorem}\label{chap:extension}

\section{Morphisms of varieties with curves}

\subsection{Jets of \'etale morphisms}
Given a nonsingular variety $X$, let $\Delta_X : X \to X^\sharp$
denote the diagonal section. Since $X^\sharp$ is a bundle of formal discs,
the sheaf
$\underline\Aut_X(X^\sharp ; \Delta_X)$
is an affine group scheme over $X$, a Zariski-locally trivial twisted form
of the group of origin-preserving automorphisms of a formal disc
of dimension $\dim(X)$.
Given another 
nonsingular variety $Y$ with $\dim(Y)=\dim(X)$, let
$$ X \xleftarrow{\pr_X} X \times Y \xrightarrow{\pr_Y} Y $$
be the two projections, and
consider the sheaves
\begin{eqnarray*} 
\sF_{X,Y} &=& \underline\Hom_{X\times Y}(\pr_X^*X^\sharp,\pr_Y^*Y^\sharp ;
\pr_X^*\Delta_X,\pr_Y^*\Delta_Y) \\
\sE_{X,Y} &=& \underline\Isom_{X\times Y}(\pr_X^*X^\sharp,\pr_Y^*Y^\sharp ;
\pr_X^*\Delta_X,\pr_Y^*\Delta_Y) 
\end{eqnarray*}
over $X\times Y$. Since $\pr_X^* X^\sharp$
and $\pr_Y^* Y^\sharp$ are bundles of formal discs
of equal dimension,
it follows that $\sF_{X,Y}$ is a Zariski-locally trivial $\AA^\infty$-bundle
over $X\times Y$, while
$\sE_{X,Y}$
is a Zariski-locally trivial right torsor
for $\underline\Aut_X(X^\sharp;\Delta_X)\times Y$, and
a Zariski-locally trivial left torsor for
$X\times\underline\Aut_Y(Y^\sharp;\Delta_Y)$.
\begin{lem}\label{lem:e-stratification}
$\sF_{X,Y} / X$ is equipped
with a natural stratification, restricting to $\sE_{X,Y}/X$.
\end{lem}
\begin{proof}
Let $\epsilon : \sF_{X,Y}\times_X X^\sharp \to Y^\sharp$ be the universal
map. 
Consider the commutative diagram
$$\begin{CD}
\sF_{X,Y}\times_XX^\sharp\times_XX^\sharp @>{\id\times p_{13}}>>
\sF_{X,Y}\times_XX^\sharp @>{\epsilon}>> Y^\sharp \\
@A{\id\times\Delta_X}AA @. @AA{\Delta_Y}A \\
\sF_{X,Y} \times_XX^\sharp @>{\epsilon}>> Y^\sharp @>{p_2}>> Y
\end{CD}$$
where $p_2 :Y^\sharp \to Y$ is the right projection,
and $p_{13} : X^\sharp\times_XX^\sharp \to X^\sharp$ is the leftmost-rightmost projection.
By the universal property of $\sF_{X,Y}$, the top composite $\epsilon \circ (\id\times p_{13})$ defines a morphism
fitting into a diagram
$$\begin{CD}
 \sF_{X,Y}\times_X X^\sharp @>>> \sF_{X,Y} \\
 @VVV @VVV \\
X^\sharp @>>> X
\end{CD}
$$
where the bottom horizontal arrow is the \emph{right} structure map.
This gives the desired morphism $\sF_{X,Y}\times_XX^\sharp \to X^\sharp\times_X\sF_{X,Y}$
over $X^\sharp$. Since $\sE_{X,Y}$ is an open subscheme of $\sF_{X,Y}$, it
inherits a stratification.
\end{proof}
The sheaf $\sE_{X,Y}$ parametrises $\infty$-jets of formally-\'etale 
maps $X \to Y$. Such maps induce isomorphisms of spaces of arcs:
the universal map
$$ \Phi_{X,Y} : \sE_{X,Y} \times_X X^\sharp \to \sE_{X,Y}\times_Y Y^\sharp $$
lifts to produce a commutative diagram
$$\begin{CD}
 \sE_{X,Y} \times_X X^\sharp\times_X\Arc_X @>{\tilde\Phi_{X,Y}}>> \sE_{X,Y}\times_Y Y^\sharp\times_Y\Arc_Y  \\
@VVV @VVV \\
 \sE_{X,Y} \times_X X^\sharp @>{\Phi_{X,Y}}>> \sE_{X,Y}\times_Y Y^\sharp 
\end{CD}$$
where both horizontal arrows are isomorphisms.

We now want to restrict to maps preserving
families of arcs induced by families of rational curves.
Let us for convenience introduce the following notion.
\begin{defn}
A \emph{good family} on a nonsingular projective variety $X$
is a irreducible component $\sM \subset \underline\Hom_\bir^n(\PP^1,X)$
such that the generic $\sM$-curve is minimal and unramified at $0$,
and the generic fibre of $\sM_1^\free \to X$ is geometrically irreducible and positive-dimensional.
\end{defn}
Given a nonsingular variety with
good family $(X,\sM)$, we will denote by $\hat\sM_1 \subset \Arc_X$
the closure of the image of the generic point of $\sM_1$ under the natural
monomorphism $\sM_1^\arc \to \Arc_X$ (in particular, $\sM_1$ and $\hat\sM_1$
are birational). Given another such pair $(Y,\sN)$,
with $\dim(Y)=\dim(X)$, consider the subsheaf
$\sC_{X,Y} \subset \sE_{X,Y}$ defined by
$$ \sC_{X,Y}(T) = \{ \varphi \in \sE_{X,Y}(T)\ |\ 
\varphi^*\tilde\Phi_{X,Y} : T \times_X X^\sharp \times_X \hat\sM_1
\xrightarrow{\simeq} T\times_Y Y^\sharp\times_Y \hat\sN_1
\}. $$
It is a closed subscheme of $\sE_{X,Y}$,
parametrising
$\infty$-jets of formally \'etale maps $X\to Y$
inducing isomorphisms of $\hat\sM_1$ onto $\hat\sN_1$.
We will also use the `generic locus'
$$ \sC_{X,Y}^\circ = x\times_X \sC_{X,Y}\times_Y y $$
where $x$, $y$ are the generic points of $X$, $Y$.
\begin{lem}\label{lem:c-stratification}
The stratification on $\sE_{X,Y} / X$ restricts to $\sC_{X,Y} / X$ and $\sC^\circ_{X,Y}/X$.
\end{lem}
\begin{proof}
It is enough to notice that
$\tilde\Phi_{X,Y}$ is horizontal with respect to the pullbacks of the canonical stratifications
on $X^\sharp\times_X\Arc_X$ and $Y^\sharp\times_Y\Arc_Y$.
\end{proof}
\begin{lem}\label{lem:c-generic}
$\sC_{X,Y}^\circ$ is dense in $\sC_{X,Y}$.
\end{lem}
\begin{proof}
Let $c_0 \in \sC_{X,Y}$ be a point with $\pr_X(c_0)=x_0$.
Let $I \subset \OO_{X,x_0}$ be the kernel of $\OO_{X,x_0} \to \pr_{X*}\OO_{\sC_{X,Y},c_0}$. 
By Lemma \ref{lem:c-stratification}, $I \subset \cap_r \fm_{x_0}^r = 0$,
so that $x \times_X \Spec \OO_{\sC_{X,Y},c_0}$ is dense in $\Spec \OO_{\sC_{X,Y},c_0}$.

Now let $c \in x \times_X \Spec \OO_{\sC_{X,Y},c_0}$ be a point with
$\pr_Y(c)=y_0$.
Let $J \subset \OO_{Y,y_0}$ be the kernel of $\OO_{Y,y_0} \to \pr_{Y*} \OO_{\sC_{X,Y},c}$. 
Since the definition of $\sC_{X,Y}$ is symmetric in $X$, $Y$, Lemma
\ref{lem:c-stratification} gives a stratification on $\sC_{X,Y}/Y$.
Hence $J \subset \cap_r \fm_{y_0}^r = 0$,
so that $\Spec\OO_{\sC_{X,Y},c}\times_Y y$ is dense in
$\Spec\OO_{\sC_{X,Y},c}$.
\end{proof}
\begin{lem}\label{lem:c-m1-n1}
There is a natural commutative diagram
$$\begin{CD}
\sC_{X,Y}^\circ \times_X X^\sharp \times_X m_1 @>>> \sC_{X,Y}^\circ\times_Y Y^\sharp\times_Y n_1 \\
@VVV @VVV \\
\sC_{X,Y}^\circ \times_X X^\sharp \times_X \hat\sM_1 @>{\tilde\Phi_{X,Y}}>> \sC_{X,Y}^\circ\times_Y Y^\sharp\times_Y\hat\sN_1
\end{CD}$$
where $m_1$, $n_1$ are the generic points of $\sM_1$, $\sN_1$.
\end{lem}
\begin{proof}
Recall that $m_1 \to \hat\sM_1$ and $n_1 \to \hat\sN_1$ are
inclusions of generic points, while the bottom horizontal arrow is an isomorphism.
The field extensions $\kappa(m_1)/\kappa(x)$ and $\kappa(n_1)/\kappa(y)$ are separable
by freeness of $m_1$, $n_1$.
It follows that 
for every $c \in C^\circ_{X,Y}$ the induced isomorphism 
$$\kappa(c)\otimes_{\kappa(x)}(x\times_X \hat\sM_1)  
\xrightarrow{\tilde\Phi_{X,Y}} \kappa(c)\otimes_{\kappa(y)}(y\times_Y \hat\sN_1)$$
sends $c\times_xm_1$ to $c\times_yn_1$.
Since $\sC_{X,Y}^\circ\times_Y n_1$ is a localisation of
$\sC_{X,Y}^\circ\times_Y\hat\sN_1$, it follows that
$$ \sC^\circ_{X,Y}\times_X X^\sharp\times_X\hat\sM_1 \xrightarrow{\tilde\Phi_{X,Y}}
\sC^\circ_{X,Y}\times_Y Y^\sharp\times_Y\sN_1 $$
sends $\sC^\circ_{X,Y}\times_XX^\sharp\times_X m_1$
into $\sC^\circ_{X,Y}\times_YY^\sharp\times_Y n_1$.
\end{proof}

\subsection{Parallel transport along curves}
Note that given a morphism $f:T \to X$, the stratification on $\sC_{X,Y}/X$
induces one on the pullback $f^*\sC_{X,Y} / T$:
$$\begin{CD}
T^\sharp \times_T f^*\sC_{X,Y}@= 
T^\sharp \times_X \sC_{X,Y} @=
T^\sharp \times_{X^\sharp} (X^\sharp\times_X\sC_{X,Y}) \\
@. @. @V{\simeq}VV \\
f^*\sC_{X,Y}\times_TT^\sharp
@= \sC_{X,Y}\times_XT^\sharp 
@= (\sC_{X,Y}\times_XX^\sharp)\times_{X^\sharp}T^\sharp.
\end{CD}$$
The key point of this section is then the following `parallel transport' result,
an analogue of Hwang and Mok's analytic continuation along rational 
curves~\cite{hwang-mok-extension}. 
\begin{pro}\label{pro:extension-m1}
Let $(X,\sM)$ and $(Y,\sN)$ be a pair of nonsingular varieties
with good families, such that $\dim(X)=\dim(Y)$. 
Let $m_2$ be the generic
point of $\sM_2$ together with evaluation maps $m_2 \rightrightarrows X$.
Then there is a natural isomorphism
$$  {\sC^\circ_{X,Y}}\times_Xm_2 \to m_2\times_X\sC^\circ_{X,Y} $$
horizontal over $m_2$ with respect to the induced stratifications.
\end{pro}

\begin{proof}
Let $\tilde m_1 \in \sM_2$
be the image of $m_1$ under the zero-section $\sM_1\to\sM_2$. 
Set
$$ M = \Spec \OO_{\sM_2,\tilde m_1},\quad \hat M =\Spf\hat\OO_{\sM_2,\tilde m_1} $$
so that $M$ is the spectrum of a discrete valuation
ring with closed point $\tilde m_1$ and generic
point $m_2$, and $\hat M$ is its completion.
We will first construct a morphism
$$
\psi  : \sC^\circ_{X,Y}\times_X X^\sharp \times_X M \to \sN_2.
$$
extending the canonical top horizontal arrow in the diagram
$$\begin{CD}
 \sC^\circ_{X,Y}\times_X X^\sharp \times_X \hat M @>{\hat\psi}>> \hat\sN_2 \\
@VVV @VVV \\
 \sC^\circ_{X,Y}\times_X X^\sharp \times_X \Arc_X \times_X X^\sharp @>>> \Arc_Y
\times_Y Y^\sharp
\end{CD}$$
where the bottom hotizontal arrow is induced by $\tilde \Phi$ and $\Phi$,
while the vertical arrows are induced by the natural
maps to the spaces of arcs and by evaluation at the second marked point.

With $(\sM_2/X)^2$ denoting the fibre product of $\sM_2$ with itself
with respect to the \emph{right} structure maps into $X$,
consider the natural diagram
$$ \sM_2 \overset{q_1}{\underset{\Delta}{\leftrightarrows}} (\sM_2/X)^\sharp \xrightarrow{q_2} X^\sharp\times_X \sM_2 $$
of morphisms over $X^\sharp$.
Its pullback by the rightmost structure
map  $\sC^\circ_{X,Y}\times_XX^\sharp\to X$
gives 
$$ \sC^\circ_{X,Y}\times_XX^\sharp\times_X M
\overset{\tilde q_1}{\underset{\tilde\Delta}{\leftrightarrows}}
\sC^\circ_{X,Y}\times_XX^\sharp\times_X (M/X)^\sharp
\xrightarrow{\tilde q_2} 
\sC^\circ_{X,Y} \times_X X^\sharp \times_X X^\sharp \times_X M. $$
Let $\pi:\sM_2\to\sN_1$, $\varpi:\sN_2\to\sN_1$ and
$p_{13} : X^\sharp\times_XX^\sharp\to X^\sharp$ denote the natural
projections. By Lemma \ref{lem:c-m1-n1}, we have the composite
\begin{eqnarray*}
\nu : \sC^\circ_{X,Y}\times_X X^\sharp\times_X (M/X)^\sharp
&\xrightarrow{\tilde q_2}&
\sC^\circ_{X,Y} \times_X X^\sharp\times_X X^\sharp\times_X M \\
&\xrightarrow{p_{13*}}& 
\sC^\circ_{X,Y}\times_XX^\sharp\times_X M \\
&\xrightarrow{\pi_*} &
\sC^\circ_{X,Y}\times_XX^\sharp\times_X m_1\\
&\xrightarrow{\tilde\Phi_{X,Y}}&
\sC^\circ_{X,Y}\times_YY^\sharp\times_Y n_1\\
&\xrightarrow{\pr_{n_1}}&
n_1.
\end{eqnarray*}
Consider now the pullback diagram
$$\begin{CD}
\tilde\Delta^*\nu^*\sN_2 @>>> \nu^*\sN_2 @>>> \sN_2 \\
@V{h}VV @VVV @V{\varpi}VV \\
\\
\sC^\circ_{X,Y}\times_XX^\sharp\times_X M
@>{\tilde\Delta}>>
\sC^\circ_{X,Y}\times_XX^\sharp\times_X (M/X)^\sharp
@>{\nu}>> \sN_1
\end{CD}$$
where the vertical arrows are $\PP^1$-bundles.
Viewing $\sN_2 \to \sN_1$ as the universal $\sN$-curve
in $Y$, let $N_{\sN_2/Y}$ be the universal normal sheaf on $\sN_2$.
Identifying $T_{M/X}$ with the pullback by $\Delta$ of the relative
tangent sheaf of $q_1$, 
the pullback  $\tilde\Delta^* d\nu$ defines by adjunction a map
$$
g : r_1^* T_{M/X} \to r_2^* N_{\sN_2/Y}
$$
of locally free sheaves on $\tilde\Delta^*\nu^*\sN_2$,
where.
$$ M \xleftarrow{r_1} \tilde\Delta^*\nu^*\sN_2 \xrightarrow{r_2} \sN_2, $$
are the natural projections.
\begin{lem}
The zero-locus of $g$
is the graph of a morphism
$$\psi: {\sC^\circ_{X,Y}}\times_XX^\sharp\times_XM \to \sN_2$$
lifting $\nu \circ \tilde\Delta$ and extending $\hat\psi$.
\end{lem}
\begin{proof}
Let ${\sC^\circ_{X,Y}}$ be the zero-locus of $g$.
Since $h$ is proper, so it $h|_{\sC^\circ_{X,Y}}$. 
Since the generic $\sN$-curve is minimal,
the ideal sheaf $\sI_{\sC^\circ_{X,Y}}$ of ${\sC^\circ_{X,Y}}$ in $\tilde\Delta^*\nu^*\sN_2$
splits along the fibres of the $\PP^1$-bundle $h$
into invertible sheaves with degrees in $\{-1,0\}$, so that
$R^1h_*\sI_{\sC^\circ_{X,Y}}=0$ and the natural map
$$ \OO_{\sC^\circ_{X,Y}\times_XX^\sharp\times_XM} \to h_*\OO_{\sC^\circ_{X,Y}} $$
is surjective. It follows that every geometric fibre of $h|_{\sC^\circ_{X,Y}}$
is either empty, a single reduced point, or a whole $\PP^1$.

Consider the pullback
$$\begin{CD}
\iota^*\tilde\Delta^*\nu^*\hat\sN_2 @>{\tilde\iota}>> \tilde\Delta^*\nu^*\hat\sN_2 
@>>> \hat\sN_2\\
@VVV @VVV @V{\varpi}VV\\
\sC^\circ_{X,Y}\times_X X^\sharp\times_X \hat M
@>{\iota}>> \sC^\circ_{X,Y}\times_X X^\sharp\times_X M
@>{\nu\circ \tilde\Delta}>> \sN_1
\end{CD}$$
where $\iota$ is the natural monomorphism.
By construction, the graph of 
$$\hat\psi : \sC^\circ_{X,Y}\times_XX^\sharp\times_X\hat M \to \hat\sN_2$$
factors through ${\sC^\circ_{X,Y}}$.
Since the geometric generic fibre of $\sM_1^\free \to X$ is positive-dimensional,
the geometric fibres of $h|_{\sC^\circ_{X,Y}}$ over $\sC^\circ_{X,Y}\times_XX^\sharp\times_X \tilde m_1$
are single reduced points, so that
the restriction ${\sC^\circ_{X,Y}} \cap \iota^*\tilde\Delta^*\nu^*\sN_2$ actually coincides with 
the graph of $\hat\psi$.
In particular, $\iota^*h|_{\sC^\circ_{X,Y}}$ is an isomorphism. 
Since $\iota$ is an epimorphism of formal schemes,
and $\varpi$ is a $\PP^1$-bundle,
it follows that $\tilde\iota$ is an epimorphism of formal schemes.
Thus $h|_{\sC^\circ_{X,Y}}$ is a closed immersion, adic and admitting a section over $\iota$,
hence an isomorphism.
\end{proof}

We have thus constructed the map $\psi$,
which will allow us to produce
a morphism
$ \sC^\circ_{X,Y}\times_X M \to M\times_X \sC^\circ_{X,Y} $
whose restriction over $m_2$ gives the isomorphism announced in the Proposition.
By freeness of the generic
$\sM$-curve, we can choose an isomorphism
$$ \rho : \sC^\circ_{X,Y}\times_X M\times_XX^\sharp \to \sC^\circ_{X,Y}\times_X
X^\sharp\times_X M $$
over $\sC^\circ_{X,Y}\times X$ (leftmost-rightmost structure map).
Let $\phi$ be the composite
$$ 
\phi : \sC^\circ_{X,Y}\times_X M\times_XX^\sharp \xrightarrow{\rho}
 \sC^\circ_{X,Y}\times_X X^\sharp\times_X M 
\xrightarrow{\psi}
\sN_2 \to Y
$$
where the rightmost arrow is the right structure map.
Consider now the pair 
$$\sC^\circ_{X,Y}\times_XM\times_XX^\sharp \overset{s^*\phi}{\underset{\phi}{\rightrightarrows}} Y $$
where
$s = \Delta_X \circ p_1:X^\sharp \to X^\sharp$
is the `retraction onto origin'.
These induce a morphism
$$ \theta = \langle \id , s^*\phi , \phi  \rangle : 
\sC^\circ_{X,Y}\times_XM\times_XX^\sharp
\to \sC^\circ_{X,Y}\times_X M \times_Y Y^\sharp $$
together with the map
$$ [\theta]:\sC^\circ_{X,Y} \times_X M \to M\times_X \sF_{X,Y} $$
defined by the universal property of $\sF_{X,Y}$.

Now, since $\psi$ is an extension of $\hat\psi$,
there is a commutative diagram
$$ 
\begin{CD}
\sC^\circ_{X,Y}\times_X \hat M \times_X X^\sharp 
@>{\theta}>> \sC^\circ_{X,Y}\times_X \hat M \times Y^\sharp\\
@V{\id\times (p_{13} \circ e)}VV @VV{\id\times\pr_{Y^\sharp}}V \\
\sC^\circ_{X,Y}\times_X X^\sharp @>{\Phi_{X,Y}}>> \sC^\circ_{X,Y}\times_Y Y^\sharp
\end{CD}$$
where $e:\hat M \to \hat X$ is the restriction
of the structure morphism $\sM_2 \to X\times X$.
Hence the restriction of $[\theta]$
to $\sC^\circ_{X,Y}\times_X\hat M$
is a pullback of the stratifying isomorphism
$$ \sC^\circ_{X,Y}\times_X X^\sharp \to X^\sharp\times_X \sC^\circ_{X,Y},$$
and in particular it is horizontal and factors through $\hat M\times_X \sC^\circ_{X,Y}$.
Since  $\sC^\circ_{X,Y}\times_X \hat M$ does not factor through
any proper subscheme of $\sC^\circ_{X,Y}\times_X M$,
it follows that $[\theta]$ factors through the open subscheme
$M\times_X \sE_{X,Y}\subset M\times_X \sF_{X,Y}$,
through the closed subscheme $M\times_X \sC_{X,Y} \subset M\times_X \sE_{X,Y}$,
and finally through the `generic locus' $M\times_X\sC^\circ_{X,Y} \subset
M\times_X \sC_{X,Y}$.
Letting $$\tau : \sC^\circ_{X,Y}\times_X M \to M \times_X  {\sC^\circ_{X,Y}}$$ be the
map induced by the point-swapping involution on $M\subset\sM_2$, we have a pair of
morphisms
$$
\sC^\circ_{X,Y}\times_X M \overset{[\theta]}{\underset{\tau[\theta]\tau}{\rightleftarrows}}
M\times_X\sC^\circ_{X,Y}
$$
such that $\tau[\theta]$ and $[\theta]\tau$ restrict to identity
over $\sC^\circ_{X,Y}\times_X \hat M$ and $\hat M\times_X \sC^\circ_{X,Y}$.
Hence the above morphisms are mutual inverses, and
their restriction over
$m_2$ gives the isomorphism announced in the Proposition, thus concluding its proof.
\end{proof}

\subsection{Induction and descent}
We can now use Proposition \ref{pro:extension-m1} inductively to
trivialise $\sC^\circ_{X,Y}$ along generic chains of $\sM$-curves.
Under suitable conditions, the trivialisation descends generically
to the base.
\begin{pro}\label{pro:c-descent}
Let $(X,\sM)$ and $(Y,\sN)$ be a pair of nonsingular varieties with
good families such that $\dim(X)=\dim(Y)$.
Suppose that $X$ is simply-connected,
of Picard number $1$.
Let $\xi$ be the generic point of $X\times X$. 
Then there is a natural isomorphism
$$\sC^\circ_{X,Y}\times_X \xi\to \xi\times_X{\sC^\circ_{X,Y}} $$
horizontal over $\xi$.
\end{pro}

We will need the following bit of commutative algebra.
\begin{lem}\label{lem:separable-closure}
Let $L/K$ be a finitely generated field extension, and $K^{s,L}$ the separable
algebraic closure of $K$ in $L$. Then the following is an equaliser diagram:
$$ K^{s,L} \to L \rightrightarrows \widehat{L\otimes_KL} $$
where we complete at the diagonal ideal.
\end{lem}
\begin{proof}
Let $K'\subset L$ be the equaliser of $L\rightrightarrows\widehat{L\otimes_KL}$.
We first observe that the claim is true in the following cases:
\begin{enumerate}
\item $L/K$ purely transcendental: then 
$L\otimes_KL \to \widehat{L\otimes_KL}$ is injective,
so that $K'=K$.
\item $L/K$ purely inseparable: then $L\otimes_KL$ is Artinian, hence already complete, 
and we argue as
above.
\item $L/K$ separable algebraic: then $L\otimes_KL$ is a product of finitely many
copies of $L$, so that the diagonal ideal is idempotent and $K'=L$.
\end{enumerate}
In the general case, by (2) we can assume that $L/K$ is separably generated, so that
there is an intermediate extension $K \subset L_0 \subset L$ such that $L/L_0$
is separable algebraic and $L_0/K$ purely transcendental. Suppose $x \in K'$.
Since $K'$ is invariant under the Galois group of $L/L_0$, the conjugates of $x$
are also contained in $K'$. It follows that the coefficients
of the minimal polynomial of $x$ in $L_0$ are contained in $K' \cap L_0$, and thus
in $K$ by (1). Hence $x$ is separable algebraic over $K$, and thus $K' \subset K^{s,L}$.
The converse follows by (3).
\end{proof}

\begin{proof}[Proof of Proposition \ref{pro:c-descent}.]
Let $m_2^i$ the generic point of $\sM_2^{i,\free}$. By Lemma \ref{lem:free-stuff},
$m_2^{i+1} \in m_2^i\times_X m_2$, so that Proposition \ref{pro:extension-m1} 
and induction on $i$
gives a horizontal isomorphism
$$ \theta^i : \sC^\circ_{X,Y}\times_X m_2^i \to m_2^i\times_X\sC^\circ_{X,Y} $$
over $m_2^i$.
By Lemma \ref{lem:free-chains} we can choose $i$ such that $m_2^i$ maps to $\xi\in X\times X$.
By horizontality of $\theta^i$, the pullbacks
$$ \sC^\circ_{X,Y}\times_X (m_2^i / \xi)^\sharp \rightrightarrows 
\sC^\circ_{X,Y} $$
of $\pr_{\sC^\circ_{X,Y}}\circ \theta^i$ by $(m_2^i / \xi)^\sharp \rightrightarrows m_2^i$
coincide. 
Hence, by Lemma \ref{lem:separable-closure}, $\theta^i$ descends
to a morphism
$$ 
\bar\theta : \sC^\circ_{X,Y}\times_X \tilde\xi
 \to \tilde\xi\times_X\sC^\circ_{X,Y} 
$$ 
where $\tilde\xi$
is the spectrum of
the separable algebraic closure
of $\kappa(\xi)$ in $\kappa(m_2^i)$.
Being an algebraic subextension
of a finitely generated extension, $\kappa(\tilde\xi)/\kappa(\xi)$
is finite.
Horizontality and invertibility of $\bar\theta$ 
follows from that of $\theta^i$ by descent.

To show that $\bar\theta$ is in fact
defined over $\kappa(\xi)$,
we first 
consider a geometric generic point $\bar\zeta$ of $X\times\sM_0$
and the corresponding rational curve
$$f : \PP^1_{\bar\zeta} \to \bar\zeta\times X \to X\times X.$$
\begin{lem}\label{lem:c-curves}
Let $W\subset f^*\tilde\xi$ be a connected component.
Then
$f^*\bar\theta:\sC^\circ_{X,Y} \times_X W \to W\times_X \sC^\circ_{X,Y}$
descends along
$W \to f^*\xi$. 
\end{lem}
\begin{proof}
Let $\bar\vartheta = \pr_2\circ \bar\theta:\sC^\circ_{X,Y}\times_X\tilde\xi\to\sC^\circ_{X,Y}$.
Recall that in the Proof of Proposition \ref{pro:extension-m1} we have actually constructed
an isomorphism
$ \theta : \sC^\circ_{X,Y} \times_X M \to M\times_X \sC^\circ_{X,Y}$
horizontal over $M = \Spec \OO_{\sM_2,\tilde m_1}$,
where $\tilde m_1$ is the image of
$m_1$ under the zero-section $\sM_1 \to \sM_2$.
Consider the diagram
with Cartesian squares
$$\begin{CD}
  \sC^\circ_{X,Y}\times_X (X \backslash \tilde\xi)^\sharp \times_{X\times X} \hat M @>>> 
  \sC^\circ_{X,Y}\times_X (X \backslash \tilde\xi)^2 \times_{X\times X} M @>{\Theta}>> 
\sC^\circ_{X,Y}\times_XM\times_X\sC^\circ_{X,Y}\\ 
@VVV @VVV @VVV\\
\sC^\circ_{X,Y}\times_X (X\backslash \tilde\xi)^\sharp @>>>
\sC^\circ_{X,Y}\times_X (X\backslash \tilde\xi)^2
@>{\bar\vartheta\times\bar\vartheta}>> \sC^\circ_{X,Y}\times\sC^\circ_{X,Y}
\end{CD}$$
The left square is induced by the natural inclusion $(X\backslash\tilde\xi)^\sharp \to
(X\backslash\tilde\xi)^2$. By horizontality, the top horizontal composite
factors through the graph of $\theta$. 

Identify $\PP^1_{\bar\zeta}$ and $\PP^1_{\bar\zeta}\times_{\bar\zeta}\PP^1_{\bar\zeta}$
with, respectively, $\bar\zeta \times_{\sM_0}\sM_1$ and
$\bar\zeta\times_{\sM_0}\sM_2$.
Consider $M_{\bar\zeta} = \bar\zeta\times_{\sM_0}M$
as a subscheme of $\PP^1_{\bar\zeta}\times_{\bar\zeta}\PP^1_{\bar\zeta}$.
The curve $f$ is identified with the natural morphism 
$\bar\zeta\times_{\sM_0}\sM_1 \to X \times X$ induced
by $\bar\zeta\to X$ and $\sM_1 \to X$.
Let $W \subset f^*\tilde\xi$ be an irreducible component.
Pulling back the top row of the above diagram, we have that the composite
$$ \sC^\circ_{X,Y} \times_X (W/\bar\zeta)^\sharp \to 
\sC^\circ_{X,Y} \times_{X} ( W / \bar\zeta  )^2 
\to \sC^\circ_{X,Y} \times \sC^\circ_{X,Y} $$
factors through the pullback of the diagram of $\theta$ by
$(W/\bar\zeta)^\sharp
\to M_{\bar\zeta}$. Hence so does the right arrow itself, and in particular
the restriction
$$ \sC^\circ_{X,Y} \times_X (W/\PP^1_{\bar\zeta})^2 \to \sC^\circ_{X,Y}\times\sC^\circ_{X,Y} $$
factors through the diagonal. Hence $f^*\bar\vartheta:\sC^\circ_{X,Y}\times W \to \sC^\circ_{X,Y}$
descends along $W \to f^*\xi$, and so does $f^*\bar\theta$.
\end{proof}

Continuing the proof of the Proposition,
fix a separable closure $\kappa(\bar\xi^s)$ of $\kappa(\xi)$.
The Galois group $\Gal(\bar\xi^s/\xi)$ acts on the set
$E$ of isomorphisms $\sC^\circ_{X,Y}\times_X \bar\xi^s\to\bar\xi^s\times_X\sC^\circ_{X,Y}$ horizontal over $\bar\xi^s$.
By the first part of the proof, there is an element
$\bar\theta \in E$ whose stabiliser
in $\Gal(\bar\xi^s/\xi)$ is of finite index.
Letting $\eta \simeq \bar\zeta\otimes k(t)$
be the generic point of $\PP^1_{\bar\zeta}$,
we have an extension $\kappa(\eta)/\kappa(\xi)$.
We can lift it to $\kappa(\bar\eta^s)/\kappa(\bar\xi^s)$
where $\bar\eta^s$ is a separable closure of $\eta$.
It then follows by Lemma \ref{lem:c-curves} that
the stabiliser of $\bar\theta$ in $\Gal(\bar\xi^s/\xi)$
contains the image of $\Gal(\bar\eta^s/\eta)$.

Let $\Gamma \to X\times X$ be a normal Galois cover corresponding to
the stabiliser of $\bar\theta$,
so that $\bar\theta$ is defined over the generic point of $\Gamma$, and
$f^*\Gamma$ is trivial by the previous paragraph.
We want to show that $\Gamma$ itself is trivial.
Since $X$ is simply-connected, it will be enough to show
that $\Gamma \to X\times X$ is \'etale.
Assuming the opposite, we have by the classical purity theorem
that it is ramified over a divisor $D \subset X\times X$.
Since the problem is symmetric under the transposition on $X\times X$,
we can assume that $D$ is not a pullback of a divisor
from the first factor.
Since
$X$ has Picard number $1$, the pullback $f^*D$ is positive,
and
there is a lift $\tilde f : \PP^1_{\bar\zeta} \to \Gamma$
of $f$ intersecting the ramification divisor. 
It follows that $f$ is tangent to $D$ at the intersection points.
But by Lemma \ref{lem:free-intersection} a generic $\sM$-curve
intersects $D$ transversely, a contradiction.
Hence $\Gamma$ is trivial, $\bar\theta$ is invariant under $\Gal(\bar\xi^s/\xi)$,
and thus finally defined over $\xi$.
\end{proof}

\section{Extension}
\begin{thm}\label{thm:extension}
Let $(X,\sM)$ and $(Y,\sN)$ be a pair of simply-connected, 
nonsingular projective Fano varieties of Picard number $1$ and
equal dimensions,
together with good families of rational curves. 
Let $K$ be an algebraically closed field, 
and $\bar c : \Spec K \to \sC^\circ_{X,Y}$ a geometric point. 
Then there is an isomorphism
$$ \phi : X \otimes K \to Y \otimes K $$
extending the canonical isomorphism
$\bar c^*\Phi : \bar c^*X^\sharp\to\bar c^*Y^\sharp$.
\end{thm}
\begin{proof}
By Lemma \ref{lem:free-chains}, $\sM_2^{i,\free} \to X\times X$
is dominant for some $i>0$.
Let $X_K = X \otimes K$ and $Y_K = Y \otimes K$ with generic points
$x_K$, $y_K$.
Then, by Proposition \ref{pro:c-descent}, there is a horizontal section
$$ \sigma : \eta_{X_K} \to \sC^\circ_{X,Y} \otimes K. $$
Its composite with projection to $Y_K$ extends to a morphism
$$ \phi_0 : X_K \setminus W \to Y_K,\quad\quad \codim_{X_K}W \ge 2 $$
whose restriction to $x_K$ is formally \'etale and induces
an isomorphism $x_K^* X^\sharp \times_X \hat\sM_1\simeq 
\phi_0|_{x_K}^* Y^\sharp\times_Y \hat\sN_1$.
It follows that
$\phi_0$ is dominant and generically \'etale. We will now show that
it is actally \'etale on entire $X_K\setminus W$.

Indeed, let $\bar\zeta$ be a geometric generic point
of $\sM_1 \otimes K$. Then the generic $\sM$-curve 
$ f : \PP^1_\zeta \to X_K $
factors through $X_K \setminus W$ (by Lemma \ref{lem:free-intersection}),
and $\phi_0 \circ f : \PP^1_\zeta \to Y_K$ is a generic
$\sN$-curve (since its restriction to $\hat\PP^1_\zeta$
maps, as an unramified morphism from a formal disc,
to the generic point of $\hat\sN_1$). Now, if $\phi_0$ is not \'etale,
then, by the classical purity theorem, it is ramified over a divisor
$D \subset Y_K$. Since $Y$ has Picard number $1$, $\phi_0\circ f$
intersects $D$. But since $\phi_0\circ f$ is free, the intersection
is transverse (again by Lemma \ref{lem:free-intersection}). 
It then follows that $f$ does not intersect the ramification divisor
in $X_K\setminus W$ -- a contradiction.

Hence $\phi_0$ is \'etale. Furthermore, since
$\phi_0 \circ f$ is free, it follows that the complement
of the image of $\phi_0$ has codimension at least $2$ in $Y_K$ (Lemma
\ref{lem:free-intersection}.
It then follows by simply-connectedness of $Y_K$ that $\phi_0$
is an isomorphism onto its image. 
Now, since $X$ and $Y$ are Fano, we can find
an integer $d>0$ such that $-dK_X$ and $-dK_Y$ are both very ample.
Being an isomorphism of open subsets whose complements have codimension
at least $2$, $\phi_0$ induces an isomorphism of Picard
groups and
of spaces of global sections for any invertible sheaf.
Using the differential $d\phi_0$ to identify
$\phi_0^*K_Y$ with $K_X$, 
we have a diagram
$$\begin{diagram}
\node{X_L}\arrow{s} \arrow{r,t,..}{\phi_0} \node{Y_L}\arrow{s} \\
\node{\PP H^0(X_K,\OO(-dK_X))^\vee} \arrow{=}
\node{\PP H^0(Y_K,\OO(-dK_Y))^\vee} 
\end{diagram}$$
where the vertical arrows are the projective embeddings
indced by $-dK_X$ and $-dK_Y$. Hence
$\phi_0$ extends to an isomorphism $\phi:X_K\to Y_K$.

It remains to check that $\phi$ is an extension
of $\bar c^*\Phi$. Consider the lift of $\phi$ to a morphism
$\tilde\phi$, horizontal over $\sM_2^i$ and fitting into a 
commutative diagram
$$\begin{CD}
 \bar c \times_X \sM_2^i @>{\tilde\phi}>>\bar c\times_X \sM_2^i \times_X \sC_{X,Y} \\
@VVV @VVV \\
\bar c \times X @>{\phi}>> \bar c \times Y
\end{CD}$$
where the left vertical arrow is induced by the right structure map
$\sM_2^i\to X$. 
Recall that in the proof of Proposition \ref{pro:extension-m1} we have
constructed an isomorphism
$$ \sC^\circ_{X,Y}\times_X M \to M \times_X \sC^\circ_{X,Y} $$
where $M = \Spec \OO_{\sM_2,\tilde m_1}$ and $\tilde m_1$
is the image of $m_1$ under the zero-section $\sM_1 \to \sM_2$.
Note that its restriction over $\tilde m_1$ is the 
identity on $\sC^\circ_{X,Y} \times_X m_1$.
By induction, we have an isomorphism
$$ \tilde\theta^i: \sC^\circ_{X,Y} \times_X (X\backslash M/X)^i \to (X\backslash M/X)^i
\times_X \sC^\circ_{X,Y} $$
extending $\theta^i$ of the proof of Proposition \ref{pro:c-descent}.
It follows that we have a commutative diagram
$$\begin{CD}
\bar c \times_X m_2^i @>>> \bar c \times_X \sM_2^i \\
@VVV @V{\tilde\phi}VV \\
\bar c \times_X (X\backslash M/X)^i @>{\bar c^*\tilde\theta^i}>> \bar c\times_X\sM_2^i \times_X \sC_{X,Y}
\end{CD}$$
i.e. $\tilde\phi$ and $\bar c^*\tilde\theta^i$ agree
on $\bar c \times_X m_2^i$. Then, by irreducibility and reducedness
of $\bar c\times_X \sM_2^{i,\free}$ (cf. Lemma \ref{lem:free-stuff}),
they agree on $(X\backslash M/X)^i \subset \sM_2^i$. In particular,
the composite
$$ \bar c \times_X m_1 \xrightarrow{\bar c^*\langle \tilde m_1,\dots,\tilde m_1\rangle} \bar c \times_X \sM_2^i \xrightarrow{\tilde \phi}
\bar c \times_X \sM_2^i \times_X \sC_{X,Y} \to \bar c\times_X \sC_{X,Y} $$
factors through
the diagonal embedding $\bar c \to \bar c\times_X \sC_{X,Y}$.
Since $\tilde\phi$ is a horizontal lift of $\phi$, it follows that
we have a commutative diagram
$$\begin{CD}
(\bar c\times_X X^\sharp) \times_{X\times X}\sM_2^i @>{\tilde\phi}>> 
\bar c \times_X
\sM_2^i \times_X \sC_{X,Y} \\
@VVV @VVV \\
\bar c\times_X X^\sharp @>{\bar c^*\Phi}>> \bar c \times Y 
\end{CD}$$
Since the left vertical arrow is an epimorphism of formal schemes,
it follows that the composite 
$$ \bar c \times_X X^\sharp \to \bar c\times X \xrightarrow{\phi} \bar c \times Y $$
coincides with $\bar c^*\Phi$.
\end{proof}

\begin{cor}\label{cor:extension}
Let $(X,\sM)$ and $(Y,\sN)$ be a pair of simply-connected, 
nonsingular projective Fano varieties of Picard number $1$ and
equal dimensions,
together with good families of rational curves. 
Let $K$ be an algebraically closed field, and
$\bar x_0:\Spec K \to X$, $\bar y_0:\Spec K \to Y$
a pair of geometric points such that there is an isomorphism
$\bar x_0^*X^\sharp \simeq \bar y_0^*Y^\sharp$
identifying $\bar x_0^*X^\sharp\times_X\hat\sM_1$ with
$\bar y_0^*Y^\sharp\times_Y\hat\sN_1$. Then
there is an isomorphism $X \simeq Y$ identifying
$\sM$ with $\sN$.
\end{cor}
\begin{proof}
By Lemma \ref{lem:c-generic}, $\sC^\circ_{X,Y}$ is nonempty,
so that there is an algebraically closed extension $L/K$
and a geometric point $\bar c : \Spec L \to \sC^\circ_{X,Y}$,
inducing by Theorem \ref{thm:extension} an isomorphism
$$ \phi_L : X\otimes L \to Y\otimes L $$
identifying $\hat\sM_1 \otimes L$ with $\hat\sN_1\otimes L$
and thus $\sM\otimes L$ with $\sN\otimes L$. Since $X$ and $Y$ are algebraic,
there is a subalgebra $A \subset L$, of finite type over $k$,
and such that $\phi_L$ is the base-change of an isomorphism
$$ \phi_A : X\otimes A \to Y\otimes A $$
identifying $\sM\otimes A$ with $\sN\otimes A$. Restricting
$\phi_A$ over a closed point of $\Spec A$ yields the desired
isomorphism $X\simeq Y$.
\end{proof}
 

%% file: chap4.tex
\chapter{The rigidity theorem}\label{chap:rigidity}

\section{The setup}

\begin{defn}\label{defn:degeneration}
Let $G/P$ be a cominuscule variety. A \emph{smooth projective degeneration}
of $G/P$ is a smooth, projective morphism $X \to S$ such that $S$ is the spectrum 
of a discrete valuation ring over $k$, with residue field $k$ and fraction field $F$, 
and the geometric generic
fibre of $X$ is isomorphic to $G/P \otimes \bar F$.
\end{defn}

Hwang and Mok~\cite{hwang-mok-rigidity} show that, over $k=\mathbb{C}$, every
smooth projective degeneration of a cominuscule variety $G/P$ is an isotrivial
fibration. Assuming from now on that $k$ is of characteristic $p>0$, we
want to find conditions on $p$ guaranteeing an analogous rigidity result. However,
we will need to introduce an additional parameter.

\begin{defn}\label{defn:rigidity}
Let $G/P$ be a cominuscule variety, and $d$ a positive integer. We will say that
$G/P$ is \emph{$d$-rigid} if every smooth projective degeneration $X \to S$ of $G/P$,
such that there exists a very ample invertible sheaf on $X$ retricting
to $\OO_{G/P}(d)\otimes\bar F$ on the geometric generic fibre, is necessarily
an isotrivial fibration.
\end{defn}

Since isotriviality can be checked after faithfully flat base change,
we can restrict to smooth projective degenerations with trivial generic fibres.
Furthermore, since the group $\underline\Aut(G/P)$
is smooth, isotriviality is equivalent to the central fibre being isomorphic to $G/P$.

It the following we will fix a cominuscule variety $G/P$ and a smooth projective
degeneration $X \to S$ with trivial generic fibre. Denote with $s_1$, resp. $s_0$,
the generic, resp. special, point of $S$. Let $X_1 = s_1 \times_S X$, $X_0 = s_0\times_SX$.
Recall that the Picard group of $G/P$ is generated by an ample invertible sheaf $\OO_{G/P}(1)$.
We let $\OO_X(1)$ be the unique extension of $\OO_{G/P}(1) \otimes F$ to an invertible sheaf
on $X$, and $\OO_{X_0}(1)$ its restriction to the central fibre. Note that the restriction map
$$ \Pic X \to \Pic X_1 \simeq \mathbb{Z}$$
is an isomorphism, so that $\OO_X(1)$ is ample by projectivity of $X\to S$.
We also let $$\sM \subset \underline\Hom^n_{S,\bir}(\PP^1_S,X)$$
be the closed subscheme, flat over $S$,
such that $s_1\times_S\sM \subset s_1\times \underline\Hom^n_\bir(\PP^1,G/P)$ is the
component of lines on $G/P$. The fibre $$s_0\times_S \sM \subset \underline\Hom^n_\bir(\PP^1,X_0)$$
is connected, but it may be in general reducible. Since every component
is an irreducible family
of rational curves of degree $1$ with respect to $\OO_{X_0}(1)$, 
hence 
unsplit
by Lemma \ref{lem:unsplit}, it follows that $s_0\times_S\sM_0$ is proper.
So is then
$\sM_0\to S$.

Let us state some immediate properties of $X_0$.
\begin{lem}
The central fibre $X_0$ is simply-connected, Fano, of Picard number $1$, and chain-connected
by $(s_0\times_S\sM)$-curves.
\end{lem} 
\begin{proof}
Every finite \'etale cover $\tilde X_0 \to X_0$ deforms to a finite \'etale cover $\tilde X \to X$,
restricting to $\tilde X_1 \to X_1$ at the generic fibre.
By simply-connectedness of $G/P$, $\tilde X_1 \otimes \bar F \to X_1 \otimes \bar F$
is trivial. After a finite separable 
base change $T \to S$ it follows that $T\times_S\tilde X \to T\times_SX$
is trivial, and thus so is the restriction $\tilde X_0 \to X_0$ over a closed point of $T$
above $s_0$.
Hence $X_0$ is simply-connected.

The relative anticanonical sheaf
$\omega^{-1}_{X/S}$ is isomorphic to $\OO_X(\ind(G/P))$, hence ample, so that in particular
$X_0$ is Fano. 
Since $\sM_0 \to S$ is proper, $\sM_2^i \to X\times_SX$ is surjective for $i$ as in 
Proposition \ref{pro:cominuscule-chains}.
Hence $X_0$ is chain-connected by
$(s_0\times_S\sM)$-curves, and furthermore of Picard number $1$ by Proposition
\ref{pro:picard-number}. 
\end{proof}

\section{Curves at the generic point}
\subsection{Smoothness}
In order to proceed with the proof of rigidity, we first
need to establish smoothness of the family of $\sM$-curves through
the generic point of $X_0$. Equivalently, we check that every
$\sM$-curve through the generic point of $X_0$ is free, a result
that comes for free in characteristic zero (cf.~\cite{kollar}). Let us begin
with a converse, stating that $\sM$ contains all free rational curves of
degree one:
\begin{lem}\label{lem:m-contains-free}
Let $W \subset \underline\Hom_\bir^n(\PP^1,X_0)$ be an irreducible
component of degree $1$ with respect to $\OO_X(1)$, and such that the generic $W$-curve
is free.
Then $W \subset \sM$.
\end{lem}
\begin{proof}
Let $f: \PP^1 \otimes k(W) \to X_0$ be the generic $W$-curve.
By freeness, $H=\underline\Hom_{S,\bir}^n(\PP^1_S,X) \to S$ is smooth at $\eta_W$,
so that in particular $\Spec \OO_{H,\eta_W} \to S$ is faithfully flat. Since $\deg f^*\OO_X(1)=1$,
it follows that $s_1 \times_S \Spec \OO_{H,\eta_W}$ is a family of lines on 
$X_1 \simeq s_1\times(G/P)$ and thus $\Spec \OO_{H,\eta_W} \subset \sM$. Hence $\eta_W \in \sM$
and $W \subset \sM$.  
\end{proof}

Letting $x_1$, resp. $x_0$, be the generic point of $X_1$, resp. $X_0$, observe
that $\OO_S \subset \OO_{X,x_0}$ is an unramified extension of discrete valuation rings,
inducing separably generated extensions of residue and fraction fields.
\begin{lem}\label{lem:m1-flat}
$\Spec \OO_{X,x_0} \times_X \sM_1 \to \Spec \OO_{X,x_0}$ is flat.
\end{lem}
\begin{proof}
Since $\sM \to S$ is flat, and $\sM \to \sM_1$ is a principal bundle,
$\sM_1\to S$ is flat. That is, multiplication by the uniformizer of $\OO_S$ 
is injective in the local rings of $\sM_1$. But 
the uniformizer of $\OO_S$
is also a uniformizer in $\OO_{X,x_0}$. 
\end{proof}

Singular points of $x_0\times_X\sM_1$ can be detected by
inseparability of their residue fields viewed as extensions of $\kappa(x_0)$:
\begin{lem}\label{lem:separable-smooth}
Let $y \in x_0\times_X\sM_1$ be a point such that $\kappa(y) / \kappa(x_0)$ is separably
generated. Then $x_0\times_X\sM_1 \to x_0$ is smooth at $y$.
\end{lem}
\begin{proof}
Let
$f:\PP^1 \otimes\overline{\kappa(y)} \to X_0$ be the corresponding rational curve.
By separability of $\kappa(y)/\kappa(x_0)$, the map $H^0(\PP^1\otimes\overline{\kappa(y)},
f^* T_{X_0}) \to f|_0^* T_{X_0}$ is surjective, so that $f$ is free. Hence
$\sM_1 \to X_0$ is smooth at $y$ (cf.~\cite[Cor. 3.5.4]{kollar}, noting that
$\sM_1$ is the universal $\PP^1$-bundle over $\sM_0$).
\end{proof}
Let now $\bar x_1$, resp. $\bar x_0$, be the geometric generic point of $X_1$, resp. $X_0$.
We will identify a condition on the characteristic $p$ 
ensuring that $\bar x_0 \times_X \sM_1$ is a nonsingular variety. The
idea is to examine the degrees of the components of its singular locus
in a suitable projective embedding.
\begin{lem}\label{lem:m1-singular}
Let
$Z$ be an irreducible component of
the singular locus of $\bar x_0 \times_X \sM_1$. 
Then $\length\OO_{Z,\eta_Z}$ is divisible by $p$.
\end{lem}
\begin{proof}
This is an immediate consequence of Lemma \ref{lem:separable-smooth}.
\end{proof}

The idea of the following crucial lemma is due to Fedor Bogomolov.\footnote{Explained to
me by Jason Starr.}
\begin{lem}\label{lem:singular-length}
Let $Y \subset \PP^N$ be a closed subscheme of degree $e$ and pure dimension $n$. Assume that
the singular locus $Y^\sing$ is zero-dimensional. Then $\length Y^\sing \le e(e-1)^n$.
\end{lem}
\begin{proof}
We claim that there is an $n$-dimensional subspace $V \subset H^0(Y,\OO(e-1))$
such that the base locus of the linear system $|V|$ zero-dimensional and contains $Y^\sing$.
We will then have $$\length Y^\sing \le Y \cdot ((d-1)H)^n = e(e-1)^n,$$
where $H$ is the hyperplane
class in $\PP^N$.

To prove the claim, we first note that
for every closed point $y \in Y \setminus Y^\sing$ there exists a linear projection
$\pi_y : \PP^N \dashrightarrow \PP^{n+1}$ 
such that $\pi_y(Y)$ is a degree $d$ hypersurface,
and $\pi_y$ is an immersion on some open neighbourhood of $y$.
Let $f_y \in H^0(\PP^{n+1},\OO(e))$ be the homogeneous polynomial cutting out $\pi_y(Y)$,
and $f_{y,i}\in H^0(\PP^{n+1},\OO(e-1))$ its partial derivatives.
Let $U \subset H^0(Y,\OO(e-1))$ be the subspace generated by $\pi_y^* f_{y,i}$
for all $0 \le i\le n+1$ and all $y \in Y\setminus Y^\sing$. For each $y \in Y\setminus Y^\sing$
we have that not all $\pi_Y^*f_{y,i}$ vanish at $y$, but all vanish on $Y^\sing$.  
It follows that the base locus of $|U|$
contains $Y^\sing$ and has the same reduced structure. In particular, it is zero-dimensional,
so that we can find an $n$-dimensional subspace $V \subset U$ such that the base-locus of
$|V|$ is zero-dimensional.
\end{proof}

\begin{lem}\label{lem:m1-smooth}
Let $m=\dim (\bar x_1\times_X\sM_1)$.
Suppose there is a degree $e$ projective embedding $\bar x_0 \times_X \sM_1
\hookrightarrow \bar x_0 \times \PP^N$. Assume $p > e(e-1)^m$. Then $\bar x_0 \times_X \sM_1$
is smooth.
\end{lem}
\begin{proof}
Suppose $\bar x_0 \times_X \sM_1 \subset \bar x_0 \times\PP^N$ is not smooth. 
Let $Z$ be an irreducible component
of maximal dimension of the singular locus of $\bar x_0 \times_X \sM_1$. By Bertini's Theorem,
there is a linear subspace $\Lambda \subset \bar x_0 \times_X \PP^N$ with $\codim\Lambda=\dim Z$
such that
$Y = \Lambda \cap (\bar x_0\times_X\sM_1)$
is a degree $e$ subscheme of pure dimension $m - \dim Z$, whose singular locus
is zero-dimensional and contains $\Lambda \cap Z$, a nonempty zero-dimensional subscheme
of length divisible by $p$. It then follows
by Lemma \ref{lem:singular-length} that 
$$0 < \length(\Lambda\cap Z) \le e(e-1)^{m-\dim Z} \le e(e-1)^m < p, $$
a contradiction.
\end{proof}
We now need an expression for the integer $e$ in Lemma \ref{lem:m1-smooth}.
By Proposition \ref{pro:kebekus}, $x_0\times_X \sM_1^\arc = x_0\times_X\sM_1$,
so that $\Spec \OO_{X,x_0}\times_X \sM_1^\arc = \Spec\OO_{X,x_0}\times\sM_1^\arc$,
and we have morphisms
$$ \Spec\OO_{X,x_0}\times_X \sM_1 \to \Spec\OO_{X,x_0}\times_X \Arc_{X/S} \to
\Spec\OO_{X,x_0}\times_X \PP T_{X/S} $$
whose composite, the \emph{tangent map}, is a finite morphism, and an isomorphism over $X_1$.  
\begin{lem}\label{lem:degree-m1}
Suppose $\OO_X(d)$ is very ample on $X$. 
Let $\tau : x_0\times_X\sM_1 \to \PP{T_{X_0,x_0}}$
be the tangent map over $X_0$. Then
$\tau^* \OO_{\PP T_{X_0,x_0}}\left(\frac{d(d+1)}{2}\right)$ is very ample on
$x_0\times_X \sM_1$.
\end{lem}
\begin{proof}
Consider the projective
embedding $\iota : X_0 \to \PP^N = \PP H^0(X_0,\OO_{X_0}(1))$
defined by the complete linear system $|\OO_{X_0}(1)|$.
Let $\sN \subset \underline\Hom_{\bir}^n(\PP^1,\PP^N)$ be the component
of degree $d$ with respect to $\OO_{\PP^N}(1)$, so that there is a natural commutative
diagram
$$\begin{CD}
x_0 \times_X \sM_1 @>{\tau}>> x_0 \times_X \PP T_{X_0} @>>> x_0 \\
@VVV @VVV @VVV \\
\sN_1^\arc @>{\theta}>> \PP T_{\PP^N} @>>> \PP^N
\end{CD}$$
where the vertical arrows are induced by $\iota$,
and $\theta$ is the tangent morphism. By linearity of $\iota_*:\PP T_{X_0,x_0}
\to \PP T_{\PP^N,\iota(x_0)}$
and $\Aut(\PP^N)$-equivariance of $\sN_1$, it will be enough to
check that, for a point $q \in \PP^N(k)$, the pullback
$\theta^*\OO_{\PP T_{\PP^N,q}}(d(d+1)/2)$ is very ample on $q\times_{\PP^N}
\sN_1^\arc$.

Let $\fm$ be the maximal ideal in $\OO_{\PP^N,q}$, and define
$$ 
D = \Spec \OO_{\PP^N,q} / \fm^{d+1},
\quad
P = \Spec k[t] / (t^{d+1}),
\quad
A = \underline\Imm(P,D;0, q)/ \Aut(P,0) $$
where $0 \in P(k)$. There are natural morphisms
$$ q \times_{\PP^N} \Arc_{\PP^N} \to A \to \PP T_{\PP^N,q} $$
where the tangent map $A \to \PP T_{\PP^N,q}$ is a Zariski-locally
trivial affine space bundle, and can be viewed as a quotient
of $q\times_{\PP^N}\Arc_{\PP^N}$ parametrising $d$-th jets of immersed
arcs through $q \in \PP^N$. Since $\sN$ parametrises rational curves
of degree $d$ on $\PP^N$, the composite
$$ q\times_{\PP^N}\sN^\arc_1 \xrightarrow{\vartheta} q\times_{\PP^N}\Arc_{\PP^N} \to A $$
is a locally closed immersion, factoring $\theta$, so that it will now be
enough to show that $\vartheta^*\OO_{\PP T_{\PP^N,q}}(d(d+1)/2)$
is very ample on $A$.

Let $T \subset A \times D$ be the universal family over $A$, so that
$\OO_T$ is a sheaf of infinitesimal extensions of $\OO_A$. The evaluation
morphism $T \to D$ induces an epimorphism
$$ \OO_A \otimes \OO_{\PP^N,q}/\fm^{d+1} \to \OO_T \to 0 $$
of filtered $\OO_A$-algebras, with filtrations induced by
$\fm$ and the ideal sheaf $\sI \subset \OO_T$ of the zero-section $A\to T$.
Since $\sI$ is a locally free $\OO_A$-module, the map 
$$ \OO_A \otimes \fm/\fm^{d+1} \to \sI \to 0 $$
induces a morphism
$$ f : A \to \Gr(\dim(\fm/\fm^{d+1}) -  d, \dim(\fm/\fm^{d+1})) $$
factoring the natural immersion $A \to \Hilb_D$.
It follows that $\det \sI$ is very ample on $A$.
Since $A$ is an affine space bundle over $\PP T_{\PP^N,q}$,
it follows that $\vartheta^* : \Pic \PP T_{\PP^N,q} \to \Pic A$ is an isomorphism,
so that $\det\sI = \vartheta^* \OO_{\PP T_{\PP^N,q}}(e)$ for some $e>0$, and
one can determine $e$ by computing the intersection of $c_1(\sI)$ with a curve.
Let $c: \PP^1 \to A$ be a rational curve such that $\vartheta\circ c$
is a line in $\PP T_{\PP^N,q}$. Then $$c^*\sI \simeq \OO(1) \oplus \dots \oplus \OO(d)$$
so that $$e = \deg c^*\det\sI = \frac{d(d+1)}{2}.$$
Hence finally
$\vartheta^*\OO_{\PP T_{\PP^N,q}}(\frac{d(d+1)}{2})$ is very ample on $A$.
\end{proof}

We have thus arrived at the following condition for smoothness
of $\bar x_0\times_X\sM_1$.
\begin{pro}\label{pro:m1-smooth}
Let $m=\dim (\bar x_1\times_X\sM_1)$ and
let $\delta$ be the degree of $\bar x_1\times_X\sM_1$ with respect to
the embedding in $\bar x_1 \times_X \PP T_{X/S}$.
Suppose $\OO_X(d)$ is very ample on $X$.
Assume $$p > \left(\frac{d(d+1)}{2}\right)^m \delta \left(\left(
	\frac{d(d+1)}{2}
\right)^m\delta-1\right)^m.$$
Then $\bar x_0 \times_X \sM_1$
is smooth.
\end{pro}
\begin{proof}
By Lemma \ref{lem:degree-m1}, the pullback of $\OO_{\PP T_{X/S}}(d(d+1)/2)$ 
by the tangent map
is very ample on $x_0 \times_X \sM_1$, and thus on $\Spec \OO_{X_0,x_0}\times_X\sM_1$.
Then, by Lemma \ref{lem:m1-flat}, the degree of the corresponding projective embedding
of $\bar x_0\times_X\sM_1$ is
$e=\left(\frac{d(d+1)}{2}\right)^m\delta$. Applying 
Lemma \ref{lem:m1-smooth}, the claim follows.
\end{proof}

\subsection{Segre case}
The preceding subsection gives a condition under which
the space of $\sM$-curves through the generic point of $X_0$
is a smooth degeneration of the space of $\sM$-curves through the
generic point of
$X_1$. As indicated in Corollary \ref{cor:cominuscule-lines},
in case of $G$ being of type $A_n$, the tangent map indentifies
the space of $\sM$-curves through
any point of $X_1$ with a Segre subvariety of
the projectivised tangent space. Smooth degenerations
in such situation are described by the following.
\begin{lem}\label{lem:segre}
Let $f:Y \to S$ be a smooth morphism whose geometric generic
fibre is isomorphic to $(\PP^a \times\PP^b) \otimes \bar F$.
Suppose $f$ factors through a finite morphism
$\varphi:Y \to \PP_S^{ab+a+b}$ whose restriction
to the geometric generic fibre of $f$ is a Segre embedding.
Then $f$ is an isotrivial fibration.
\end{lem}
\begin{proof}
Set $Y_1 = s_1\times_SY$, $Y_0=s_0\times_SY$. After faithfully flat base change,
we can assume that $Y_1 \simeq (\PP^a\times\PP^b)\otimes F$. Let
$$ P \subset \Hilb_{Y/S} $$
be the closed subscheme, flat over $S$, such that $P_1=s_1 \times_S P \simeq \PP^a_F$
is the component parametrising subspaces of the form $\{*\}\times \PP^b \subset Y_1$.
Let $\Lambda_P \subset P \times_S Y$ 
be the universal family, so that $s_1\times_S\Lambda_P \simeq P_1 \times \PP^b$.
We will show that
$\Lambda_P \to P$ is an  isotrivial bundle
of projective spaces. By flatness of the universal family, 
and by smoothness of $\PGL_{b+1}$,
it will be enough to check that for each closed point $p \in P_0 = s_0\times_SP$,
the fibre $p\times_P\Lambda_P$ is isomorphic to $\PP^b$. Given $p\in P_0$,
there is a finite flat base change $T \to S$ to the spectrum of a discrete
valuation ring with closed point $t_0$ and generic point $t_1$, 
together with a section $\sigma : T \to T\times_SP$ such that $\sigma(t_0)=p$.
The composite
$$ \sigma|_{t_1}^* \Lambda_P \to t_1\times_S Y \xrightarrow{\varphi} t_1\times\PP^{ab+a+b}$$
is a family of linear subspaces over $t_1$, thus defining a $t_1$-point
of the appropriate Grassmannian. By properness of the Grassmannian,
the $t_1$-point extends to a $T$-point, and thus defines a family
$$ \bar \Lambda_P \subset \PP^{ab+a+b}_T $$
of linear subspaces, flat over $T$. In fact,
$\bar\Lambda_P \simeq\PP^b_T$.
The
inclusion $t_1\times_T \bar\Lambda_P \subset t_1\times_S Y$ 
extends to a rational map $i : \bar\Lambda_P\dashrightarrow Y$
such that $\varphi\circ i$ is the identity on $\bar\Lambda_P$. Since
$\varphi$ is finite, it follows that $i$ is in fact regular, hence
a closed immersion. Thus $\sigma^*\Lambda_P$ and $\bar\Lambda_P$
are closed subschemes of $T\times_SY$, flat over $T$, and
with identical restrictions over $t_1$ -- hence $\sigma^*\Lambda_P = \bar\Lambda_P$,
and in particular $p^*\Lambda_P \simeq \PP^b$.

We have thus defined the subscheme $P \subset \Hilb_{Y/S}$
such that the universal family $\Lambda_P \to P$ is an isotrivial
$\PP^b$-bundle. Symmetrically, we define a closed
subscheme $Q \subset \Hilb_{Y/S}$, flat over $S$, such that
$Q_1 = s_1\times_SQ \simeq \PP^b_F$ is the component
parametrising subspaces of the form $\PP^a \times \{*\} \subset Y_1$,
and with the universal family $\Lambda_Q \to Q$ being an
isotrivial $\PP^a$-bundle.
We now claim that the arrows in the projection diagram
$$ P\times_SQ \leftarrow \Lambda_P \times_Y \Lambda_Q \to Y $$
are  isomorphisms. Since they do become isomorphisms
after restriction over $s_1$, it will be enough to check
that their restrictions over $s_0$ are bijective on closed points.
Let us refer to the closed fibres of $s_0 \times_S \Lambda_P \to s_0\times_SP$
(resp. $s_0\times_S \Lambda_Q \to s_0\times_SQ$) as $P$-planes
(resp. $Q$-planes) in $Y_0$. We then need to check that all
$P$-planes (resp. $Q$-planes) are disjoint, and that
a $P$-plane intersects a $Q$-plane in a single point.
Using smoothness of $Y$, this follows by intersection theory,
specialising relevant classes from the generic fibre.
Finally, we have $Y \simeq P\times_SQ$, where the
double fibration
$ P \leftarrow P\times_SQ \rightarrow Q $
is a pair of bundles of projective spaces. It thus follows that
$P \simeq \PP^a_S$, $Q \simeq \PP^b_S$, so that $Y \simeq \PP^a_S\times_S\PP^b_S$.
\end{proof}

\subsection{General case}

Recall the description of $\bar x_1 \times_X \sM_1$ given in Proposition
\ref{pro:cominuscule-lines}, Corollary \ref{cor:cominuscule-lines} and Table
\ref{tab:vmrt}. It is isomorphic to $\bar x_1 \times V$, where $V$ is the variety
of line tangents at the origin of $G/P$: either a Segre variety, or degree $2$
Veronese, or a minimally embedded cominuscule variety. Combining
the results of the two preceding subsections, we are ready to state
a condition on the characteristic $p$ ensuring that
the space of $\sM$-curves through the generic point
does not degenerate.

\begin{pro}\label{pro:rigidity-m1}
Let $V$ be the variety of line tangents at the origin of $G/P$.
Suppose $\OO_X(d)$ is very ample on $X$,
and assume one of the following holds.
\begin{enumerate}
\item $G$ is of type $A_n$, $P$ is associated to $\alpha_i$, and 
$$p>\left(\frac{d(d+1)}{2}\right)^{n-1}{n-1\choose i-1}
\left(\left(\frac{d(d+1)}{2}\right)^{n-1}{n-1\choose i-1}-1\right)^{n-1}$$
\item $G$ is of type $C_n$ and 
$$p>\left(d(d+1)\right)^{n-1} 
\left(\left(d(d+1)\right)^{n-1}-1\right)^{n-1}$$
\item $G$ is of one of remaning types, $V$ is $\frac{d(d+1)}{2}$-rigid, and
$$p>\left(\frac{d(d+1)}{2}\right)^m(\deg V)(
\left(\frac{d(d+1)}{2}\right)^{m}(\deg V-1)^m$$ where $m=\dim V$ and
$\deg V$ is the degree of $V$ embedded in the projectivised
tangent space at the origin of $G/P$.
\end{enumerate}
Then $\Spec \OO_{X,x_0}\times_X \sM_1 \to \Spec \OO_{X,x_0}$ is an isotrivial $V$-bundle.
\end{pro}
\begin{proof}
We first use Proposition \ref{pro:m1-smooth} to conclude smoothness of $\bar x_0 \times_X \sM_1$,
and thus -- by Lemma \ref{lem:m1-flat} -- of $\Spec \OO_{X,x_0}\times_X \sM_1 \to \Spec\OO_{X,x_0}$.
The conditions on $p$ correspond precisely to the hypothesis of Proposition
\ref{pro:m1-smooth}, where in cases (1) and (2) we use well-known expressions for degrees
of Segre and Veronese varieties. 

Then, to check isotriviality, after a faithfully flat base-change and
replacing $k$ with $\kappa(\bar x_0)$, we may replace
$\Spec\OO_{X,x_0}$ with $S$. We thus obtain a smooth projective
morphism $Y \to S$ whose generic fibre is isomorphic to 
a base-change of $V$.
Now, depending on the type of $G$, we use:
\begin{enumerate}
\item Lemma \ref{lem:segre} with the finite morphism $Y \to \PP^{\dim(G/P)-1}_S$ induced
by the tangent map;
\item well-known rigidity of $\PP^{n-1}$;
\item $\frac{d(d+1)}{2}$-rigidity of $V$, where $\OO_Y(d(d+1)/2)$ is very ample
by Lemma \ref{lem:degree-m1}.\qedhere
\end{enumerate}
\end{proof}

\section{Arcs at the generic point}
\subsection{Minimality}
We have so far described, under suitable conditions, the abstract
space of $\sM$-curves through the generic point, concluding that
it does not degenerate in the central fibre. In order
to apply the main Theorem of Chapter \ref{chap:extension}, we
need a similar result for the embedding into the space of arcs.
This will be achieved in three steps: we first check that,
after discarding non-dominant components of $s_0\times_X\sM$,
we are left with an irreducible component whose generic member
is minimal; next, we obtain a condition on the characteristic $p$
ensuring that the corresponding variety of rational tangents
at the generic point of $X_0$ is linearly nondegenerate; finally,
we check that the corresponding family of arcs through a formal neighbourhood
of the generic point is isomorphic to that on $G/P$.

\begin{lem}\label{lem:m-star}
Assume $\bar x_0\times_X\sM_1$ is smooth.
Then there is a unique irreducible component $\sM_*$ of $s_0\times_X\sM$
such that $\sM_*$ is a dominating family of rational curves on $X_0$.
Furthermore, 
$\sM_*$ is an irreducible component of $\underline\Hom_\bir^n(\PP^1,X_0)$,
and
the generic $\sM_*$-curve is free.
\end{lem}
\begin{proof}
Existence and uniqueness of $\sM_*$ follows
from smoothness, and thus irreducibility of $x_0\times_X\sM_1$.
By Lemma \ref{lem:m-contains-free}, $M$ contains
every irreducible component of $\underline\Hom_\bir^n(\PP^1,X_0)$ whose
generic member is free, hence is a component itself.
Since $x_0\times_X\sM_{*,1} = x_0\times_X\sM_1$,
the generic $\sM_*$-curve is free.
\end{proof}

In the remainder of this section we will assume
the hypotheses of Proposition \ref{pro:rigidity-m1},
so that in particular $\bar x_0\times_X\sM_1$ is smooth
and Lemma \ref{lem:m-star} applies. We use our standard notation
$\sM_{*,1}$, $\sM_{*,2}^i$, etc.
Note that by Lemma \ref{lem:free-chains}, $\sM_{*,2}^{i,\free} \to X_0\times X_0$
is dominant for some $i\ge0$, so that $X_0$ is chain-connected by $\sM_*$-curves. 

By smoothness of $\bar x_0\times_{X_0}\sM_{*,1}$, the generic $\sM_*$-curve is free.
In fact, our hypotheses on $p$ imply more.
\begin{lem}\label{lem:m-star-minimal}
The generic $\sM_*$-curve is minimal.
\end{lem}
\begin{proof}
Let $(V,\sL)$ be the variety of line tangents
at the origin $o$ in $G/P$, together with the very ample
invertible sheaf defining the embedding 
$V \subset \PP T_{G/P,o}$ into the projectivised
tangent space. More concretely, for $G$ of type $A_n$, $V$ is a Segre
variety with $\sL = \OO(1,1)$; for $G$ of type $C_n$, $V$ is a Veronese
variety with $\sL = \OO(2)$; for remaining types, $V$ is cominuscule
with $\sL = \OO(1)$.
By Proposition \ref{pro:rigidity-m1}, we 
have a commutative diagram with Cartesian squares
$$\begin{CD}
V\otimes\kappa(\bar x_0) 
@>>> \Spec \OO_{X,x_0}\times_X\sM_1
@<<< V\otimes\kappa(\bar x_1) \\
@V{\mathfrak{d}}VV @VVV @V{|\sL\otimes\kappa(\bar x_1)|}VV \\
\bar x_0\times_X\PP T_{X/S} @>>> \Spec\OO_{X,x_0}\times_X\PP T_{X/S} @<<<
\bar x_1\times_X\PP T_{X/S}
\end{CD}$$
where $\mathfrak{d} \subset |\sL\otimes\kappa(\bar x_0)|$ is a linear
subsystem. In particular, it follows that
the tangent morphism
$ \bar x_0\times_X \sM_{*,1} \to \bar x_0^*\PP T_{X_0} $
factors as
$$\bar x_0\times_X \sM_{*,1} \simeq V \otimes \kappa(\bar x_0)
\xrightarrow{|\sL|} \PP T_{G/P,o} \otimes \kappa(\bar x_0)
\dashrightarrow \bar x_0^* \PP T_{X/S} 
$$
where the rightmost arrow is a linear projection onto a subspace.
Now, the hypotheses of Proposition
\ref{pro:rigidity-m1} 
ensure that the degree of the embedding defined by $|\sL|$
is less than $p$. It follows that the tangent morphism, being finite,
is generically unramified, so that the generic $\sM_*$-curve
$f: \PP^1 \otimes k(\sM_*)\to X_0$
does not admit nontrivial infinitesimal deformations
fixing $f(0)$ and the tangent direction
in $f|_0^* \PP T_{X_0}$.
Hence $f$ is minimal.
\end{proof}

\subsection{Linear nondegeneracy}

We are going to derive linear nondegeneracy of the variety
of $\sM_*$-rational tangents at $x_0$ from Proposition \ref{pro:linear-nondegeneracy}.
That will require $X_0$ to be \emph{separably} connected by chains of free $\sM_*$-curves,
a condition we shall satisfy by constructing a projective embedding
of a suitable blow-down of the generic fibre of $\sM_{*,2}^{i,\free} \to X_0\times X_0$,
and imposing its degree as another bound on the characteristic $p$.
As a consequence, it will follow that the space of $\sM$-curves
through the generic point of $X_0$, the corresponding variety of rational tangents,
and the tangent map between the two, are isomorphic to those on $G/P$.
 
Recall that by Lemma \ref{lem:free-chains}, $\sM_{*,2}^{i,\free} \to X_0\times X_0$
is dominant for some $i\ge0$. 
Letting $\xi_1$, resp. $\xi_0$, be the generic point of $X_1 \times_{s_1}X_1$, resp.
 $X_0\times X_0$, observe that
$\OO_S \subset \OO_{(X/S)^2,\xi_0}$ is an unramified extension of discrete valuation
rings, inducing separably generated extensions of residue and fraction fields.

\begin{lem}\label{lem:m2-flat}
Let $r_0$ be the smallest $i>0$ such that $\sM_{*,2}^{i,\free} \to X_0\times X_0$ is dominant.
Then: 
\begin{enumerate}
\item $r_0 \le \dim (G/P)$;
\item The total evaluation morphism
$\xi_0\times_{X_0^2}\sM_{*,2}^{r_0,\free} \to \xi_0\times_{X_0^2} X_0^{r_0+1}$
is quasi-finite.
\item $\xi_0 \times_{X_0^2} \sM_{*,2}^{r_0,\free}$ 
is a dense open subscheme of an irreducible component of
$\xi_0\times_{(X/S)^2}\sM_2^{r_0}$;
\end{enumerate}
\end{lem}
\begin{proof}
\begin{enumerate}
\item
Let $n = \dim (G/P) =\dim X_0$. Denote by $d_i$
the dimension of the closed image of $\sM_{*,2}^{i,\free} \to X_0\times X_0$.
We then have that the sequence $d_i$ is nondecreasing, $d_0 = n$, and $d_{i+1}=d_i$ if and only if
$d_i = 2n$. We then have $d_{r_0}=d_n = 2n$, so that $r_0 \le n$.
\item
Suppose not, so that there is a point
$\vec{x} \in \xi_0\times_{X_0^2}X_0^{r_0+1}$
with a positive-dimensional fibre in $\xi_0\times_{X_0^2}\sM_{*,2}^{r_0,\free}$.
By Bend-and-Break, this can only happen if $\vec{x}$ factors through
one of the diagonals in $X_0^{r_0+1}$. It follows that
a chain corresponding
to a point in the fibre contains a segment whose two marked points coincide. 
Removing the segment, we obtain a chain of length
$r_0-1$, corresponding to a point
in $\xi_0\times_{X_0^2}\sM_{*,2}^{r_0-1,\free}$. This contradicts minimality of $r_0$.
\item
Since $\sM_*^\free$ is open $\sM_*$, we have
that $\xi_0\times_{X_0^2}\sM_{*,2}^{r_0,\free}$ is open
in
$\xi_0\times_{X_0^2}\sM_{*,2}^{r_0}$.
Since $\bar x_0 \times_{X_0} \sM_{*,1}$ is smooth and connected,
we have by Lemma \ref{lem:free-stuff} that $\xi_0\times_{X_0^2}\sM_{*,2}^{r_0,\free}$
is irreducible, and thus its closure in $\xi_0\times_{X_0^2}\sM_{*,2}^{r_0}$ is an irreducible
component.
Finally, $\xi_0\times_{X_0^2}\sM_{*,2}^{r_0}$ is a union
of irreducible components
of
$\xi_0\times_{(X/S)^2} \sM_2^{r_0}$. 
\end{enumerate}\end{proof}

\begin{lem}\label{lem:m2-separable}
Let $V$ be the variety of line tangents
at the origin of $G/P$, and assume $p>\delta_{G/P}(r_0)$ (cf. \ref{ss:numerical}). 
Then $\sM_{*,2}^{r_0,\free} \to X_0\times X_0$ is separably dominant.
\end{lem}
\begin{proof}
Let $\bar\sM_2^{r_0} \subset \Spec \OO_{(X/S)^2,\xi_0} \times_{(X/S)^2}\sM_2^{r_0}$ 
be the flat limit
of $\xi_1\times_{(X/S)^2}\sM_2^{r_0}$ over $\Spec\OO_{(X/S)^2,\xi_0}$.
Note that $\xi_0\times_{(X/S)^2}\bar\sM_2^{r_0}$ contains
$\xi_0\times_{X_0^2}\sM_{*,2}^{r_0,\free}$.
Consider the total evaluation morphism
$$ \nu:\bar\sM_2^{r_0} \to 
\Spec \OO_{{X/S}^2,\xi_0} \times_{(X/S)^2} (X/S)^{r_0+1}. $$
Letting $p_j : (X/S)^{r_0+1} \to X$, $0\le j\le r_0$, be the natural projections,
we have an invertible sheaf
$$\sL = \bigotimes_{j=1}^{r_0-1} \nu^{*}p_j^* \OO_X(1)$$
on $\Spec\OO_{(X/S)^2,\xi_0}\times_{(X/S)^2}\sM_2^{r_0}$.
By the definition in \ref{ss:numerical},
$$
(\bar\xi_1 \times_{(X/S)^2} \sM_2^{r_0}) \cdot c_1(\sL)^{\dim \bar\xi_1^*\sM_2^{r_0}}
 = \delta_{G/P}(r_0)
$$
holds on the geometric generic fibre,
so that
$$
(\bar\xi_0\times_{(X/S)^2} \bar\sM_2^{r_0}) \cdot c_1(\sL)^{\dim\bar\xi_0^*\sM_2^{r_0}} = 
\delta_{G/P}(r_0)
$$
holds on the geomteric special fibre.

Let now $W \subset \xi_0\times_{(X/S)^2}\bar\sM_2^{r_0}$ be the closure
of $\xi_0\times_{X_0^2}\sM_{*,2}^{r_0,\free}$, an irreducible component.
Suppose $\kappa(\eta_W) / \kappa(\xi_0)$ is not separably generated.
Then local rings of generic points of $\bar\xi_0\times_{\xi_0}W$
have length divisible by $p$, so that
$$(\bar\xi_0\times_{\xi_0}W) \cdot c_1(\sL)^{\dim \bar\xi_0^*W} \in p\mathbb{Z}. $$
By Lemma \ref{lem:m2-flat}, the restriction of $\nu$ to $W$ is generically finite, 
and the above intersection
number is positive.
We then have
$$\delta_{G/P}(r_0) \ge (\bar\xi_0\times_{\xi_0}W) \cdot c_1(\sL)^{\dim \bar\xi_0^*W} \ge p.$$
But $p > \delta_{G/P}(r_0)$, a contradiction.
It follows that $\sM_{*,2}^{r_0,\free} \to X_0\times X_0$, dominant by
Lemma \ref{lem:m2-flat}, is separable.
\end{proof}

\begin{pro}\label{pro:m1-tangent}
Let $(V,\sL)$ be the variety of line tangents
at the origin of $G/P$, and the invertible sheaf
defining the projective embedding
into the projectivised tangent space at the origin.
Assume the hypotheses of Proposition \ref{pro:rigidity-m1},
and furthermore $p>\delta_{G/P}(\dim (G/P))$ and $p>\ind(G/P)$. 
Then the tangent morphism
$$ \Spec \OO_{X,x_0} \times_X \sM_1 \to \Spec \OO_{X,x_0}\times_X \PP T_{X/S}, $$
a map between an isotrivial $V$-bundle and a trivial projective space
bundle, is defined by the complete
linear system associated with $\sL$.
\end{pro}
\begin{proof}
Recall the factorisation 
$$\bar x_0\times_X \sM_{*,1} \simeq V \otimes \kappa(\bar x_0)
\xrightarrow{|\sL|} \PP T_{G/P,o} \otimes \kappa(\bar x_0)
\dashrightarrow \bar x_0\times_X \PP T_{X/S} 
$$
of the tangent morphism $\bar x_0\times_X\sM_{*,1}\to \bar x_0\times_X\PP T_{X_0}$
(cf. the proof of Lemma \ref{lem:m-star-minimal}).
It will be enough to show that the linear projection corresponding
to the dashed arrow
is an isomorphism, i.e. that the image of the tangent morphism
is linearly nondegenerate in $\bar x_0^*\PP T_{X/S}$.
We apply Proppsition \ref{pro:linear-nondegeneracy} to $X_0$
and $\sM_*$. Note that $\ind(X_0)=\ind(G/P)$.
Hypothesis (1) is satisfied by Lemma \ref{lem:m-star-minimal}.
Hypothesis (2) is satisfied by smoothness of $\bar x_0 \times_{X_0}\sM_{*,1}$.
Hypothesis (3) is satisfied by Lemma \ref{lem:m2-separable}. Hypothesis (4) is satisfied
by Lemma \ref{lem:cominuscule-wedge2}, using the above factorisation.
\end{proof}

\subsection{Identification with model}

We now wish to extend Proposition \ref{pro:m1-tangent} to
a statement about the family of arcs through a formal neighbourhood
of the generic point, essentially saying that this family cannot
acquire `curvature' when specialising to $X_0$. The extension theorem of
Chapter \ref{chap:extension} will then apply immediately. 
 
It will be convenient to introduce the following notation:
$$ 
Y = G/P,\quad y \in Y\ \textrm{origin},\quad \hat Y = y\times_YY^\sharp\ \textrm{completion at $y$}
$$
$$
 \sN \subset\underline\Hom_\bir^n(\PP^1,Y)\ \textrm{lines},\quad
V = y\times_Y \sN_1 \hookrightarrow \PP T_{Y,y}\ \textrm{line tangents at $y$}. 
$$
In addition to the hypotheses of Proposition \ref{pro:m1-smooth},
we now assume those of Proposition \ref{pro:m1-tangent},
so that $$\Spec\OO_{X,x_0}\times_X \sM_1 \to \Spec\OO_{X,x_0}\times_X \PP T_{X/S}$$
is a closed immersion, \'etale locally isomorphic to the embedding
$$V \to \PP T_{Y,y}.$$
\begin{lem}\label{lem:v-trivial}
There is a faithfully flat base-change $T \to \Spec\OO_{X,x_0}$ such that
$T$ is the spectrum of a discrete valuation ring over $\kappa(\bar x_0)$,
with the latter as its residue field,
and there is a commutative diagram
$$\begin{CD}
T\times_X\sM_1 @>>> T\times_X \PP T_{X/S} \\
@VV{\simeq}V @VV{\simeq}V \\
T\times V @>>> T\times \PP T_{Y,y}
\end{CD}$$
where the vertical arrows are isomorphisms, and the horizontal
arrows are the natural tangent embeddings.
\end{lem}
\begin{proof}
Immediate by the preceding paragraph.
\end{proof}

Since $\Spec\OO_{X,x_0}\times_X\sM_1
= \Spec\OO_{X,x_0}\times_X\sM_1^\arc$ (Proposition \ref{pro:kebekus}),
it follows that the tangent map, a closed immersion, 
factors through the arc space:
$$\begin{diagram}
\node{} \node{\Spec\OO_{X,x_0}\times_X\Arc_{X/S}}\arrow{s} \\
\node{\Spec\OO_{X,x_0}\times_X\sM_1}\arrow{e}\arrow{ne}
\node{\Spec\OO_{X,x_0}\times_X\PP T_{X/S}}
\end{diagram}$$
where the diagonal arrow is an isomorphism onto $\Spec\OO_{X,x_0}\times_X
\hat\sM_1$.
Since $\Spec\OO_{X,x_0}\times_X\sM_1$ is flat over
$\Spec\OO_{X,x_0}$, the above diagram defines a section
$$ \sigma_{X} : \Spec\OO_{X,x_0} \to \Spec\OO_{X,x_0}\times_X \ArcHilb_{X/S}. $$
Note that $\sN_1$ defines a corresponding section
$\sigma_{Y} : Y \to \ArcHilb_{Y}$.
\begin{pro}
\label{pro:v-trivial}
There is a faithfully flat base-change $T \to \Spec\OO_{X,x_0}$ such that
$T$ is the spectrum of a discrete valuation ring over $\kappa(\bar x_0)$,
with the latter as its residue field,
and there is a commutative diagram
$$\begin{CD}
T\times_X(X/S)^\sharp @>{\id_T \times \sigma_X}>> T\times_X(X/S)^\sharp\times_X\ArcHilb_{X/S} \\
@VV{\simeq}V @VV{\simeq}V \\
T\times \hat Y @>{\id_T\times \sigma_Y}>> T\times \hat Y\times_{Y}\ArcHilb_{Y}
\end{CD}$$
where the vertical arrows are isomorphisms.
\end{pro}
\begin{proof}
Let $T$ be as in Lemma \ref{lem:v-trivial}, together with the isomorphism
$$\bar\phi : T\times_X \PP T_{X/S} \to T\times \PP T_{Y,y}$$
identifying pullbacks of $\sM_1$ and $V$.
Let
$$\phi : T\times_X (X/S)^\sharp \to T\times \hat Y$$ be any
isomorphism of bundles of formal discs whose restriction 
to the first infinitesimal neighbourhood
of $T$ in $T\times_X(X/S)^\sharp$, viewed as an isomorphism
of pullbacks of tangent bundles, projectivises to $\bar\phi$.
There is a natural lift of $\phi$ to a commutative diagram
$$\begin{CD}
T\times_X(X/S)^\sharp\times_X\ArcHilb_{X/S} @>{\tilde\phi}>>
T\times \hat Y\times_{Y}\ArcHilb_{Y} \\
@VVV @VVV \\
T\times_X(X/S)^\sharp @>{\phi}>> T\times \hat Y
\end{CD}$$
where both horizontal arrows are isomorphisms. The section $\sigma_X$
then induces a commutative diagram
$$\begin{CD}
T\times_X(X/S)^\sharp\times_X\ArcHilb_{X/S} @>{\tilde\phi}>>
T\times \hat Y\times_{Y}\ArcHilb_{Y} \\
@A{\id_T\times\sigma_X}AA @AA{\sigma_X^\phi}A \\
T\times_X(X/S)^\sharp @>{\phi}>> T\times \hat Y
\end{CD}$$
where both vertical arrows are sections. We now have a pair of sections
$$
T \times \hat Y
\overset{\sigma_X^\phi}{\underset{\id_T\times\sigma_{Y}}{\rightrightarrows}}
T\times \hat Y\times_{Y}\ArcHilb_{Y}
$$
thus defining a pair morphisms
$$
T \overset{t_X}{\underset{t_{Y}}{\rightrightarrows}} \prod(\ArcHilb_{\hat Y}/\hat Y)
$$
where $t_{Y}$ is constant, i.e. factors through a $k$-point, which
we will denote with the same symbol.
By construction, the restriction of $t_X$ to the generic point of $T$
factors through the $\sR_u\underline\Aut(\hat Y,y)$-orbit of $t_{Y}$.
Hence, by Proposition \ref{pro:closed-orbits}, $t_X$ itself factors
through the same orbit. It follows that, possibly after a further
faithfully flat base change, there is
$g : T \to \sR_u\underline\Aut(\hat Y,y)$ such that
$g\phi$ gives the desired isomorphism.
\end{proof}

\begin{cor}\label{cor:x-is-y}
Under the hypotheses of Propositions \ref{pro:m1-smooth} and \ref{pro:rigidity-m1}, $X_0$
is isomorphic to $G/P$.
\end{cor}
\begin{proof}
By the above Proposition, there is an isomorphism $\bar x_0^*X_0^\sharp \to \bar x_0
\times \hat Y$ identifying $\bar x_0^*X_0^\sharp\times_X\hat\sM_1$ with
$\bar x_0\times \hat Y\times_Y\hat\sN_1$. Let $\bar y_0 : \Spec \kappa(\bar x_0) \to Y$
be the geometric point factoring through $y$. Note that $\sM_*$ and $\sN$
are \emph{good families} in the language of Chapter \ref{chap:extension}. Hence,
applying Corollary \ref{cor:extension}
to the pair $(X_0,\sM_*)$, $(Y,\sN)$, and geometric points
$\bar x_0$, $\bar y_0$, we have an isomorphism $X \simeq Y$. 
\end{proof}

\section{Conclusion}
Corollary \ref{cor:x-is-y} essentially concludes the proof of a rigidity
theorem for $G/P$, establishing its $d$-rigidity under the hypotheses 
we have been gradually introducing. Combining these, and abstracting
from the particular situation $X \to S$, we can restate the main result of this chapter as follows. 

\begin{thm}\label{thm:rigidity}
Let $G/P$ be a cominuscule homogeneous variety, and $d>0$ an integer. Assume one of the following
holds:
\begin{enumerate}
\item
$G$ is of type $A_n$, $P$ is associated to $\alpha_i$, and
\begin{align*}
 p > \max\{&
\left(\frac{d(d+1)}{2}\right)^{n-1}{n-1 \choose  i-1}
\left(\left(\frac{d(d+1)}{2}\right)^{n-1}{n-1 \choose  i-1}-1\right)^{n-1},
\\
& n+1,\ \delta_{G/P}(i(n+1-i)) \}
\end{align*}
\item
$G$ is of type $C_n$,
$$ p > \max\{
\left(d(d+1)\right)^{n-1}\left(\left(d(d+1)\right)^{n-1}-1\right)^{n-1},
\
n+1,\ \delta_{G/P}(n(n+1)/2)
\}
$$
\item
$G$ is of one of remaining types,
its variety $V$ of line tangents at the origin is $\frac{d(d+1)}{2}$-rigid, and
$$ p > \max\{
\left(\frac{d(d+1)}{2}\right)^{m}\delta_V\left(\left(\frac{d(d+1)}{2}\right)^m\delta_V-1\right)^m,
\
\ind(G/P),\ \delta_{G/P}(\dim (G/P))
\}$$
where $m=\dim V$.
\end{enumerate}
Then $G/P$ is $d$-rigid.
\end{thm}
\begin{proof}
It is enough to show isotriviality for every smooth projective degeneration $X \to S$
of $G/P$ with trivial generic fibre and very ample $\OO_X(d)$.
In such situation, the hypotheses of the Theorem imply those of Propositions
\ref{pro:m1-smooth} and \ref{pro:rigidity-m1}, so that $X_0 \simeq G/P$ by Corollary
\ref{cor:x-is-y}.
\end{proof}
Note that for $G$ not of type $A_n$, $C_n$, the Theorem derives
$d$-rigidity of $G/P$ from $\frac{d(d+1)}{2}$-rigidity of its
variety of line tangents, a cominuscule homogeneous variety of lower dimension
(i.e. with lower rank of the corresponding simple algebraic group). Given
an integer $d$, this allows one to obtain a lower bound on $p$ guaranteeing
$d$-rigidity by applying the Theorem inductively, eventually terminating
at a cominuscule homogeneous variety for a group of type $A_n$ or $C_n$.\footnote{
Unfortunately, the bounds obtained in subsequent steps of the induction
tend to grow due to replacing $d$ with $\frac{d(d+1)}{2}$, the expression 
appearing in Lemma \ref{lem:degree-m1}. This situation would be greatly
improved if one could prove the Lemma with $d$ in place of the former expression.
An observation due to David Jensen indicates that this is indeed possible.
}
 